\documentclass[11pt, reqno]{amsart}
\usepackage{txfonts}

\usepackage{amsmath,amssymb,amsthm,mathrsfs,enumerate,bm,xcolor,multirow,pbox}
\usepackage{graphicx,color,framed,tikz,caption,subcaption}
\usepackage{enumitem}
\setlist{leftmargin=9mm}
\usepackage[colorlinks,linkcolor=black,citecolor=black,urlcolor=black]{hyperref}
\allowdisplaybreaks[4]
\numberwithin{equation}{section}
\newcommand{\N}{\mathbb{N}}
\newcommand{\R}{\mathbb{R}}

\newcommand{\pnorm}[2]{\lVert #1\rVert_{#2}}
\newcommand{\bigpnorm}[2]{\big\lVert#1\big\rVert_{#2}}
\newcommand{\biggpnorm}[2]{\bigg\lVert#1\bigg\rVert_{#2}}
\newcommand{\abs}[1]{\lvert#1\rvert}
\newcommand{\bigabs}[1]{\big\lvert#1\big\rvert}
\newcommand{\biggabs}[1]{\bigg\lvert#1\bigg\rvert}
\newcommand{\iprod}[2]{\left\langle#1,#2\right\rangle}

\renewcommand{\epsilon}{\varepsilon}

\renewcommand{\d}[1]{\mathrm{d}#1}

\newcommand{\floor}[1]{\left\lfloor #1 \right\rfloor}
\newcommand{\ceil}[1]{\left\lceil #1 \right\rceil}
\newcommand{\smallop}{\mathfrak{o}_{\mathbf{P}}}
\newcommand{\smallopx}{\mathfrak{o}_{\mathbf{P}|\xi}}
\newcommand{\bigop}{\mathcal{O}_{\mathbf{P}}}
\newcommand{\bigopx}{\mathcal{O}_{\mathbf{P}|\xi }}
\newcommand{\smallo}{\mathfrak{o}}
\newcommand{\bigo}{\mathcal{O}}

\newcommand{\equald}{\stackrel{d}{=}}
\renewcommand{\hat}{\widehat}
\renewcommand{\tilde}{\widetilde}

\DeclareMathOperator{\E}{\mathbb{E}}
\DeclareMathOperator{\Prob}{\mathbb{P}}

\DeclareMathOperator{\dv}{div}
\DeclareMathOperator{\sign}{\mathsf{sgn}}
\DeclareMathOperator{\J}{\mathsf{J}}
\DeclareMathOperator{\tr}{tr}
\DeclareMathOperator{\var}{Var}

\DeclareMathOperator{\op}{op}

\DeclareMathOperator{\dis}{dist}
\DeclareMathOperator{\lrt}{\mathsf{lrt}}
\DeclareMathOperator{\dof}{\mathsf{dof}}
\DeclareMathOperator{\err}{\mathsf{err}}
\DeclareMathOperator{\seq}{\mathsf{seq}}
\DeclareMathOperator{\plasso}{\mathsf{pL}}
\DeclareMathOperator{\pequiv}{\,\stackrel{\mathrm{p}}{\simeq}\,}
\DeclareMathOperator{\pequivx}{\,\stackrel{\mathrm{p}|\xi}{\simeq}\,}
\let\liminf\relax
\DeclareMathOperator*\liminf{\underline{lim}}
\let\limsup\relax
\DeclareMathOperator*\limsup{\overline{lim}}
\DeclareMathOperator*{\argmax}{arg\,max\,}
\DeclareMathOperator*{\argmin}{arg\,min\,}

\newcommand{\beq}{\begin{equation}}
\newcommand{\eeq}{\end{equation}}
\newcommand{\beqa}{\begin{equation} \begin{aligned}}
\newcommand{\eeqa}{\end{aligned} \end{equation}}
\newcommand{\beqas}{\begin{equation*} \begin{aligned}}
\newcommand{\eeqas}{\end{aligned} \end{equation*}}

\newcommand{\bit}{\begin{itemize}}
	\newcommand{\eit}{\end{itemize}}
\newcommand{\bmat}{\begin{bmatrix}}
	\newcommand{\emat}{\end{bmatrix}}

\theoremstyle{definition}\newtheorem{problem}{Problem}[section]
\theoremstyle{definition}\newtheorem{definition}[problem]{Definition}
\theoremstyle{remark}\newtheorem{assumption}{Assumption}

\theoremstyle{remark}\newtheorem{remark}[problem]{Remark}
\theoremstyle{definition}\newtheorem{example}[problem]{Example}
\theoremstyle{plain}\newtheorem{theorem}[problem]{Theorem}
\theoremstyle{plain}
\theoremstyle{plain}\newtheorem{lemma}[problem]{Lemma}
\theoremstyle{plain}\newtheorem{proposition}[problem]{Proposition}
\theoremstyle{plain}\newtheorem{corollary}[problem]{Corollary}
\theoremstyle{plain}

\AtBeginDocument{%
	\def\MR#1{}
}

\begin{document}

\title[Risk asymptotics]{Noisy linear inverse problems under convex constraints: exact risk asymptotics in high dimensions}
\thanks{The research of Q. Han is partially supported by NSF grant DMS-1916221.}

\author[Q. Han]{Qiyang Han}

\address[Q. Han]{
Department of Statistics, Rutgers University, Piscataway, NJ 08854, USA.
}
\email{qh85@stat.rutgers.edu}

\date{\today}


\begin{abstract}
In the standard Gaussian linear measurement model $Y=X\mu_0+\xi \in \R^m$ with a fixed noise level $\sigma>0$, we consider the problem of estimating the unknown signal $\mu_0$ under a convex constraint $\mu_0 \in K$, where $K$ is a closed convex set in $\R^n$. We show that the risk of the natural convex constrained least squares estimator (LSE) $\hat{\mu}(\sigma)$ can be characterized exactly in high dimensional limits, by that of the convex constrained LSE $\hat{\mu}_K^{\seq}$ in the corresponding Gaussian sequence model at a different noise level. Formally, we show that 
\begin{align*}
\pnorm{\hat{\mu}(\sigma)-\mu_0}{}^2 \big/ (nr_n^2) \to 1 \hbox{ in probability},
\end{align*}
where $r_n^2>0$ solves the fixed point equation
\begin{align*}
 \E \bigpnorm{ \hat{\mu}_K^{\seq}\Big(\sqrt{ \big(r_n^2+\sigma^2\big)\big/\big(m/n\big) }\Big) -\mu_0 }{}^2 = nr_n^2.
\end{align*}
This characterization holds (uniformly) for risks $r_n^2$ in the maximal regime that ranges from constant order all the way down to essentially the parametric rate, as long as certain necessary non-degeneracy condition is satisfied for $\hat{\mu}(\sigma)$.  

The precise risk characterization reveals a fundamental difference between noiseless (or low noise limit) and noisy linear inverse problems in terms of the sample complexity for signal recovery. A concrete example is given by the isotonic regression problem: While exact recovery of a general monotone signal requires $m\gg n^{1/3}$ samples in the noiseless setting, consistent signal recovery in the noisy setting requires as few as $m\gg \log n$ samples. Such a discrepancy occurs when the low and high noise risk behavior of $\hat{\mu}_K^{\seq}$ differ significantly. In statistical languages, this occurs when $\hat{\mu}_K^{\seq}$ estimates $0$ at a faster `adaptation rate' than the slower `worst-case rate' for general signals. Several other examples, including non-negative least squares and generalized Lasso (in constrained forms), are also worked out to demonstrate the concrete applicability of the theory in problems of different types.

The proof relies on a collection of new analytic and probabilistic results concerning estimation error, log likelihood ratio test statistics, and degree-of-freedom associated with $\hat{\mu}_K^{\seq}$, regarded as stochastic processes indexed by the noise level. These results are of independent interest in and of themselves.
\end{abstract}

\maketitle

\setcounter{tocdepth}{1}
\tableofcontents


\sloppy

\section{Introduction}

\subsection{Overview}

Consider the standard Gaussian linear measurement model
\begin{align}\label{model:linear_inverse}
Y = X \mu_0+ \xi,
\end{align}
where $X \in \R^{m\times n}$ is a design/measurement matrix with Gaussian ensembles $\mathcal{N}(0,1/n)$, $\mu_0 \in \R^n$ is the signal of interest, and $\xi \in \R^m$ is an error vector whose coordinates are i.i.d. random variables with mean $0$ and variance $\sigma^2$. Here $n,m$ stand for the signal dimension and the sample size respectively. We are interested in estimating/recovering the signal vector $\mu_0\in \R^n$ based on the observation $Y\in \R^m$. In a variety of applications, structure information on $\mu_0$ can be described by a convex constraint $\mu_0 \in K$, where $K$ is a closed convex set in $\R^n$. A canonical estimator in this setting is the convex constrained least squares estimator (LSE)
\begin{align}\label{def:generic_est}
\hat{\mu}(\sigma) \in \argmin_{\mu \in K} \pnorm{Y-X\mu}{}^2,
\end{align}
which is also the maximum likelihood estimator of $\mu_0$ when the error vector $\xi$ is further assumed to be a standard Gaussian vector.  As (\ref{def:generic_est}) is a convex program, a (near) minimizer can in principle be computed efficiently. In addition to problem specific computational techniques, general iterative methods such as approximate message passing (AMP) algorithms may also be used to facilitate efficient computation for (\ref{def:generic_est}), cf. \cite[Section 7.2]{berthier2020state}.

In this paper we will be interested in the precise risk behavior of the constrained LSE $\hat{\mu}(\sigma)$ in (\ref{def:generic_est}). This problem, in its equivalent or generalized form, has received considerable attention in the literature; we only refer the reader to the more recent papers  \cite{chandrasekaran2012convex,oymak2013squared,stojnic2013framework,amelunxen2014living,thrampoulidis2014simple,thrampoulidis2015regularized,tropp2015tropp,oymak2016sharp,thrampoulidis2018precise}; more references can be found therein. From these cited works, the (risk) behavior of $\hat{\mu}(\sigma)$ is now well understood in the \emph{noiseless} setting $\sigma=0$ and in the \emph{low noise limit} $\sigma \downarrow 0$ setting. In the noiseless setting, as $\hat{\mu}(0)=\mu_0$ is clearly a feasible solution, the problem is to determine whether $\hat{\mu}(0)=\mu_0$ is the unique solution for a given sample size $m$. The work \cite{amelunxen2014living} discovers a precise phase transition mechanism that can be described solely by the conic geometry of $K$ near $\mu_0$. Formally, let $T_K(\mu_0)$ be the `tangent cone' of $K$ at $\mu_0$ (precise meaning see Definition \ref{def:tangent_cone}), and $\delta_{T_K(\mu_0)}$ be the `statistical dimension' of the closed cone $T_K(\mu_0)$ (precise meaning see Definition \ref{def:delta_K}). At this point, the reader may be content with the rough idea that more `structures' within $\mu_0 \in K$ lead to a smaller `dimension' $\delta_{T_K(\mu_0)}$. Using this quantity,  \cite{amelunxen2014living} shows that:
\begin{itemize}
	\item If $m\geq (1+\epsilon) \cdot \delta_{T_K(\mu_0)}$, then with high probability the convex program (\ref{def:generic_est}) has a unique solution $\hat{\mu}(0)=\mu_0$; in this sense (\ref{def:generic_est}) achieves exact recovery.
	\item If $m\leq (1-\epsilon)\cdot \delta_{T_K(\mu_0)}$, then with high probability the solution to the convex program (\ref{def:generic_est}) is not unique; in this sense (\ref{def:generic_est}) fails. 
\end{itemize}
In fact \cite{amelunxen2014living} proves a stronger sub-Gaussian tail for the recovery/failure probability with respect to $\epsilon$ at a proper scale, whereas \cite{goldstein2017gaussian} further shows that the shape of this tail is exactly Gaussian in suitable high dimensional limits.  

The quantity $\delta_{T_K(\mu_0)}$ continues to play an important role in determining the risk behavior of $\hat{\mu}(\sigma)$ in the low noise limit $\sigma \downarrow 0$ setting. For instance, \cite[Theorem 3.1]{oymak2013squared} shows that when $m\geq (1+\epsilon)\cdot \delta_{T_K(\mu_0)}$, 
\begin{align}\label{intro:risk_low_noise_limit}
\lim_{\sigma \downarrow 0} \frac{ \pnorm{\hat{\mu}(\sigma)-\mu_0}{}^2}{n\sigma^2} = \frac{ \delta_{T_K(\mu_0)}}{m-\delta_{T_K(\mu_0)}}
\end{align}
holds in a suitable probabilistic sense. Consequently, the behavior of $\hat{\mu}(\sigma)$ in both the noiseless setting and the low noise limit setting can be completely described by the quantity $\delta_{T_K(\mu_0)}$  alone, and the sample size $m$ need to (substantially) exceed $\delta_{T_K(\mu_0)}$ to guarantee exact recovery in the noiseless setting and consistent recovery in the low noise limit setting. 

The major goal of this paper is to gain a precise understanding for the behavior of the risk $\pnorm{\hat{\mu}(\sigma)-\mu_0}{}^2$ in the statistically more common noisy setting, where asymptotics take place in the high dimensional limit $n\to \infty$ of problem instances $(n,m,\mu_0,K)$, keeping the noise level $\sigma^2>0$ fixed. Throughout this manuscript, explicit dependence of $(m,\mu_0,K)$ and related quantities on the signal dimension $n$ will be suppressed for ease of notation. As we will see, in such a high dimensional limiting setting, the precise behavior of $\pnorm{\hat{\mu}(\sigma)-\mu_0}{}^2$ can no longer be described by the simple quantity $\delta_{T_K(\mu_0)}$ alone, and the right hand side of (\ref{intro:risk_low_noise_limit}) can be far from accurate even in order. As a consequence, the sample size $m$ needed for consistent recovery of the signal $\mu_0$ in high dimensions need not apriori exceed the threshold $\delta_{T_K(\mu_0)}$---in fact $m$ can be much smaller in order than $\delta_{T_K(\mu_0)}$ for the convex constrained LSE $\hat{\mu}(\sigma)$ to consistently recover certain highly structured signals $\mu_0$. 

\subsection{Risk asymptotics}

Define the Gaussian sequence model 
\begin{align}\label{model:seq}
y = \mu_0+\sigma_h \cdot h,
\end{align}
where $h \sim \mathcal{N}(0, I_n)$ and $\sigma_h>0$.\footnote{We will always use $y \in \R^n$ (resp. $Y \in \R^m$) for the response vector in the Gaussian sequence model (\ref{model:seq}) (resp. the Gaussian linear measurement model (\ref{model:linear_inverse})).} We will characterize the exact risk of $\hat{\mu}(\sigma)$ by relating it to the convex constrained least squares estimator (LSE) in the Gaussian sequence model (\ref{model:seq}), defined by
\begin{align}\label{def:seq_LSE}
\hat{\mu}_K^{\seq}(\sigma_h)\equiv \argmin_{\mu \in K} \pnorm{y-\mu}{}^2. 
\end{align}
The risk behavior of $\hat{\mu}_K^{\seq}$ and of more general empirical risk minimizers (ERM) is a well-studied topic in statistical theory; see e.g. the monographs \cite{van1996weak,van2000empirical,massart2007concentration,koltchinskii2008oracle,gine2015mathematical} for an in-depth treatment of how the size of expected suprema of localized Gaussian/empirical processes can be inverted to upper bounds for the risk of $\hat{\mu}_K^{\seq}$ and more general ERMs. For $\hat{\mu}_K^{\seq}$ in (\ref{def:seq_LSE}), the work \cite{chatterjee2014new} shows that its risk is completely characterized by the location of maximum of certain quadratically drifted Gaussian process. In essence, a large number of existing tools can be directly employed to compute (bounds for) the risk of $\hat{\mu}_K^{\seq}$. 

The main abstract result of this paper, Theorem \ref{thm:risk_asymp}, shows that the risk of $\hat{\mu}(\sigma)$ can be computed exactly in high dimensional limits, by looking at the risk of $\hat{\mu}_K^{\seq}(\cdot)$ with a different noise level. Let $r_n>0$ be the solution to the fixed point equation
\begin{align}\label{intro:fixed_point_equation}
\frac{1}{n}\cdot \E \biggpnorm{ \hat{\mu}_K^{\seq}\bigg(\sqrt{  \frac{r_n^2+\sigma^2}{m/n}  }\bigg) -\mu_0 }{}^2 = r_n^2,
\end{align}
which exists uniquely if and only if $m>\delta_K$ ($\delta_K$ is a `generalized statistical dimension' defined formally in (\ref{def:delta_K})). Then under certain necessary `non-degeneracy' condition on the residual of the convex program (\ref{def:generic_est}) (= condition (R2) in Theorem \ref{thm:risk_asymp}), the risk asymptotics
\begin{align}\label{intro:risk_asymp}
\frac{ \pnorm{\hat{\mu}(\sigma)-\mu_0}{}^2}{nr_n^2} \to 1 \hbox{ in probability}
\end{align}
hold (uniformly) for $r_n^2$ ranging from the constant order $r_n^2 = \bigo(1)$, all the way down to the parametric rate $\mathcal{O}(1/m)$ (up to a multiplicative logarithmic factor). As the left hand side of (\ref{intro:risk_asymp}) may possibly be a non-degenerate random variable when the risk $r_n^2$ is of a parametric order, the prescribed regime in which the risk asymptotics (\ref{intro:risk_asymp}) hold cannot be further expanded at the current level of generality. 

The fixed point equation equation (\ref{intro:fixed_point_equation}) is in general highly non-linear and therefore does not admit a closed form solution for $r_n$, except for extremely simple instances of $(K,\mu_0)$. Nonetheless,  (\ref{intro:fixed_point_equation}) is indeed compatible with the low noise limit risk in (\ref{intro:risk_low_noise_limit}) in that the solution $r_n^2$ recovers the right hand side of (\ref{intro:risk_low_noise_limit})  as $\sigma \downarrow 0$ whenever $m>\delta_{T_K(\mu_0)}$ (cf. Proposition \ref{prop:exist_unique_fixed_pt}). Such a coincidence is intrinsically due to the \emph{low noise} risk behavior of $\hat{\mu}_K^{\seq}$  that can actually be characterized by $\delta_{T_K(\mu_0)}$ alone: $\lim_{\sigma \downarrow 0} \E \pnorm{\hat{\mu}_K^{\seq}(\sigma)-\mu_0}{}^2/\sigma^2= \delta_{T_K(\mu_0)}$ (cf. Proposition \ref{prop:est_err_variance}).  On the other hand, as one may expect from the fixed point equation (\ref{intro:fixed_point_equation}), in the most interesting high dimensional limiting regime $m\ll n$, the behavior of $r_n$ should also critically depend on the \emph{high noise} risk behavior of the $\hat{\mu}_K^{\seq}$: $\lim_{\sigma \uparrow \infty} \E \pnorm{\hat{\mu}_K^{\seq}(\sigma)-\mu_0}{}^2/\sigma^2= \delta_{K}\leq \delta_{T_K(\mu_0)}$ (cf. Proposition \ref{prop:est_err_variance}). In fact, a closer investigation reveals that the convex constrained LSE $\hat{\mu}(\sigma)$ in (\ref{def:generic_est}) and its risk exhibit different behavior, in accordance to the magnitude of $m$ with respect to the three regimes determined by the low and high noise risk behavior of $\hat{\mu}_K^{\seq}$:

\begin{figure}
	\centering
	\includegraphics[width=13cm]{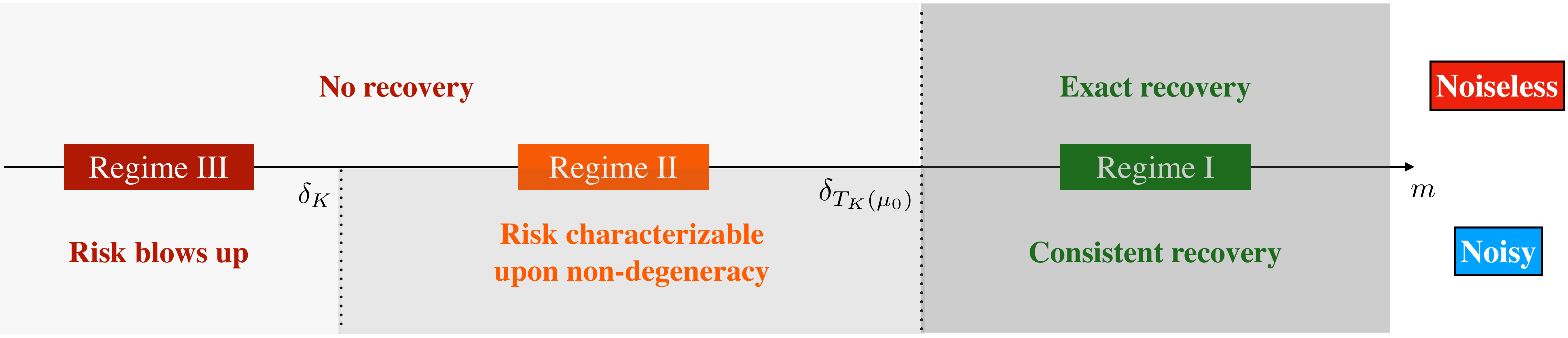}
	\caption{Three regimes of $m$ for noiseless and noisy Gaussian linear inverse problems. $\delta_K=\lim_{\sigma \uparrow \infty} \E \err(\sigma)/\sigma^2$ (resp. $\delta_{T_K(\mu_0)}=\lim_{\sigma \downarrow 0} \E \err(\sigma)/\sigma^2$) corresponds to high noise (resp. low noise) limit of the normalized risk of $\hat{\mu}_K^{\seq}$. }
	\label{fig:m_regime}
\end{figure}

\begin{itemize}
	\item (\emph{Regime I}) If $m$ exceeds the low noise risk limit of $\hat{\mu}_K^{\seq}$ in that $\liminf_n (m/\delta_{T_K(\mu_0)})>1$, degeneracy almost never occurs for the residual of the convex program (\ref{def:generic_est}), and the risk of all its near minimizers can be characterized via (\ref{intro:fixed_point_equation})-(\ref{intro:risk_asymp}) in high dimensional limits. 
	\item (\emph{Regime II}) If $m$ falls in between the low and high noise risk limit of $\hat{\mu}_K^{\seq}$ in that $\limsup_n (m/\delta_{T_K(\mu_0)})<1$ while $\liminf_n (m/\delta_K)>1$, the risk of $\hat{\mu}(\sigma)$ is characterizable only if the residual of the convex program (\ref{def:generic_est}) is non-degenerate. Degeneracy may occur in this regime that results in multiple distinct minimizers of (\ref{def:generic_est}) that are too far way from each other for a well-defined limiting risk characterization as (\ref{intro:risk_asymp}) to exist.
	\item (\emph{Regime III}) If $m$ falls below the high noise risk limit of $\hat{\mu}_K^{\seq}$ in that $\limsup_n (m/\delta_K)<1$, with high probability the convex program (\ref{def:generic_est}) admits a minimizer whose risk is arbitrarily large for each and every possible underlying signal $\mu_0 \in K$, at least when $K$ is a closed convex cone. 
\end{itemize}

As exact recovery in the noiseless Gaussian linear measurement model (\ref{model:linear_inverse}) is possible only in Regime I, while risk characterization of $\hat{\mu}(\sigma)$ in the noisy setting is possible in both Regimes I and II, consistent recovery of the signal $\mu_0$ via $\hat{\mu}(\sigma)$ in the noisy linear inverse problems may require (much) fewer samples $m$ than those required for exact recovery in the noiseless setting. Such a phenomenon occurs in Regime II when the low and high noise risk behavior of $\hat{\mu}_K^{\seq}$ differ significantly in the sense that $\delta_K\ll \delta_{T_K(\mu_0)}$.

\subsection{Examples}

Several examples of the risk asymptotics (\ref{intro:fixed_point_equation})-(\ref{intro:risk_asymp})  are worked out in Section \ref{section:examples}, including (1) non-negative least squares, (2) shape constrained regression problems, and (3) generalized Lasso problems (in constrained forms). These examples not only serve as an illustration of the wide applicability of the theory (\ref{intro:fixed_point_equation})-(\ref{intro:risk_asymp}) in concrete problems, some of the examples above also give a clear demonstration of the possibility of consistent recovery in Regime II for noisy linear inverse problems. In fact, for example (2), although it is not feasible to give an exact computation of the risk via the fixed point equation (\ref{intro:fixed_point_equation}) for general shape constrained regression problems, an asymptotically `sharp oracle inequality' is established for $\hat{\mu}(\sigma)$, showing that consistent recovery of $\mu_0$ is possible for `good enough' shape constrained signals, as soon as $m$ exceeds $\delta_K$ which is typically far smaller in order compared to $\delta_{T_K(\mu_0)}$. 

For instance, in the canonical example of isotonic regression $K=K_{\uparrow}$ (formally defined in Section \ref{section:shape}), consistent recovery for general smooth monotone signals $\mu_0$ of bounded variation requires as few as $m\gg \log n$ samples in the noisy Gaussian linear measurement model (\ref{model:linear_inverse}), while exact recovery of such $\mu_0$'s in the noiseless setting requires at least $m\gg n^{1/3}$ many samples. Such a discrepancy is intimately due to the inhomogeneity of the high and low noise risk behavior of $\hat{\mu}_{K_{\uparrow}}^{\seq}$ in that $\log n\simeq \delta_{K_{\uparrow}} \ll \delta_{T_{K_{\uparrow}}(\mu_0)}\asymp n^{1/3}$ for the prescribed $\mu_0$'s. Equivalently, this gap occurs due to the fact that $\hat{\mu}_{K_{\uparrow}}^{\seq}$ estimates $0$ at a much faster `adaptation rate' $\log n$, compared to the slower `worst-case rate' $n^{1/3}$ for general monotone signals $\mu_0$. It is now well understood that such rate adaptation at $0$ occurs in a variety of shape constrained problems corresponding to different choices of $K$, cf. \cite{meyer2000degrees,zhang2002risk,chatterjee2015risk,han2016multivariate,chatterjee2018matrix,bellec2018sharp,han2019isotonic,kur2020convex,fang2021multivariate}; see also the review article \cite{guntuboyina2017nonparametric}. As such, in all these problems where `adaptation' occurs, there (in principle) persists a large gap between the sample complexity for consistent recovery in the noisy linear inverse problems and that for exact recovery in the noiseless setting.  

\subsection{Proof techniques}

The basic approach for the proof of the risk asymptotics in (\ref{intro:fixed_point_equation})-(\ref{intro:risk_asymp}) is to reduce the optimization problem (\ref{def:generic_est}) to another simpler, but probabilistically almost `equivalent' optimization problem via the Gaussian min-max theorem, initially proved by Gordon \cite{gordan1985some,gordon1988milman}. This basic reduction approach \cite{thrampoulidis2018precise}, together with the approach of explicitly constructing an AMP algorithm \cite{bayati2011dynamics} that approximates the estimator under study, has gained prominence in recent years in the risk analysis for a number of high dimensional problems in the so-called `proportional high dimensional regime' $m/n\to \tau \in (0,\infty)$. There the goal is to pin down the precise value of the risk when it is of constant order; see e.g. \cite{bayati2012lasso,thrampoulidis2015regularized,donoho2016high,elkaroui2018impact,thrampoulidis2018precise,sur2019modern,miolane2021distribution,bellec2021debias} and many references therein for this line of research.

In our problem, both the regime with constant order risk and the more `classical' regime with vanishing risk are of significant interest. In fact, the `effective dimension' of the problem is implicitly determined by $(K,\mu_0)$, which in many cases necessarily fails to be proportional to the sample size. As such, the major challenge in proving the risk characterization (\ref{intro:fixed_point_equation})-(\ref{intro:risk_asymp}) lies in establishing its validity in the maximal regime all the way down to the parametric rate. This is achieved by a carefully designed proof architecture of conditional localization/de-stochastization and gap analysis for the reduced, `equivalent' optimization problem; see Section \ref{section:proof_sketch} for a sketch and Section \ref{section:proof_main_result} for details. 

The prescribed method of analysis relies crucially on a collection of newly developed analytic and probabilistic results for three interrelated stochastic processes: the estimation error, log likelihood ratio test statistics, and degree-of-freedom associated with $\hat{\mu}_K^{\seq}$, viewed as processes indexed by the noise level in the Gaussian sequence model (\ref{model:seq}). Of particular importance are several qualitative monotonicity properties for these processes and their normalized versions, quantitative uniform concentration inequalities whose variance components can be directly related to the fixed point equation (\ref{intro:fixed_point_equation}), and variational characterizations that facilitate tight upper and lower bounds relating the three processes. These results, to be detailed in Section \ref{section:est_err_lrt_process}, are proved using a suite of Gaussian and convex analysis techniques, and are of significant independent interest in and of themselves.

\subsection{Organization}

The rest of the paper is organized as follows. Section \ref{section:risk_asymp} presents the abstract theory of risk asymptotics via (\ref{intro:fixed_point_equation})-(\ref{intro:risk_asymp}). Section \ref{section:examples} gives a detailed treatment of the abstract theory in the examples mentioned above. Section \ref{section:est_err_lrt_process} develops a collection of analytic and probabilistic results for the estimation error, log likelihood ratio test statistics, and degree-of-freedom associated with $\hat{\mu}_K^{\seq}$. An outline of the proof for the main theory is provided in Section \ref{section:proof_sketch}, with most proof details presented in Sections \ref{section:proof_main_result}-\ref{section:proof_est_lrt_dof} and the appendices.

\subsection{Notation}\label{section:notation}

For any positive integer $n$, let $[1:n]$ denote the set $\{1,\ldots,n\}$. For $a,b \in \R$, $a\vee b\equiv \max\{a,b\}$ and $a\wedge b\equiv\min\{a,b\}$. For $a \in \R$, let $a_\pm \equiv (\pm a)\vee 0$. For $a>0$, let $\log_+(a)\equiv 1\vee \log(a)$. For $x \in \R^n$, let $\pnorm{x}{p}$ denote its $p$-norm $(0\leq p\leq \infty)$. We simply write $\pnorm{x}{}\equiv\pnorm{x}{2}$. Let $\bm{1}_n=(1,\ldots,1)^\top \in \R^n$. For a matrix $M \in \R^{m\times n}$, let $\pnorm{M}{\op}$ denote the spectral norm of $M$. For $f=(f_1,\ldots,f_n):\R^n \to \R^n$, let $\J_f(z) \equiv (\partial f_i(z)/\partial z_j)_{i,j=1}^n$ denote the Jacobian of $f$ and $
\dv f(z) \equiv \sum_{i=1}^n \frac{\partial}{\partial z_i} f_i (z)=\tr(\J_f(z))$ for $z \in \R^n$ whenever definable.

For a closed convex set $K$ and $\mu_0 \in K$, the tangent cone of $K$ at $\mu_0$, denoted as $T_K(\mu_0)$, is defined as 
\begin{align}\label{def:tangent_cone}
T_K(\mu_0)\equiv \mathrm{cl}\big\{\alpha(\nu-\mu_0): \alpha\geq 0, \nu \in K\big\}.
\end{align}
The indicator function for a closed convex set $K$ is written as $\bm{0}_K(x) = 0\cdot \bm{1}_{x \in K}+ \infty\cdot \bm{1}_{x \notin K}$. Its convex conjugate, also known as support function of $K$, is written as $\bm{0}_K^\ast(s)=\sup_{t \in K}s^\top t$. For a closed convex cone $K\subset \R^n$, let $K^\ast$ be its polar cone defined via $K^\ast \equiv \{\nu \in \R^n: \iprod{\nu}{\mu}\leq 0, \forall \mu \in K\}$.

We use $C_{x}$ to denote a generic constant that depends only on $x$, whose numeric value may change from line to line unless otherwise specified. $a\lesssim_{x} b$ and $a\gtrsim_x b$ mean $a\leq C_x b$ and $a\geq C_x b$ respectively, and $a\asymp_x b$ means $a\lesssim_{x} b$ and $a\gtrsim_x b$ ($a\lesssim b$ means $a\leq Cb$ for some absolute constant $C$). For two nonnegative sequences $\{a_n\}$ and $\{b_n\}$, we write $a_n\ll b_n$ (respectively~$a_n\gg b_n$) if $\lim_{n\rightarrow\infty} (a_n/b_n) = 0$ (respectively~$\lim_{n\rightarrow\infty} (a_n/b_n) = \infty$). We write $a_n\simeq b_n$ (resp. $a_n\pequiv b_n$) if $\lim_n (a_n/b_n)=1$ (resp. in probability). We follow the convention that $0/0 = 0$. $\bigo$ and $\smallo$ (resp. $\mathcal{O}_{\mathbf{P}}$ and $\mathfrak{o}_{\mathbf{P}}$) denote the usual big and small O notation (resp. in probability). For a generic random variable $\xi$, we write $\Prob^\xi, \E^\xi$ the conditional probability and expectation on $\xi$. Similar meaning applies to $\bigopx,\smallopx,\pequivx$. We reserve the notation $\xi=(\xi_1,\ldots,\xi_m)$ for an $m$-dimensional error vector, and $h=(h_1,\ldots,h_n)\sim \mathcal{N}(0,I_n)$ be an $n$-dimensional standard normal random vector.

\section{Theory}\label{section:risk_asymp}

This section presents the main abstract theory of this paper. Except for the main Theorem \ref{thm:risk_asymp}, proofs for most other results can be found in Section \ref{section:proof_remain_main}.

\subsection{Assumptions and further notation}

We shall formally record below the assumptions on $X,\xi$ in the Gaussian linear measurement model (\ref{model:linear_inverse}).

\begin{assumption}\label{assump:X_xi}
	Suppose $X$ and $\xi$ satisfy the following:
	\begin{enumerate}
		\item $X \in \R^{m\times n}$ contains i.i.d. $\mathcal{N}(0,1/n)$ entries.
		\item  $\xi \in \R^m$ is an error vector independent of $X$, containing i.i.d. coordinates with mean $0$ and finite, non-degenerate variance $\sigma^2 \in (0,\infty)$. 
	\end{enumerate}
\end{assumption}

These conditions are commonly used in the literature; cf. \cite{oymak2013squared,thrampoulidis2014simple,thrampoulidis2015recovering,thrampoulidis2015regularized,oymak2016sharp,thrampoulidis2018precise}. The choice of the variance level $1/n$ in  $X$ is to ensure that $\pnorm{X\mu}{}^2/m$ will be of the same order as $\pnorm{\mu}{}^2/n$. 

Next we formally record the assumption on $K$.

\begin{assumption}\label{assump:K_cone}
	$K\subset \R^n$ is a closed convex set. 
\end{assumption}

Now we define a notion of `generalized statistical dimension' $\delta_K$ associated with a closed convex set $K$: Let
\begin{align}\label{def:delta_K}
\delta_K\equiv \lim_{\sigma \uparrow \infty} \frac{\E \pnorm{\Pi_K(\sigma h)}{}^2}{\sigma^2} \leq n. 
\end{align}
Here $\Pi_K: \R^n\to \R^n$ is the natural projection map onto $K$. In the definition above, the limit is well-defined due to the monotonicity of the map $\sigma\mapsto \E \pnorm{\Pi_K(\sigma h)}{}^2/\sigma^2$ (cf. Lemma \ref{lem:monotone_est_err}). When $K$ is a closed convex cone, by homogeneity of the projection map $\Pi_K(\sigma h)=\sigma \Pi_K(h)$, the above definition recovers the usual statistical dimension for the closed convex cone $\delta_K = \E \pnorm{\Pi_K(h)}{}^2$. We refer the reader to \cite[Section 3]{amelunxen2014living} for comprehensive background review of the notion of statistical dimension associated with a closed convex cone.

We introduce some further notation that will be used throughout the paper. Let $\err_{(K,\mu_0)}(\sigma_h)$ be the squared error of $\hat{\mu}_K^{\seq}$ at noise level $\sigma_h$:
\begin{align}\label{def:F}
\err(\sigma_h)\equiv \err_{(K,\mu_0)}(\sigma_h)\equiv \pnorm{\hat{\mu}_K^{\seq}(\sigma_h)-\mu_0}{}^2,
\end{align}
and let $\lrt_{(K,\mu_0)}(\sigma_h)$ be  the (scaled) log likelihood ratio test statistics of testing the mean vector being $\mu_0 \in K$ under the Gaussian sequence model (\ref{model:seq}) (cf. \cite{han2022high}):
\begin{align}\label{def:lrt}
\lrt(\sigma_h)\equiv \lrt_{(K,\mu_0)}(\sigma_h) &\equiv \pnorm{y-\mu_0}{}^2-\pnorm{y-\hat{\mu}_K^{\seq}(\sigma_h)}{}^2.
\end{align}
The subscript $(K,\mu_0)$ in $\err$ and $\lrt$ is usually suppressed for notational simplicity. 

We need one further definition. For $r\geq 0,\delta\geq 0$, let
\begin{align}
\omega_\delta(r)\equiv \omega_\delta(r,\sigma)= \sqrt{ \frac{ r^2+\sigma^2}{\delta} }.
\end{align}
For $\sigma>0$, it is understood that $\omega_0(r)=\omega_0(r,\sigma)=\infty$ for any $r\geq 0$.

\subsection{The fixed point equation}

\begin{proposition}\label{prop:exist_unique_fixed_pt}
The following hold.
\begin{enumerate}
	\item The fixed point equation 
	\begin{align}\label{eqn:fixed_pt_eqn}
	n^{-1} \E \err\big(\omega_{m/n}(r)\big) = r^2
	\end{align}
	has at most one solution in $r \in (0,\infty)$ that exists if and only if $m>\delta_K$.
	\item Suppose $m>\delta_K$ and let $r_n$ be the unique solution to (\ref{eqn:fixed_pt_eqn}). Let the iterations $\{r_{n,t}:t=0,1,2,\ldots\}\subset \R_{\geq 0}$ be defined through 
	\begin{align}\label{def:iteration_fixed_point}
	 r_{n,t+1}^2&\equiv n^{-1} \E \err\big(\omega_{m/n}(r_{n,t})\big),\quad t =0,1,2,\ldots
	\end{align}
	with initialization $r_{n,0}^2\equiv 0$. Then for any $\rho \in \big(r_n^2/(r_n^2+\sigma^2),1\big)$, there exists $T_\rho \in \N$ such that $\abs{r_{n,t}-r_n}/r_n\leq \rho^{(t-T_\rho)_+}$. In particular, $\lim_{t \to \infty} r_{n,t}=r_n$ and the convergence is linear eventually.  
	\item Suppose $m>\delta_K$ and let $r_n=r_n(\sigma)$ be the unique solution to (\ref{eqn:fixed_pt_eqn}). Then 
	\begin{align}\label{ineq:rn_bounds_generic}
	\delta_K/(m-\delta_K)\leq r_n^2(\sigma)/\sigma^2\leq  \delta_{T_K(\mu_0)}/(m- \delta_{T_K(\mu_0)})_+.
	\end{align}
	The upper bound is tight in the  low noise limit $\sigma\downarrow 0$ when $m>\delta_{T_K(\mu_0)}$. 
\end{enumerate}	
\end{proposition}
Proposition \ref{prop:exist_unique_fixed_pt}-(1) provides a complete picture for the solution of the key fixed point equation (\ref{eqn:fixed_pt_eqn}) to exist uniquely; in fact the solution will either be non-existent, or exist uniquely. To get a feel of this result, let us consider the toy case where $K$ is a subspace of $\R^n$ of dimension $\dim(K)$. Then it is easy to calculate that $\E\err(\sigma)=\sigma^2\cdot \dim(K)$, and therefore (\ref{eqn:fixed_pt_eqn}) reduces to
\begin{align*}
 nr^2 = \omega_{m/n}^2(r)\cdot \dim(K)\, \Leftrightarrow\, (m-\dim(K))\cdot r^2 = \dim(K)\cdot \sigma^2. 
\end{align*}
Clearly the above equation admits a unique solution for $r$ if and only if $m>\dim(K)$. Proposition \ref{prop:exist_unique_fixed_pt}-(1) proves that the above simple calculation for a subspace $K$ can be taken as far as $K$ being a general closed convex set. The only formal difference is to replace `$\dim(K)$' by the `generalized statistical dimension' $\delta_K$ defined in (\ref{def:delta_K}). 

For a general closed convex set $K$, the non-linear equation (\ref{eqn:fixed_pt_eqn}) does not admit a simple closed form solution. However, as long as the map $\sigma\mapsto \E \err(\sigma)$ can be evaluated efficiently, Proposition \ref{prop:exist_unique_fixed_pt}-(2) shows that a simple yet linearly converging iterative scheme  (\ref{def:iteration_fixed_point}) can be used to find an approximate solution $r_{n,t}$, whenever the solution $r_n$ to (\ref{eqn:fixed_pt_eqn})  exists uniquely.

Finally (\ref{ineq:rn_bounds_generic}) in Proposition \ref{prop:exist_unique_fixed_pt}-(3) provides simple upper and lower bounds for the rate $r_n^2$. The low noise limiting behavior shows that the upper bound in (\ref{ineq:rn_bounds_generic}) cannot be further improved at this level of generality, and that the fixed point equation (\ref{eqn:fixed_pt_eqn}) is compatible with the precise risk formula obtained in \cite[Theorem 3.1]{oymak2013squared} (see also (\ref{intro:risk_low_noise_limit})) in the low noise limit for $m>\delta_{T_K(\mu_0)}$. As we will see below, the behavior of $r_n(\sigma)$ in the high dimensional limit $n\to \infty$ with a fixed $\sigma>0$ is significantly different from that in the simple low noise limit $\sigma \downarrow 0$ with a fixed problem instance $(n,m,\mu_0,K)$.

\subsection{Abstract results}

To describe our main result, let
\begin{align}\label{def:Ln}
\mathfrak{L}_n \equiv \log (1+\delta_{T_K(\mu_0)})+ \log \log (16n) \lesssim \log (en).
\end{align} 
At this point, $\mathfrak{L}_n$ may be simply regarded as $\log n$. The slightly more complicated form for $\mathfrak{L}_n$ we adopt above will be useful in terms of a further reduction from $\log n$ to $\log\log n$, in several examples to be studied in Section \ref{section:examples}.

We are now in position to state the main abstract result of this paper.

\begin{theorem}\label{thm:risk_asymp}
	Suppose Assumptions \ref{assump:X_xi}-\ref{assump:K_cone} hold and $m>\delta_K$. Let $r_n$ be the solution to the fixed point equation (\ref{eqn:fixed_pt_eqn}) (which exists uniquely according to Proposition \ref{prop:exist_unique_fixed_pt}).
	Further assume the following:
	\begin{enumerate}
		\item[(R1)] $r_n^2\lesssim 1$ and $m\gg \mathfrak{L}_n$.
		\item[(R2)] 
		With $\omega_n\equiv \omega_{m/n}(r_n)$,
		\begin{align}\label{cond:R2}
		\limsup_n  \frac{1}{2n\sigma^2}\big(\E \lrt(\omega_n)-\E \err(\omega_n)\big)<1.
		\end{align}
	\end{enumerate}
    Then for any (sequence of) near minimizer(s) $\hat{\mu}(\sigma)\in K$ such that 
    \begin{align}\label{def:approximate_minimizer}
    m^{-1} \pnorm{Y-X\hat{\mu}(\sigma)}{}^2\leq \min_{\mu \in K}\, m^{-1}\pnorm{Y-X\mu}{}^2+\smallop(r_n^2),
    \end{align}
    it holds as $n \to \infty$ that
	\begin{align}\label{eqn:risk_LSE}
	n^{-1} \pnorm{\hat{\mu}(\sigma)-\mu_0}{}^2\pequiv r_n^2+\bigop\big(\mathfrak{L}_n/m\big). 
	\end{align}
\end{theorem}

A detailed proof of the above theorem can be found in Section \ref{section:proof_main_result}. As the proof is quite technical, a sketch is outlined in Section \ref{section:proof_sketch}. From the proof, (\ref{eqn:risk_LSE}) can actually be strengthened to a uniform statement with respect to the constant in (R1)-(R2) in the following sense: For any $M_n\to \infty$ and a fixed constant $L>1$, let $\mathscr{C}(M_n,L)$ be all problem instances $(n,m,\mu_0,K)$ such that $L^{-1}M_n\mathfrak{L}_n/m\leq r_n^2\leq L, m\geq L^{-1}M_n\mathfrak{L}_n$ and $(\E \lrt(\omega_n)-\E \err(\omega_n))/(2n\sigma^2)\leq 1-L^{-1}$. Then for any fixed $\epsilon>0$, any sequence of minimizers $\hat{\mu}(\sigma)\in \argmin_{\mu \in K} \pnorm{Y-X\mu}{}^2$ satisfies
\begin{align}\label{eqn:risk_uniform}
\lim_n\sup_{(n,m,\mu_0,K) \in \mathscr{C}(M_n,L)} \Prob\bigg(\biggabs{ \frac{\pnorm{\hat{\mu}(\sigma)-\mu_0}{}^2}{nr_n^2}-1}>\epsilon  \bigg)=0.
\end{align}
A non-asymptotic explicit error bound can in principle be obtained by tracking the proof; we have refrained from doing so here, as obtaining an optimal error bound with respect to $\epsilon$ in (\ref{eqn:risk_uniform}) seems to require genuinely new ideas.

Theorem \ref{thm:risk_asymp} shows that the exact risk behavior for the constrained LSE $\hat{\mu}(\sigma)$ in the model (\ref{model:linear_inverse}) is completely characterized by the high noise (resp. low noise) risk behavior of the corresponding LSE $\hat{\mu}_K^{\seq}$ in the Gaussian sequence model (\ref{model:seq}) in the under-sampling regime $m\ll n$ (resp. over-sampling regime $m\gg n$). As we will be mostly interested in the regime $m\ll n$, the key to understand the risk of $\hat{\mu}(\sigma)$ will typically be the \emph{high noise} risk behavior of $\hat{\mu}_K^{\seq}$.

It is important to note that from (\ref{eqn:risk_LSE}) and the condition (R1), the characterization (\ref{eqn:risk_LSE}) is asymptotically exact in the regime $\mathfrak{L}_n/m\ll r_n^2\lesssim 1$. The requirement $r_n^2\lesssim 1$ is barely a condition as typically we are interested in the case when the risk is not too big. On the other hand, the condition $r_n^2\gg \mathfrak{L}_n/m$ requires the problem to be intrinsically high dimensional, as $1/m$ is the squared parametric rate in this setting. In fact, modulo the multiplicative (logarithmic) factor $\mathfrak{L}_n$, the requirement $r_n^2\gg \mathfrak{L}_n/m$ cannot be further relaxed beyond the parametric rate for the characterization (\ref{eqn:risk_LSE}) to hold in probability.\footnote{This can be seen by considering the linear regression setting with $m\gg n$ (so $K=\R^n$); then the risk of LSE $\hat{\mu}(\sigma)$ is approximately $n^{-1}\pnorm{\hat{\mu}(\sigma)-\mu_0}{}^2\stackrel{d}{\approx} \sigma^2 m^{-1}\sum_{i=1}^n Z_i^2$ where $Z_i$'s are i.i.d. $\mathcal{N}(0,1)$.} 

The second condition (R2) in Theorem \ref{thm:risk_asymp} looks mysterious at this point, so deserves some further understanding. 

\begin{proposition}\label{prop:R2}
Suppose $\mathfrak{L}_n/m\ll r_n^2\lesssim 1$.
\begin{enumerate}
	\item For any (sequence of) $\hat{\mu}(\sigma)$ satisfying (\ref{def:approximate_minimizer}), the residual 	$m^{-1} \pnorm{Y-X\hat{\mu}(\sigma)}{}^2$ is bounded away from, or convergent to, $0$ in probability, according to whether the left hand side of (\ref{cond:R2}) is $<1$ or $>1$. In the latter case, we have $m^{-1} \pnorm{Y-X\hat{\mu}(\sigma)}{}^2=\smallop(r_n^2)$.
	\item The residual $m^{-1} \pnorm{Y-X\hat{\mu}(\sigma)}{}^2\pequiv \sigma^2$ is a consistent estimator of the variance $\sigma^2$ if and only if the limit in (\ref{cond:R2}) is 0 and $r_n\to 0$.
	\item  There exists a sequence of $K$, $\mu_0 \in K$ with the following properties: (i) $\limsup_n (m/\delta_{T_K(\mu_0)})<1$, (ii) the left hand side of (\ref{cond:R2}) is $>1$ and (iii) $r_n^2\gg \mathfrak{L}_n/m$, such that with asymptotically probability $1$, there exists a sequence of $\{\hat{\mu}(\sigma)\}$ satisfying (\ref{def:approximate_minimizer}) whose risk $\{\pnorm{\hat{\mu}(\sigma)-\mu_0}{}^2/(n\tilde{r}_n^2)\}$ does not converge to a deterministic limit for any normalizing sequence $\{\tilde{r}_n^2\}$. 
\end{enumerate}	
\end{proposition}

Proposition \ref{prop:R2}-(1)(2) show that the condition (R2) is intrinsically tied to the non-degenerate limiting behavior of the normalized residual $m^{-1} \pnorm{Y-X\hat{\mu}(\sigma)}{}^2$. An explicit counter-example is constructed in the proof of Proposition \ref{prop:R2}-(3), showing that when (R2) fails, for some problem instances $(K,\mu_0)$, there may exist multiple distinct near minimizers $\hat{\mu}(\sigma)$ that are far away from each other in the regime $\limsup_n (m/\delta_{T_K(\mu_0)})<1$, so the risk cannot be stabilized. Interestingly, the following proposition shows that this is almost the only possible regime in which the condition (R2) may not hold.

\begin{proposition}\label{prop:R2_verify}
The following hold.
\begin{enumerate}
	\item (\ref{cond:R2}) is fulfilled if $\liminf_n (m/\delta_{T_K(\mu_0)})>1$. If $m\gg \delta_{T_K(\mu_0)}$,  (\ref{cond:R2}) holds with limit $0$ and therefore the residual satisfies $m^{-1} \pnorm{Y-X\hat{\mu}(\sigma)}{}^2\pequiv \sigma^2$. 
	\item Suppose that $K$ is a closed convex cone. Then (\ref{cond:R2}) is fulfilled provided that
	\begin{align}\label{cond:R2_suff}
	\sigma^{-2}\cdot \limsup_n  r_n\cdot \Big(\inf\limits_{\nu \in L(K)}\pnorm{\mu_0-\nu}{}/\sqrt{n}\Big)<1,
	\end{align} 
	where $L(K)$ denotes the maximal linear subspace contained in $K$. In particular, when $r_n\to 0$ and $\inf_{\nu \in L(K)}\pnorm{\mu_0-\nu}{}/\sqrt{n}=\mathcal{O}(1)$, (\ref{cond:R2}) holds with limit $0$ and therefore the residual satisfies $m^{-1} \pnorm{Y-X\hat{\mu}(\sigma)}{}^2\pequiv \sigma^2$.
\end{enumerate}	
\end{proposition}
 
Proposition \ref{prop:R2_verify}-(1) formalizes the aforementioned claim above the statement of the proposition. Of course,  although counter-examples exist for which (R2) fails when $m$ falls outside the regime $\liminf_n (m/\delta_{T_K(\mu_0)})>1$, whether (R2) actually fails for a given problem is not within the scope of Proposition \ref{prop:R2_verify}-(1). Proposition \ref{prop:R2_verify}-(2) gives a general recipe along this line under the further condition that $K$ is a closed convex cone:  (R2) is fulfilled in `regular situations' where neither the rate $r_n$ nor the signal strength $\inf_{\nu \in L(K)}\pnorm{\mu_0-\nu}{}/\sqrt{n}$ is not too big. This result will be convenient in some of the examples to be studied in Section \ref{section:examples} ahead.

\subsection{Three regimes of $m$ and connections to noiseless linear inverse problems}

Theorem \ref{thm:risk_asymp} and Propositions \ref{prop:R2}-\ref{prop:R2_verify} taken together suggest three regimes of $m$, according to its size compared to $\delta_K=\lim_{\sigma \uparrow \infty} \E \err(\sigma)/\sigma^2 $ and $\delta_{T_K(\mu_0)}=\lim_{\sigma \downarrow 0} \E \err(\sigma)/\sigma^2$. 

\vspace{0.5em}
\noindent \textbf{Regime I: $\liminf_n (m/\delta_{T_K(\mu_0)})>1$}. 
\vspace{0.5em}

In Regime I, $m$ exceeds low noise limit of the normalized risk of $\hat{\mu}_K^{\seq}$. This is an easy regime for both noiseless and noisy Gaussian linear inverse problems:

\begin{itemize}
	\item For noiseless Gaussian linear inverse problems ($\sigma=0$), \cite{amelunxen2014living,goldstein2017gaussian} showed that in this Regime I, exact recovery of $\mu_0 \in K$ via $\hat{\mu}(0)$ (i.e., $y=X\mu$ has a unique solution $\mu_0\in K$) is possible with high probability. 
	\item 
	For noisy Gaussian linear inverse problems ($\sigma>0$), Theorem \ref{thm:risk_asymp} admits the following simplified risk characterization.
\end{itemize}

\begin{proposition}
Suppose Assumptions \ref{assump:X_xi}-\ref{assump:K_cone} hold, and $\liminf_n (m/\delta_{T_K(\mu_0)})>1$. Let $r_n$ be the solution to the fixed point equation (\ref{eqn:fixed_pt_eqn}) which exists uniquely eventually. If $m\gg \mathfrak{L}_n$, then for any (sequence of) near minimizer(s) $\hat{\mu}(\sigma)\in K$ satisfying (\ref{def:approximate_minimizer}), it holds as $n\to \infty$ that $n^{-1} \pnorm{\hat{\mu}(\sigma)-\mu_0}{}^2\pequiv r_n^2+\bigop\big(\mathfrak{L}_n/m\big)$. 
\end{proposition}
\begin{proof}
By Proposition \ref{prop:R2_verify}-(1), condition (R2) in Theorem \ref{thm:risk_asymp} is verified. As $\delta_{T_K(\mu_0)}\geq \delta_K$, $m>\delta_K$ eventually and $r_n^2=\bigo(1)$ by Proposition \ref{prop:exist_unique_fixed_pt}-(3). Now apply Theorem \ref{thm:risk_asymp} to conclude. 
\end{proof}

In particular, using  Proposition \ref{prop:exist_unique_fixed_pt}-(3), consistent recovery is guaranteed: $n^{-1} \pnorm{\hat{\mu}(\sigma)-\mu_0}{}^2=\smallop(1)$ if $m\gg \delta_{T_K(\mu_0)}$.

\vspace{0.5em}
\noindent \textbf{Regime II: $\liminf_n (m/\delta_K)>1, \limsup_n (m/\delta_{T_K(\mu_0)})<1$}.
\vspace{0.5em}

In Regime II, $m$ falls in between the high and low noise limits of the normalized risk of $\hat{\mu}_K^{\seq}$. Different behavior of $\hat{\mu}(\sigma)$ appears for noiseless and noisy settings:

\begin{itemize}
	\item For noiseless Gaussian linear inverse problems ($\sigma=0$), \cite{amelunxen2014living} showed that in the regime $\liminf_n (m/\delta_{T_K(\mu_0)})<1$ (including Regime III below), exact recovery of $\mu_0 \in K$ via $\hat{\mu}(0)$ fails with high probability. 
	\item For noisy Gaussian linear inverse problems ($\sigma>0$), as have been shown in Theorem \ref{thm:risk_asymp} and Proposition \ref{prop:R2}, this Regime II is quite subtle, in that the validity of the risk characterization in Theorem \ref{thm:risk_asymp} depends on the condition (R2) which need be investigated in a case-by-case manner. 
\end{itemize}

As such, in Regime II, while exact recovery of $\mu_0$ in the noiseless linear inverse problem setting fails, consistent recovery of $\mu_0$ may still be possible in the noisy setting. A significant example is given by the isotonic regression problem in Section \ref{section:shape} ahead: For general smooth monotone signals, while $m\gg n^{1/3}$ as many samples are needed for exact recovery in the noiseless Gaussian linear measurement model, only $m\gg \log n$ many samples are required for consistent recovery in the noisy setting.

\vspace{0.5em}
\noindent \textbf{Regime III: $\limsup_n (m/\delta_K)<1$}. 
\vspace{0.5em}

In Regime III, $m$ falls below high noise limit of the normalized risk of $\hat{\mu}_K^{\seq}$. This is a hard regime for both noiseless and noisy Gaussian linear inverse problems:

\begin{itemize}
	\item For noiseless Gaussian linear inverse problems ($\sigma=0$), as mentioned above, exact recovery fails already in Regime II and so does it in Regime III. 
	\item For noisy Gaussian linear inverse problems ($\sigma>0$), as the fixed point equation (\ref{eqn:fixed_pt_eqn}) does not admit a solution in Regime III, it is natural to conjecture that the risk of $\hat{\mu}(\sigma)$ blows up. We formalize this below in the case where  $K$ is a closed convex cone with a diverging statistical dimension.
\end{itemize}

\begin{proposition}\label{prop:regime_low_m}
Suppose $K$ is a closed convex cone with $\delta_K\to \infty$. In the regime $\limsup_n (m/\delta_K)<1$, 
\begin{align*}
\Prob\bigg(\sup_{\tilde{\mu}(\sigma) \in \argmin\limits_{\mu \in K} \pnorm{Y-X\mu}{}^2} n^{-1}\pnorm{\tilde{\mu}(\sigma)-\mu_0}{}^2=\infty\,\hbox{ for all $\mu_0 \in K$}\bigg)\to 1. 
\end{align*}
\end{proposition}

The above proposition shows that in Regime III, with asymptotic probability $1$ there exists an exact minimizer of the constrained least squares problem (\ref{def:generic_est}) for which the risk can be arbitrarily large. The major difficulty here is similar to that in the noiseless setting: In Regime III, the random null space of $X$ must intersect $K$ in a non-trivial way with high probability, so the cone structure of $K$ entails the existence of an exact minimizer with arbitrarily large signal size.

\subsection{Vanishing and non-vanishing risks}

First we study a version of Theorem \ref{thm:risk_asymp} when the risk is vanishing.

\begin{theorem}\label{thm:risk_asymp_consist}
Suppose Assumptions \ref{assump:X_xi} and \ref{assump:K_cone} hold. If $m\gg \mathfrak{L}_n$, 
\begin{align}\label{cond:risk_consist}
\bar{r}_n^2\equiv n^{-1} \E \err\big( \omega_{m/n}(0)\big) \to 0,
\end{align}
and (R2) holds, then as $n \to \infty$,
\begin{align*}
n^{-1} \pnorm{\hat{\mu}(\sigma)-\mu_0}{}^2\pequiv \bar{r}_n^2+\bigop(\mathfrak{L}_n/m). 
\end{align*}
\end{theorem}

Theorem \ref{thm:risk_asymp_consist} above provides a convenient reduction of the abstract Theorem \ref{thm:risk_asymp} in the regime where consistent estimation of $\mu_0$ via $\hat{\mu}(\sigma)$ is possible. In this regime, we may bypass the non-linear fixed point equation (\ref{eqn:fixed_pt_eqn}) and directly resort to (\ref{cond:risk_consist}) to solve for the asymptotically exact convergence rate. Note that we have not explicitly assumed any real condition on $m$ beyond the technical one $m\gg \mathfrak{L}_n$. In fact, the proof shows that $m\gg \delta_K$ is a consequence of the condition (\ref{cond:risk_consist}), so the fixed point equation (\ref{eqn:fixed_pt_eqn}) eventually has a unique solution $r_n^2$. An implicit compatibility issue here, which will be verified during the proof of Theorem \ref{thm:risk_asymp_consist}, is that the $r_n^2$ solved from (\ref{eqn:fixed_pt_eqn}) will be asymptotically equivalent to the $\bar{r}_n^2$ defined via (\ref{cond:risk_consist}), whenever $\bar{r}_n^2\to 0$.

Next we study a version of Theorem \ref{thm:risk_asymp} with non-vanishing risks. 

\begin{theorem}\label{thm:risk_asymp_const}
	Suppose Assumptions \ref{assump:X_xi}-\ref{assump:K_cone} hold and $m>\delta_K$. Let $r_n$ be the unique solution to (\ref{eqn:fixed_pt_eqn}) such that $r_n\to \mathsf{r}$ for some $\mathsf{r} \in (0,\infty)$, and (R2) replaced by the following:
	\begin{enumerate}
		\item[(R2-c)] Either (R2) holds, or holds with $\omega_n$ replaced by $\omega_{\tau}(\mathsf{r})$ if additionally $\lim m/n=\tau \in (0,\infty)$.
	\end{enumerate}
	If $m\gg \mathfrak{L}_n$, then as $n \to \infty$, $n^{-1} \pnorm{\hat{\mu}(\sigma)-\mu_0}{}^2\pequiv \mathsf{r}^2$. 
\end{theorem}

Characterizing the exact risk when it is of constant order has recently received much attention in the literature, in the so-called `proportional high dimensional regime' $m/n\to \tau \in (0,\infty)$. We will not attempt to give a full literature review here; interested readers are referred to e.g. \cite{bayati2012lasso,thrampoulidis2015regularized,donoho2016high,elkaroui2018impact,thrampoulidis2018precise,sur2019modern,miolane2021distribution,bellec2021debias} for more references in this direction. In the setting of Theorem \ref{thm:risk_asymp_const}, a constant order risk does not apriori postulate $m/n \to \tau \in (0,\infty)$, and vice versa. However, simplification of the condition (R2) is possible in the prescribed proportional growth regime.

\section{Examples}\label{section:examples}

In this section we work out several concrete examples for the abstract theory in Theorem \ref{thm:risk_asymp}. Proofs for all results in this section can be found in Section \ref{section:proof_example}.

\subsection{Non-negative least squares}\label{section:orthant}

Consider the non-negative least squares (NNLS) problem, where $K\equiv K_+\equiv \{\mu \in \R^n: \mu\geq 0\}$. Such a non-negativity constraint arises naturally in a variety of statistical and optimization problems; see e.g. \cite{kudo1963multivariate,raubertas1986hypothesis,lawson1995solving,kim2008nonegative,chen2010nonnegativity} and references therein. The convex program (\ref{def:generic_est}) under the constraint $K_+$ is a quadratic programming with a (simple) linear constraint, so can be computed easily.

We will need some further notation to describe our results in this section. Let $\varphi,\Phi$ be the normal p.d.f. and c.d.f., and let for $x\geq 0$
\begin{align}\label{def:orthant_GH}
\mathsf{G}(x)&\equiv \Phi(x)-x\varphi(x)+x^2 \Phi(-x),\nonumber\\ \mathsf{H}(x)&\equiv \Phi(x)-\mathsf{G}(x)=x\varphi(x)-x^2\Phi(-x).
\end{align}

%
To avoid unnecessary notational complications, we will work out in the following theorem the asymptotics in a specific case where the coordinates of the signal $\mu_0$ follow the same distribution.

\begin{theorem}\label{thm:orthant}
Suppose Assumption \ref{assump:X_xi} holds. Suppose that $m/n>1/2$, and that the coordinates of $\mu_0=(\mu_{0,i}) \in K_+$ are independent and identically distributed as a non-negative random variable $U$ with $\E U^2<\infty$. The following hold for any (sequence of) near minimizer(s) $\hat{\mu}(\sigma)$ satisfying (\ref{def:approximate_minimizer}).
\begin{enumerate}
	\item The fixed point equation
	\begin{align}\label{eqn:orthant_fixed_point_eqn}
	\omega^2_{m/n}(r)\cdot \E \mathsf{G}\bigg(\frac{U}{\omega_{m/n}(r)}\bigg) = r^2
	\end{align}
	admits a unique solution $r_{n,+} \in (0,\infty)$. If furthermore $\liminf_n m/n>1/2$ and
	\begin{align}\label{cond:orthant_R2_1}
	 \frac{1}{\sigma^2}\cdot \limsup_n \omega_{m/n}^2(r_{n,+}) \E \mathsf{H}\bigg(\frac{U}{\omega_{m/n}(r_{n,+})}  \bigg)<1,
	\end{align}
	then as $n \to \infty$, $n^{-1}\pnorm{\hat{\mu}(\sigma)-\mu_0}{}^2\pequiv r_{n,+}^2$.
	\item Suppose $m/n\to \delta \in (1/2,\infty)$. Then the fixed point equation (\ref{eqn:orthant_fixed_point_eqn}), with $m/n$ replaced by $\delta$, admits a unique solution $\mathsf{r}_+(\delta)$ in $(0,\infty)$, for which $\delta \mapsto \mathsf{r}_+(\delta)$ is non-increasing on $(1/2,\infty)$ with $\lim_{\delta \uparrow \infty} \mathsf{r}_+(\delta)=0$, $\lim_{\delta \downarrow 1/2} \mathsf{r}_+(\delta)=\infty$. Furthermore,  if
	\begin{align}\label{cond:orthant_R2_2}
	\omega_\delta^2(\mathsf{r}_+(\delta)) \E \mathsf{H}\big(U/ \omega_\delta(\mathsf{r}_+(\delta))\big)<\sigma^2,
	\end{align}
	 we have $n^{-1}\pnorm{\hat{\mu}(\sigma)-\mu_0}{}^2\pequiv \mathsf{r}_+^2(\delta)$ as $n \to \infty$.

	\item Suppose $m/n \to \infty$. Then as $n \to \infty$, $
	n^{-1}\pnorm{\hat{\mu}(\sigma)-\mu_0}{}^2\pequiv (1-p_0/2)\sigma^2 n/m$ where $p_0\equiv \Prob(U=0)$.
\end{enumerate}
The equivalence in probability in the above statements is taken with respect to the randomness due to both $h$ and $\mu_0$.
\end{theorem}

The setting in Theorem \ref{thm:orthant} where the coordinates of $\mu_0$ are random with a `prior' distribution has been commonly adopted in a different direction; see e.g. \cite{bayati2012lasso,donoho2016high,sur2019modern} in a number of different high dimensional problems. From a purely technical standpoint, the proof of Theorem \ref{thm:orthant} suggests a technique to apply Theorem \ref{thm:risk_asymp} in such a random signal setting. 

Let us now examine some more concrete examples of Theorem \ref{thm:orthant}.

\begin{example}
Suppose that the random variable $U$ charges a point mass at $u\geq 0$, so $\mu_0 = u \bm{1}_n$ is deterministic.

\noindent (\textbf{Case 1}). Let $u=0$. Then $\mu_0=0$. This is a `degenerate' case in which Regime II in Figure \ref{fig:m_regime} does not exist: $\delta_{T_{K_+}(\mu_0)}=\delta_{K_+}=n/2$ (cf. \cite[Table 3.1]{amelunxen2014living}). Using $\mathsf{G}(0)=1/2$, it is easy to solve the fixed point equation (\ref{eqn:orthant_fixed_point_eqn}) to obtain $r_{n,+}^2 = \sigma^2\cdot (n/2)/(m-(n/2))$, and the condition (\ref{cond:orthant_R2_1}) is automatically fulfilled. Consequently, $n^{-1}\pnorm{\hat{\mu}(\sigma)}{}^2\pequiv \sigma^2\cdot (n/2)/(m-(n/2))$, provided $\liminf_n (m/n)>1/2$.

\noindent (\textbf{Case 2}). Let $u>0$. Then $\delta_{T_{K_+}(\mu_0)}=\delta_{\R_n}=n$, while $\delta_{K_+} = n/2$. In Regime I where $\liminf_n (m/n)>1$, the condition (\ref{cond:orthant_R2_1}) is satisfied by Proposition \ref{prop:R2_verify}-(1) (this can also be verified directly by using $H< 0.5\leq1-G$), so the risk is directly solvable from (\ref{eqn:orthant_fixed_point_eqn}). In Regime II where $1/2<\liminf_n (m/n)\leq \limsup_n (m/n)< 1$, risk asymptotics exist for $\hat{\mu}(\sigma)$ only if (\ref{cond:orthant_R2_1}) is verified. In fact, the counter-example constructed in the proof of Proposition \ref{prop:R2}-(3) falls in this regime that violates (\ref{cond:orthant_R2_1}).
\end{example}

\noindent \textbf{Illustrative simulation I}. We carry out a small simulation in Case 2 of the above example with $n=50, u=5$ and Gaussian error with noise level $\sigma=1$. The simulation result is summarized in the left panel of Figure \ref{fig:risk_asymp}. The theoretical risk $r_{n,+}$ from the fixed point equation (\ref{eqn:orthant_fixed_point_eqn}) (red curve) is computed via the iterative scheme in Proposition \ref{prop:exist_unique_fixed_pt}-(2). The approximate theoretical risk $\tilde{r}_{n,+}=(\sigma^2 n/m)^{1/2}$ (green curve) refers to the risk asymptotics in Theorem \ref{thm:orthant}-(3) that is proved to be valid in the regime $m/n \to \infty$. The empirical risk (blue curve) is computed via the Monte Carlo average over 1000 repetitions. By the left panel of Figure \ref{fig:risk_asymp}, the theoretical and empirical risk curves are almost indistinguishable. The gap between these curves and the approximate theoretical risk curve diminishes as $m$ grows. These numerical findings match the theory in Theorem \ref{thm:orthant}.

\subsection{Shape constrained problems}\label{section:shape}

In the Gaussian sequence model (\ref{model:seq}), shape constrained regression consists of a class of problems that impose certain qualitative structures on $K$. Two canonical examples are the monotone cone $K_{\uparrow}$ corresponding to univariate isotonic regression, and the cone $K_{\vee}$ corresponding to univariate convex regression with equally spaced design points:
\begin{align*}
K_{\uparrow}&\equiv \big\{\mu \in \R^n: \mu_1\leq \cdots \leq \mu_n\big\},\\
K_{\vee}& \equiv \big\{\mu \in \R^n: 2\mu_i\leq \mu_{i-1}+\mu_{i+1},\, i=2,\ldots,n-1\big\}.
\end{align*}
It is now well understood that the LSE $\hat{\mu}_K^{\seq}$ can `adapt' to certain `low-dimensional structures' associated with $K$. This is usually formulated using the so-called (sharp) oracle inequalities, e.g., \cite{chatterjee2015risk,bellec2018sharp,chatterjee2018matrix,han2019isotonic,fang2021multivariate,kur2020convex}. For instance, in the example of monotone cone $K_{\uparrow}$, the `low-dimensional structures' in $K_{\uparrow}$ refer to the class of piecewise constant signals $\cup_{1\leq k\leq n}\mathcal{M}_k$, where $\mathcal{M}_k$ denotes the class of all piecewise constant $\mu \in K_{\uparrow}$ with at most $k$ pieces. The isotonic LSE $\hat{\mu}_K^{\seq}$ satisfies a sharp oracle inequality (cf. \cite[Theoem 3.2]{bellec2018sharp}):
\begin{align}\label{ineq:orac_ineq_iso}
 \frac{1}{n}\cdot \E \err(\sigma)\leq \inf_{1\leq k\leq n} \bigg(\inf_{\mu \in \mathcal{M}_k}\frac{\pnorm{\mu-\mu_0}{}^2}{n}+ \frac{\sigma^2 k\log (en/k)}{n}\bigg).
\end{align}
The adaptive behavior of $\hat{\mu}_K^{\seq}$ refers to the fact if $\mu_0 \in \mathcal{M}_k$, then the LSE  $\hat{\mu}_K^{\seq}$ achieves a near parametric risk $n^{-1}\E \err(\sigma)\leq \sigma^2 k \log (en/k)/n$, as opposed to the much bigger nonparametric risk $n^{-1/3}$ for general $\mu_0 \in K_{\uparrow}$ (cf. \cite{zhang2002risk}). 

\begin{figure}[t]
	\begin{minipage}[t]{0.495\textwidth}
		\includegraphics[width=\textwidth]{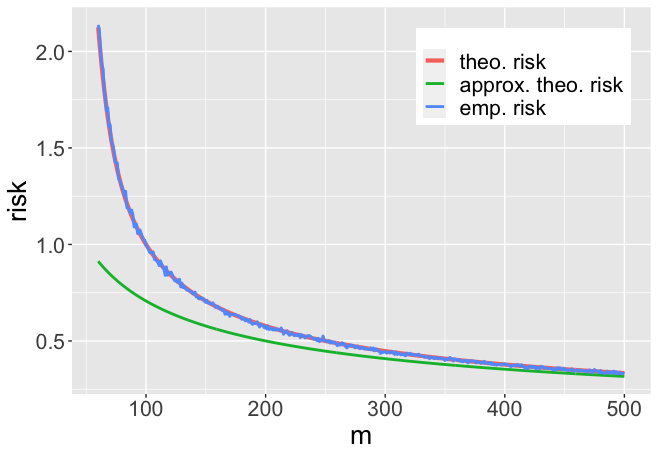}
	\end{minipage}
	\begin{minipage}[t]{0.495\textwidth}
		\includegraphics[width=\textwidth]{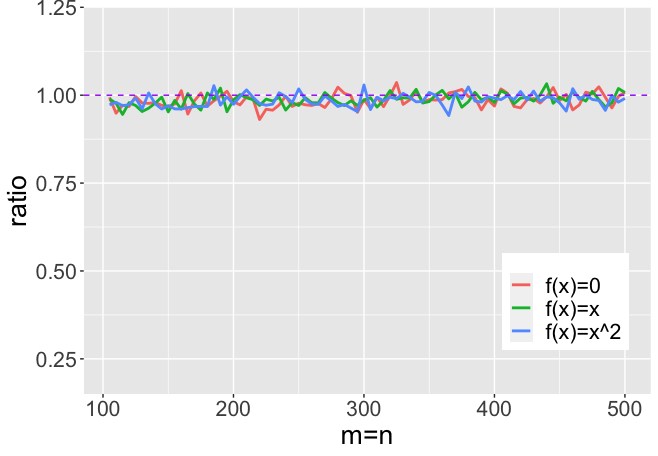}
	\end{minipage}
	\caption{$\sigma=1$. \emph{Left panel}: (Non-negative least squares) $n=50,u=5$. Theoretical risk = $r_{n,+}$; approximate theoretical risk = $(\sigma^2 n/m)^{1/2} = \tilde{r}_{n,+}$. \emph{Right panel}: (Isotonic regression) $100\leq m=n\leq 500, (\mu_0)_i=f(i/n)$. Ratio = theoretical risk/empirical risk. }
	\label{fig:risk_asymp}
\end{figure}

Our goal here is to provide (asymptotically) sharp oracle inequalities analogous to (\ref{ineq:orac_ineq_iso}), for the constrained LSE $\hat{\mu}(\sigma)$ in the noisy Gaussian linear measurement model (\ref{model:linear_inverse}) for general pairs of $(K,\mu_0)$. To this end, let us give a general formulation of `low-dimension structures' in $K$: For $1\leq k\leq n$, let
\begin{align*}
\mathscr{K}_k&\equiv \big\{K_S\cap K:  S=\{S_\ell\}_{\ell=1}^k \in \mathcal{P}_n\hbox{ is a partition of } [1:n],\\
&\qquad\qquad K_S\equiv \times_{\ell=1}^k K_{S_\ell}\equiv \{\nu \in \R^n: \nu|_{S_\ell} \in K_{S_\ell}\subset \R^{S_\ell}\}   \\
&\qquad\qquad \hbox{where $K_{S_\ell}$ is a closed convex cone s.t. $K|_{S_\ell}-K_{S_\ell}\subset K|_{S_\ell}$} \big\}.
\end{align*}

\begin{theorem}\label{thm:orac_ineq_shape}
Suppose Assumption \ref{assump:X_xi} holds and $K$ is a closed convex cone. Let
\begin{align*}
\mathcal{E}_K(\mu_0)\equiv \inf_{1\leq k\leq n}\inf_{ \substack{\mu \in K_S\cap K \in \mathscr{K}_k,\\ S=\{S_\ell\}_{\ell=1}^k \in \mathcal{P}_n }}\bigg(\frac{\pnorm{\mu-\mu_0}{}^2}{n}+ \frac{\sigma^2}{m}\sum_{\ell=1}^k \delta_{K|_{S_\ell}}\bigg).
\end{align*}
Then for any (sequence of) near minimizer(s) $\hat{\mu}(\sigma)$ satisfying (\ref{def:approximate_minimizer}), as $n \to \infty$ the normalized estimation error satisfies
\begin{align*}
n^{-1}\pnorm{\hat{\mu}(\sigma)-\mu_0}{}^2\leq (1+\smallop(1))\mathcal{E}_K(\mu_0) +\bigop\big(\mathfrak{L}_n/m\big),
\end{align*}
provided $m\gg \mathfrak{L}_n$, $\mathcal{E}_K(\mu_0)=\smallo(1)$ and $\inf_{\nu \in L(K)}\pnorm{\mu_0-\nu}{}/\sqrt{n}=\bigo(1)$.
\end{theorem}

Now we examine the implications of Theorem \ref{thm:orac_ineq_shape} in the canonical shape constrained problem of isotonic regression with $K=K_{\uparrow}$.

\begin{corollary}\label{cor:risk_asymp_iso}
Suppose Assumption \ref{assump:X_xi} holds and $K=K_{\uparrow}$. Then for any (sequence of) near minimizer(s) $\hat{\mu}(\sigma)$ satisfying (\ref{def:approximate_minimizer}), as $n \to \infty$ the normalized estimation error $n^{-1}\pnorm{\hat{\mu}(\sigma)-\mu_0}{}^2$ is bounded from above by
\begin{align*}
(1+\smallop(1))\bigg[\inf_{1\leq k\leq n} \bigg(\inf_{\mu \in \mathcal{M}_k}\frac{\pnorm{\mu-\mu_0}{}^2}{n}+ \frac{\sigma^2 k\log (en/k)}{m}\bigg) \bigg] +\bigop\bigg(\frac{\mathfrak{L}_n}{m}\bigg) =\smallop(1),
\end{align*}
provided $m\gg \log n$ and $\mu_{0,n}-\mu_{0,1}=\bigo(1)$. The term $\bigop\big(\mathfrak{L}_n/m\big)$ can be eliminated when $\mu_0 \in \mathcal{M}_k$. 
\end{corollary}

In isotonic regression, $\delta_{K_\uparrow}=\sum_{i=1}^n (1/i)\simeq \log n$ (cf. \cite[Eqn. (D.12)]{amelunxen2014living}), while $\delta_{T_{K_{\uparrow}}(\mu_0)}\gtrsim n^{1/3}$ holds for general smooth monotone signals $\mu_0\in K_{\uparrow}$\footnote{To see this, let $\mu_{0,i} = f_0(i/n)$, where $f_0:[0,1]\to [0,1]$ is a smooth increasing function with $f'$ continuously bounded away from $0$ and $\infty$, then \cite[Theorem 2]{meyer2000degrees} yields that
\begin{align*}
\delta_{T_{K_{\uparrow}}(\mu_0)}\geq \E \pnorm{\hat{\mu}_{K_{\uparrow}}(1)-\mu_0}{}^2\simeq \E \mathbb{D}_1^2\cdot \int_0^1 \big(f'(x)/2\big)^{2/3}\,\d{x}\cdot n^{1/3},
\end{align*}
where $\mathbb{D}_1$ is the (scaled) Chernoff distribution (cf. \cite{groeneboom2014nonparametric,han2019limit,han2022berry}). }. Consequently, exact recovery of such smooth $\mu_0$'s requires $m\gg n^{1/3}$ many samples (Regime I in Figure \ref{fig:m_regime}) in the noiseless Gaussian linear measurement model, while the above corollary shows that only as few as $m\gg \log n$ samples (Regime I+II in Figure \ref{fig:m_regime}) are needed for consistent recovery in the noisy setting.

One heuristic way to understand this phenomenon can be described as follows. The isotonic LSE $\hat{\mu}_{K_{\uparrow}}^{\seq}$ in the Gaussian sequence model is known to `adapt' to constant signals with a much smaller mean squared error of order $\log n$ compared to the general order $n^{1/3}$ for smooth signals (cf. \cite{zhang2002risk,chatterjee2015risk,bellec2018sharp}). Now as the risk behavior of $\hat{\mu}(\sigma)$ is intrinsically tied to that of $\sigma\mapsto \E\pnorm{\hat{\mu}_{K_{\uparrow}}^{\seq}(\sigma)-\mu_0}{}^2= \E\err_{(K_{\uparrow},\mu_0)}(\sigma)$ \emph{in the high noise limit}, a regime in which the regular signal $\mu_0$ is `collapsed' to a constant-like signal after rescaling by the noise level. This suggests that the LSE $\hat{\mu}_{K_{\uparrow}}^{\seq}$ is essentially learning a constant signal at this noise scale, which requires $m\gg \log n$ instead of $m\gg n^{1/3}$ for consistent recovery.

It is easy to generalize Corollary \ref{cor:risk_asymp_iso} to other common shape constrained cones $K$, for instance the ones corresponding to multiple isotonic/convex regression on a fixed lattice design by using the results in \cite{han2019isotonic,kur2020convex}. The phenomenon described above also continues hold for those problems. We omit the details. 

\vspace{0.5em}

\noindent \textbf{Illustrative simulation II}. We carry out another illustrative simulation study for the isotonic regression studied above. The constrained LSE $\hat{\mu}(\sigma)$ is computed by the AMP algorithm described essentially in \cite[Section 7.2]{berthier2020state}\footnote{ The design matrix in the AMP literature (cf. \cite{bayati2011dynamics,javanmard2013state,berthier2020state}) usually works with i.i.d. $\mathcal{N}(0,1/m)$ entries instead of $\mathcal{N}(0,1/n)$ entries used in this paper, so a proper rescaling is needed.}: Let the initialization be $\mu^0 \equiv 0 \in \R^n, r^0 \equiv Y \in \R^m$. Now suppose for $t\geq 0$, $(\mu^0,r^0),\ldots,(\mu^{t},r^{t}) \in \R^n\times \R^m$ have been computed, then the $(t+1)$-iteration is
\begin{align*}
\mu^{t+1}& \equiv \Pi_{K_{\uparrow}} \big((n/m)X^\top r^t+\mu^t\big) \in \R^n,\\
r^{t+1}&\equiv Y-X\mu^{t+1}+ \frac{k_{\uparrow}\big((n/m)X^\top r^t+\mu^t\big)}{m}\cdot r^t \in \R^m. 
\end{align*}
Here $k_{\uparrow}(\nu) =$ the number of constant pieces in $\Pi_{K_{\uparrow}}(\nu)$. In principle $\hat{\mu}(\sigma)$ can also be computed using quadratic programming with a linear constraint. However, the above AMP algorithm seems in general much faster, in particular for larger scales of $m,n$ where its convergence typically only takes a few iterations. 

The right panel of Figure \ref{fig:risk_asymp} shows the ratio of (the square root of) the theoretical risk and the empirical risk for three signals $(\mu_0)_i = f(i/n)$ corresponding to $f(x)\in \{0,x,x^2\}$, where $m=n$ ranges from $100$ to $500$. Gaussian error with noise level $\sigma=1$ is used. The empirical risk of $\hat{\mu}(\sigma)$ is computed via the AMP algorithm described above via 500 Monte Carlo averages. The theoretical risk, i.e., the solution to the fixed point equation (\ref{eqn:fixed_pt_eqn}), is computed via the iterative scheme in Proposition \ref{prop:exist_unique_fixed_pt}-(2). The risk map $\sigma \mapsto \E \pnorm{\hat{\mu}_{K_{\uparrow}}^{\seq}(\sigma)-\mu_0}{}^2$ does not have a closed form formula for general monotone $\mu_0$'s, so is evaluated by Monte Carlo simulations as well. All curves in the right panel of Figure \ref{fig:risk_asymp} are quite uniformly close to $1$, giving strong support for our theory and (\ref{eqn:risk_uniform}). Note that although here the choice $m/n=1$ seems superficially to fall in the proportional high dimensional regime, the risk of $\hat{\mu}(\sigma)$ actually vanishes so the setting intrinsically requires Theorem \ref{thm:risk_asymp}.

\subsection{Generalized Lasso}\label{section:lasso}
The Lasso \cite{tibshirani1996regression} in its constrained form can be realized in our setup by taking $K\equiv K_{\lambda}\equiv K_{\ell_1,\lambda}\equiv \{\mu\in \R^n: \pnorm{\mu}{1}\leq \lambda\}$. This is a special case of the more general formulation, where for some closed convex function $\mathsf{f}$, the constraint is described by $K_{(\mathsf{f},\mu_0)}\equiv \{\mu\in \R^n: \mathsf{f}(\mu)\leq \mathsf{f}(\mu_0)\}$, cf. \cite{oymak2013squared,thrampoulidis2014simple}.

\begin{theorem}\label{thm:risk_asymp_lasso}
Let $	\delta_0\equiv \delta_{T_{K_{(\mathsf{f},\mu_0)} }(\mu_0)}$. Suppose $\liminf_n (m/\delta_0)>1$ and $m\gg \log\log n$. Then for any (sequence of) near minimizer(s) $\hat{\mu}(\sigma)$ satisfying (\ref{def:approximate_minimizer}), as $n \to \infty$ the normalized estimation error satisfies
\begin{align*}
n^{-1}\pnorm{\hat{\mu}(\sigma)-\mu_0}{}^2&\pequiv r_n^2 + \bigop\big(\mathfrak{L}_n/m\big) \leq \sigma^2\frac{ \delta_0}{m-\delta_0}+\bigop\bigg(\frac{\log \log n}{m}\bigg).
\end{align*}
The inequality takes equality if $K_{(\mathsf{f},\mu_0)}-\mu_0$ is a closed convex cone. 
\end{theorem}

Let us now examine two concrete examples of Theorem \ref{thm:risk_asymp_lasso}.

\begin{example}
Let $K$ be a closed convex cone, $\mathsf{f}=\bm{0}_K$ and $\mu_0=0$. This is an exceptionally simple case, where Theorem \ref{thm:risk_asymp_lasso} applies to obtain
\begin{align*}
n^{-1}\pnorm{\hat{\mu}(\sigma)-\mu_0}{}^2\pequiv \sigma^2\frac{ \delta_K}{m-\delta_K}+\bigop\bigg(\frac{\log \log n}{m}\bigg)
\end{align*}
under $\liminf_n (m/\delta_K)>1$ and $m\gg \log \log n$. To put this result in the literature, \cite[Theorem 3.1]{oymak2013squared} proved the above formula in a low noise limit $\sigma \downarrow 0$ setting; \cite[Eqn. (37)]{thrampoulidis2018precise} proved the above formula in the proportional high dimensional regime $m/n\to \tau \in (0,\infty)$ with $\delta_{K_0}/n\to \bar{\delta}_0 \in (0,1)$. Non-exact results, i.e., upper bounds for the risk, are proved in e.g. \cite{thrampoulidis2014simple}. To the best knowledge of the author, the exact risk result above is new in this simple setting. 
\end{example}

\begin{example}\label{example:lasso}
Consider the constrained Lasso setting, where $\mathsf{f}(\mu)\equiv \pnorm{\mu}{1}$.  \cite[Eqn. (4.4)]{amelunxen2014living} shows that with $s\equiv \pnorm{\mu_0}{0}$, 
\begin{align*}
n\big(\psi(s/n)-2/\sqrt{sn}\big)\leq \delta_0= \delta_{ T_{ K_{(\pnorm{\cdot}{1},\mu_0)}(\mu_0) }  }\leq n\psi(s/n),
\end{align*}
where $\psi(\rho)\equiv \inf_{\gamma\geq 0}\big\{\rho(1+\gamma^2)+(1-\rho)\E\big(\abs{\mathcal{N}(0,1)}-\gamma\big)_+^2\big\}$. A simple upper bound for $\delta_0$ is given by \cite[Proposition 3.10]{chandrasekaran2012convex}, which states that $\delta_0\leq 2s \log(n/s)+5s/4$. Consequently, Theorem \ref{thm:risk_asymp_lasso} applies to obtain that $n^{-1}\pnorm{\hat{\mu}(\sigma)-\mu_0}{}^2\pequiv r_n^2$, provided $\liminf_n (m/\delta_0)>1$ and $m\gg \log\log n$ (which holds in particular in the common regime $m\gg s\log (n/s))$. 

Using the well-known fact on the representation of projection onto $\ell_1$ norm balls (cf. Lemma \ref{lem:proj_l1_ball}), $r_n^2$ can also be described by a more explicit fixed point equation. For $x \in \R^n$ and $\lambda\geq 0$, let $\mathsf{u}_{\lambda}(x) \in \R$ be the unique solution to the equation $
\bigpnorm{\big(\abs{x}-\mathsf{u}_\lambda (x)\bm{1}_n\big)_+ }{1} = \lambda$. 
Here the functions $\abs{\cdot},(\cdot)_+,\sign(\cdot)$ are understood as applied component-wise. Then $r_n^2$ can be characterized as the unique solution to
\begin{align}\label{eqn:fixed_pt_eqn_lasso}
n^{-1} \E \biggpnorm{\sign\big(y_n(r_n)\big)\Big(\abs{y_n(r_n)}- \big\{\mathsf{u}_{\pnorm{\mu_0}{1}}\big(y_n(r_n)\big)\big\}_+\bm{1}_n\Big)_+}{}^2 = r_n^2,
\end{align}
where $y_n(r)\equiv \mu_0+\omega_{m/n}(r) h$. 
\end{example}

\begin{remark}
We compare the risk asymptotics for the constrained Lasso in Example \ref{example:lasso} to the penalized Lasso: For $\sigma^2,\lambda>0$, define
\begin{align*}
\hat{\mu}_{\plasso}(\sigma;\lambda) \in \argmin_{\mu \in \R^n} \bigg\{\frac{1}{2} \pnorm{Y-X\mu}{}^2+\lambda \pnorm{\mu}{1}\bigg\}, 
\end{align*}
and its `counterpart' in the Gaussian sequence model (\ref{model:seq})
\begin{align*}
\hat{\mu}_{\plasso}^{\seq}(\sigma;\lambda) \equiv \argmin_{\mu \in \R^n} \bigg\{\frac{1}{2} \pnorm{y-\mu}{}^2+\lambda \pnorm{\mu}{1}\bigg\} = \Big((\abs{y_i}-\lambda)_+\cdot \sign(y_i)\Big)_i. 
\end{align*}
In the proportional high dimensional regime $m/n\to \tau \in (0,\infty)$, \cite[Corollary 1.6]{bayati2012lasso} proved that, under several other conditions, for any fixed $\sigma^2,\lambda>0$, there exists some $\lambda_\ast>0$ and $\omega_\ast>\sigma$ which solves certain fixed point equation, and
\begin{align}\label{eqn:plasso_risk}
n^{-1}\pnorm{ \hat{\mu}_{\plasso}(\sigma;\lambda)-\mu_0 }{}^2 \pequiv n^{-1} \E \pnorm{\hat{\mu}_{\plasso}^{\seq}(\omega_\ast;\lambda_\ast)-\mu_0}{}^2.
\end{align}
Our results show that for the constrained Lasso as in Example \ref{example:lasso}, as long as $m\gg s\log (n/s)$, with $\omega_n\equiv \omega_{m/n}(r_n)$ where $r_n$ solves the fixed point equation (\ref{eqn:fixed_pt_eqn_lasso}), 
\begin{align}\label{eqn:classo_risk}
n^{-1}\pnorm{ \hat{\mu}(\sigma)-\mu_0 }{}^2 \pequiv n^{-1} \E \bigpnorm{\hat{\mu}^{\seq}_{ K_{(\pnorm{\cdot}{1},\mu_0)} }(\omega_n)-\mu_0}{}^2.
\end{align}
Clearly (\ref{eqn:plasso_risk}) and (\ref{eqn:classo_risk}) are similar in spirit: the risks of the penalized and constrained Lasso can be characterized by their counterparts in the Gaussian sequence model with a different noise level. On the other hand, for the penalized Lasso the effective noise level is always inflated $\omega_\ast>\sigma$ in the proportional high dimensional regime $m/n\to \tau \in (0,\infty)$, while for the constrained Lasso the effective noise level $\omega_n$ can in principle be either inflated or deflated. Furthermore, (\ref{eqn:classo_risk}) for the constrained Lasso holds in a much wider regime $m\gg s \log (n/s)$ than the proportional high dimensional regime $m/n \to \tau \in (0,\infty)$ as required for (\ref{eqn:plasso_risk}). 
\end{remark}

\section{Estimation error, LRT and DoF processes}\label{section:est_err_lrt_process}

In this section we present several important analytic and probabilistic results for $\err(\cdot)$, $\lrt(\cdot)$ defined in (\ref{def:F}) and (\ref{def:lrt}) respectively, as well as $\dof(\cdot)$ to be defined in (\ref{def:dof}) below. These results will be essential to the proofs of the main results in Section \ref{section:risk_asymp}. Proofs for most results in this section can be found in Section \ref{section:proof_est_lrt_dof}. 

\subsection{Estimation error process}\label{section:est_process}

We start with analytic properties of $\err(\cdot)$.

\begin{lemma}\label{lem:monotone_est_err}
	The following hold.
	\begin{enumerate}
		\item The map $\sigma \mapsto \err(\sigma)$ is non-decreasing on $[0,\infty)$.
		\item $0\leq \err'(\sigma)\leq 2\err(\sigma)/\sigma$ a.e., or equivalently, the map $\sigma \mapsto \err(\sigma)/\sigma^2$ is non-increasing on $(0,\infty)$. 
	\end{enumerate}
\end{lemma}

Both claims in Lemma \ref{lem:monotone_est_err} are important qualitative statements for the estimation error process. The monotonicity of $\sigma \mapsto \err(\sigma)/\sigma^2$ will be fundamental in the proof of a number of results, including the existence and uniqueness of the solution to the fixed point equation (\ref{eqn:fixed_pt_eqn}) in Proposition \ref{prop:exist_unique_fixed_pt}.  Another particularly important consequence of the monotonicity properties of the estimation process processes proved in the above lemma is its stability, explicitly formulated as below. 
\begin{proposition}\label{prop:est_err_stability}
	For any $\sigma\geq 0, M\geq 1$, 
	\begin{align*}
	\err(\sigma)\leq \err(M\sigma)\leq M^2 \err(\sigma).
	\end{align*}
\end{proposition}
\begin{proof}
The case $\sigma=0$ is trivial so we prove the claim for $\sigma>0$ below. By monotonicity of $\sigma \mapsto \err(\sigma)/\sigma^2$ in Lemma \ref{lem:monotone_est_err}-(2), we have
\begin{align*}
\frac{\err(M\sigma)}{(M\sigma)^2}\leq \frac{\err(\sigma)}{\sigma^2}\quad\Rightarrow\quad  \err(M\sigma)\leq M^2 \err(\sigma),
\end{align*}
proving the right inequality. The left equality follows from the monotonicity of $\sigma \mapsto \err(\sigma)$ in Lemma \ref{lem:monotone_est_err}-(1).
\end{proof}

Next we derive several useful probabilistic properties for $\err(\cdot)$.

\begin{proposition}\label{prop:est_err_variance}
	The following hold.
	\begin{enumerate}
		\item $\lim_{\sigma\uparrow \infty}\pnorm{\Pi_K(\sigma h)}{}^2/\sigma^2 \leq \err(\sigma)/\sigma^2 \leq \pnorm{h}{}^2$ for $\mu_0 \in K$. The lower bound is achieved as $\sigma \uparrow\infty$, as well as its expectation version: $\lim_{\sigma\uparrow \infty} \E \err(\sigma)/\sigma^2 =\delta_K$. Furthermore, $\lim_{\sigma \downarrow 0} \E \err(\sigma)/\sigma^2=\delta_{T_K(\mu_0)}$. Consequently, $\delta_K\leq \delta_{T_K(\mu_0)}$ for any $\mu_0 \in K$.
		\item The variance bound $\var\big(\err(\sigma)\big)\leq 4\sigma^2\E \err(\sigma)$ holds. 
		\item For any $\lambda < 1/(2\sigma^2)$,
		\begin{align*}
		\E \exp\big(\lambda (\err(\sigma)-\E \err(\sigma))\big)\leq \exp\bigg(\frac{2\sigma^2\lambda^2 \E \err(\sigma)}{1-2\sigma^2\lambda }\bigg).
		\end{align*} 
		Consequently, for all $t\geq 0$,
		\begin{align*}
		\Prob\big(\abs{\err(\sigma)-\E \err(\sigma)}\geq \sqrt{8\sigma^2\cdot \E \err(\sigma)\cdot t}+2\sigma^2 t\big)\leq 2e^{-t}.
		\end{align*}
	\end{enumerate}
	
\end{proposition}

The expected low noise limit in Proposition \ref{prop:est_err_variance}-(1) is known, see e.g. \cite[Theorem 1.1]{oymak2016sharp}. The variance bound and the exponential inequality in \ref{prop:est_err_variance}-(2)(3), proved using Poincar\'e and log-Sobolev inequalities, appear to be new. These results are closely related to some results in \cite{chatterjee2014new,van2015concentration}. For instance, \cite[Theorem 1.1]{chatterjee2014new} shows the (Gaussian) concentration of $\err^{1/2}(\sigma)=\pnorm{\hat{\mu}_K^{\seq}(\sigma)-\mu_0}{}$ around $t_{\mu_0}$, defined as the location of maximum for the map $t\mapsto  \E\sup_{\nu \in K: \pnorm{\nu-\mu_0}{}\leq t}\iprod{\sigma h}{\nu-\mu_0}-t^2/2$. When $t_{\mu_0}$ is replaced by $\E \err^{1/2}(\sigma)= \E \pnorm{\hat{\mu}_K^{\seq}(\sigma)-\mu_0}{}$, Gaussian concentration follows by the Lipschitz property of $h\mapsto \err^{1/2}(\sigma)$ (see e.g. \cite[Theorem 2.1]{van2015concentration} for more general formulations). These results imply non-exact large deviation inequalities for $\err(\sigma)$ with respect to $(1\pm \epsilon)\E \err(\sigma)$. Here we show in (3) via a different method that the concentration of $\err(\sigma)$ that can be centered exactly at $\E \err(\sigma)$. Furthermore, the variance bound in (2) does not contain a Poisson component. 

As an illustration of the usefulness of the developed analytic and probabilistic properties of the estimation error process, we prove the following result which is essentially Proposition \ref{prop:exist_unique_fixed_pt}-(1).

\begin{proposition}\label{prop:unique_fixed_point_eqn}
	The following hold.
	\begin{enumerate}
		\item The map $r\mapsto \err\big(\omega_{m/n}(r)\big)/r^2$ is non-increasing and strictly decreasing at $r>0$ such that $\err\big(\omega_{m/n}(r)\big)>0$. The same conclusion holds when $\err(\cdot)$ is replaced by its expectation $\E \err(\cdot)$. 
		\item The fixed point equation 
		\begin{align*}
		\E \err\big(\omega_{m/n}(r)\big) = n r^2
		\end{align*}
		has at most one solution in $r \in (0,\infty)$ that exists if and only if $m>\delta_K$. 
	\end{enumerate}
\end{proposition}
\begin{proof}
	\noindent (1). We consider a rescaled version
	\begin{align*}
	G(r)\equiv \frac{\err\big(\omega_{m/n}(r)\big) }{nr^2} &= \frac{\err\big(\omega_{m/n}(r)\big)
	}{(r^2+\sigma^2)/(m/n)}\cdot \frac{1+(\sigma^2/r^2) }{m}\equiv G_1(r)G_2(r).
	\end{align*}
	By Lemma \ref{lem:monotone_est_err}-(2), $G_1$ is non-increasing. Clearly $G_2$ is strictly decreasing, so $G$ is non-increasing and strictly decreasing when $G_1>0$. By Proposition \ref{prop:est_err_stability}, $G$ is also continuous. The same argument applies to the expectation version.

	\noindent (2). 
	If $K=\{\mu_0\}$, then $r=0$ is the only solution. So let us  assume  $K\neq \{\mu_0\}$. Then $\E \err(\sigma)>0$ for all $\sigma>0$, and therefore the map $r\mapsto \E \err\big(\omega_{m/n}(r)\big)/r^2 = \E G(r)$ is strictly decreasing. This means that there can be at most one solution to  the equation $\E G(r)=1$. Now as $r\mapsto \E G(r)$ is continuous and strictly decreasing with $\lim_{r \downarrow 0} \E G(r)=\infty$. The (unique) solution $r \in (0,\infty)$ to $\E G(r)=1$ exists if and only if $\lim_{r\uparrow \infty} \E G(r)<1$. Clearly $\lim_{r\uparrow \infty} \E G_2(r) = 1/m$, and by Proposition \ref{prop:est_err_variance}-(1), $
	\lim_{r\uparrow \infty} \E G_1(r) = \lim_{\sigma \uparrow \infty} {\E \err(\sigma)}/{\sigma^2}=\delta_K$. Consequently, $
	\lim_{r \uparrow \infty} \E G(r) = \delta_K/m$, 
	which would be less than $1$ if and only if $m>\delta_K$, completing the proof of (2).
\end{proof}

\subsection{LRT process}\label{section:lrt_process}

 First some analytic properties for $\lrt(\cdot)$:

\begin{lemma}\label{lem:lrt_qual}
	The following hold for $\lrt(\cdot)$ defined in (\ref{def:lrt}).
	\begin{enumerate}
		\item 	For any $\sigma>0$, we have $
		\lrt(\sigma) -\sigma\cdot \lrt'(\sigma) = - \err(\sigma)$. 
		\item For any (possibly random) $\nu$ such that $K+\nu \subset K$ and $\hat{\mu}_K^{\seq}(\sigma)-\nu \in K$, 
		\begin{align*}
		\lrt(\sigma)=\min_{s \in \R^n} \big\{\pnorm{\sigma h-s}{}^2-2\big(s^\top (\mu_0-\nu)-\bm{0}_K^\ast(s)\big)\big\},
		\end{align*}
		where $\bm{0}_K^\ast(\cdot)=\sup_{t\in K}t^\top(\cdot)$ is the support function of $K$, or equivalently, the convex conjugate of the indicator function $\bm{0}_K$. 
		\item $(\lrt(\sigma)/\sigma)' = \err(\sigma)/\sigma^2$, so $\sigma\mapsto \lrt(\sigma)/\sigma$ is non-decreasing and concave. 
		\item The map $\sigma\mapsto \lrt(\sigma)/\sigma^2$ is non-increasing. 
	\end{enumerate}
\end{lemma}

Lemma \ref{lem:lrt_qual}-(1) is a simple but useful result that can be verified directly by definition. Lemma \ref{lem:lrt_qual}-(2) gives an important variational characterization of $\lrt(\cdot)$. This characterization, proved using convex duality and Sion's min-max theorem (cf. Lemma \ref{lem:sion_minmax}), will be particularly useful in terms of bounding $\lrt(\cdot)$ by $\err(\cdot)$ \emph{from above}. Similar to the stability estimate in Proposition \ref{prop:est_err_stability} for $\err(\cdot)$, the monotonicity properties for $\lrt(\cdot)$ proved Lemma \ref{lem:lrt_qual}-(3)(4) immediately yield the following.

\begin{proposition}\label{prop:lrt_stability}
For any $\sigma\geq 0$ and $M\geq 1$,
\begin{align*}
M\lrt(\sigma) \leq \lrt(M\sigma) \leq M^2 \lrt(\sigma^2).
\end{align*}
\end{proposition}
\begin{proof}
We only need to consider $\sigma>0$. The right inequality follows from the same arguments as that of the proof of Proposition \ref{prop:est_err_stability} upon using Lemma \ref{lem:lrt_qual}-(4). The left inequality follows the same strategy, but using Lemma \ref{lem:lrt_qual}-(3), which entails that $\lrt(\sigma)/\sigma \leq \lrt(M\sigma)/(M\sigma)$. 
\end{proof}

Next we derive several useful probabilistic properties for $\lrt(\cdot)$.

\begin{proposition}\label{prop:lrt_prop}
	The following hold for $\lrt(\cdot)$ defined in (\ref{def:lrt}).
	\begin{enumerate}
		\item $\E \lrt (\sigma)\leq 2\sigma^2 \E \dv \hat{\mu}_K^{\seq}(\sigma)\leq 2\sigma^2 n$.
		\item For any $\sigma>0$, we have
		\begin{align*}
		\err(\sigma)\leq \lrt(\sigma)\leq \err(\sigma)+ 2\sigma \int_0^\sigma \bigg(\frac{\err(\tau)}{\tau^2}-\frac{\err(\sigma)}{\sigma^2}\bigg)\,\d{\tau}.
		\end{align*}		
		Suppose further that $K$ is a closed convex cone. For any $\sigma >0$, and any (possibly random) $\nu$ such that $K+\nu \subset K$ and $\hat{\mu}_K^{\seq}(\sigma)-\nu \in K$,
		\begin{align*}
		\err(\sigma)&\leq \lrt(\sigma)\leq \err(\sigma)+2\iprod{ \hat{\mu}_K^{\seq}(\sigma)-\mu_0-\sigma h}{\mu_0-\nu}.
		\end{align*}
		The right hand side inequality takes equality when $\nu = \hat{\mu}_K^{\seq}(\sigma)$. 
		\item The variance bound $\var\big(\lrt(\sigma)\big)\leq 4\sigma^2 \E \err(\sigma)$ holds. 
		\item For any $\lambda \in \R$ with $\lambda^2< \E \err(\sigma)/(8\sigma^2)$, 
		\begin{align*}
		&\E \exp\bigg[\lambda\cdot \frac{\lrt (\sigma)-\E \lrt (\sigma)}{\sigma\cdot \E^{1/2} \err(\sigma)} \bigg]\leq \exp\bigg(\frac{4\lambda^2}{1-8 \big(\sigma^2/\E \err(\sigma)\big) \lambda^2 }\bigg).
		\end{align*}
		Consequently, there exists some absolute constant $L>0$ such that
		\begin{align*}
		&\Prob\Big( L^{-1}\bigabs{\lrt (\sigma)-\E \lrt (\sigma)} \geq \sqrt{\sigma^2\cdot \E \err(\sigma)\cdot t}+\sigma^2 t \Big)\leq L e^{-t/L}
		\end{align*}
		holds for all $t\geq 0$.
	\end{enumerate}
\end{proposition}

Proposition \ref{prop:lrt_prop}-(2) provides two powerful inequalities tracking the difference between $\lrt(\cdot)$ and $\err(\cdot)$. The first inequality is generic and tight for some (many) choice(s) of $(K,\mu_0)$, while the second inequality  is tight for \emph{every} choice of $(K,\mu_0)$ at the cost of a stronger cone condition on $K$. These inequalities will be essential in understanding and verifying the condition (R2) in Theorem \ref{thm:risk_asymp}.

We note that although the appearance of Proposition \ref{prop:lrt_prop}-(4) is similar to Proposition \ref{prop:est_err_variance}-(3), the proof takes a rather different route by resorting to `exponential Poincar\'e-type inequalities' due to \cite{bobkov1999exponential}. The subtle point here is that we need the variance component to scale like $\sigma^2 \cdot \E \err(\sigma)$ rather than the bigger quantity $\sigma^2 \cdot \E \lrt(\sigma)$ in the exponential inequality in the proof of the main Theorem \ref{thm:risk_asymp}. This sub-gaussian tail behavior of $\lrt(\cdot)$ is quite natural and cannot be improved in view of the Gaussian approximation for $\lrt(\cdot)$ proved in \cite[Theorem 3.1]{han2022high}.

\subsection{DoF process}\label{section:dof}

Define the (scaled) degree-of-freedom (DoF) process
\begin{align}\label{def:dof}
\dof(\sigma)\equiv \iprod{\hat{\mu}_K^{\seq}(\sigma)-\mu}{\sigma h}.
\end{align}
This definition is different from  \cite{meyer2000degrees,kato2009degrees}, where the quantity $\dv\hat{\mu}_K^{\seq}(\sigma)$ is defined as the `degree-of-freedom' associated with $\hat{\mu}_K^{\seq}$ in the Gaussian sequence model (\ref{model:seq}). As we will see below, the two definitions agree \emph{in expectation} modulo a multiplicative scaling factor $\sigma^2$. We will work with the definition (\ref{def:dof}) above, as thus defined $\dof(\cdot)$ is both directly connected to, and also shares similar properties as $\err(\cdot)$ and $\lrt(\cdot)$ studied in the previous subsections. Below we summarize some useful analytic and probabilistic properties for $\dof(\cdot)$.

\begin{proposition}\label{prop:dof}
The following hold for $\dof(\cdot)$ defined in (\ref{def:dof}).
\begin{enumerate}
	\item $\E \dof(\sigma)=\sigma^2\E \dv\hat{\mu}_K^{\seq}(\sigma)$.
	\item $\sigma \mapsto \dof (\sigma)/\sigma^2 $ is non-increasing, and for any $\sigma\geq 0, M\geq 1$, we have $\dof (\sigma)\leq \dof (M\sigma) \leq M^2 \dof(\sigma)$.
	\item For any $\sigma>0$, we have
	\begin{align*}
	\err(\sigma)\leq \dof(\sigma)\leq \sigma \int_0^\sigma \frac{\err(\tau)}{\tau^2}\,\d{\tau}.
	\end{align*}
	Consequently, $\lim_{\sigma \downarrow 0} \E \dv\hat{\mu}_K^{\seq}(\sigma)= \lim_{\sigma\downarrow 0} \E\err(\sigma)/\sigma^2 = \delta_{T_K(\mu_0)}$.
	\item The variance bound $\var(\dof(\sigma))\leq 4\sigma^2 \E \err(\sigma)$ holds. Furthermore there exists some absolute constant $L>0$ such that
	\begin{align*}
	&\Prob\Big( L^{-1}\bigabs{\dof (\sigma)-\E \dof (\sigma)} \geq \sqrt{\sigma^2\cdot \E \err(\sigma)\cdot t}+\sigma^2 t \Big)\leq L e^{-t/L}
	\end{align*}
	holds for all $t\geq 0$.
\end{enumerate}
\end{proposition}

The inequality in Proposition \ref{prop:dof}-(3) provides an important quantitative link between $\err,\dof,\lrt$ by tracking the tightness of the easy inequality $\err(\sigma)\leq \dof(\sigma)\leq \lrt(\sigma)$. This inequality also plays a key role in the proof of Proposition \ref{prop:lrt_prop}-(2). In the expectation form, this inequality reads
\begin{align*}
\E\err(\sigma)\leq \sigma^2\E \dv\hat{\mu}_K^{\seq}(\sigma) \leq \sigma \int_0^\sigma \frac{\E\err(\tau)}{\tau^2}\,\d{\tau}.
\end{align*}
The left hand side is essentially known in \cite[Eqn. (10), pp. 1086]{meyer2000degrees}. The right hand side seems genuinely new. Furthermore, the upper bound is tight in the low noise limit as $\sigma \downarrow 0$.

\subsection{A uniform concentration inequality}\label{section:exp_ineq_err_lrt}

Using the analytic and probabilistic properties for $\err(\cdot),\lrt(\cdot),\dof(\cdot)$ derived in the previous subsections, we may prove the following uniform concentration inequality.

\begin{proposition}\label{prop:est_err_lrt_sup}
	Let $H\in \{\err,\lrt, \dof\}$. There exists a universal constant $C>0$ such that for any $0<\epsilon_0<M_0$ and $t\geq 1$, 
	\begin{align*}
	&\Prob\bigg[\abs{H(\sigma)-\E H(\sigma)}\leq C\big( \sigma \cdot \E^{1/2} \err(\sigma)\cdot \sqrt{t}+ \sigma^2\cdot t \big),\quad \forall \sigma \in [\epsilon_0, M_0]\bigg]\\
	&\geq 1- Ce^{-t/C}\cdot \big(1+\delta_{T_K(\mu_0)}\big)\cdot  \log_+(M_0/\epsilon_0)\geq 1-2C e^{-t/C}\cdot n \log_+(M_0/\epsilon_0).
	\end{align*}
\end{proposition}

For the above inequality to be meaningful in applications, we need to choose $t$ that grows at a certain rate depending on the growth of $M_0/\epsilon_0$. In the proof of Theorem \ref{thm:risk_asymp} in the next section, we will use $\mathfrak{L}_n$ (defined in (\ref{def:Ln})) to control the growth of $t$. Compared to the choice $t\asymp \log n$, this refined choice is beneficial in e.g. Corollary \ref{cor:risk_asymp_iso} and Theorem \ref{thm:risk_asymp_lasso} that further reduces $\log n$ to $\log \log n$.

\section{Proof outline of Theorem \ref{thm:risk_asymp}}\label{section:proof_sketch}

The basic approach of the proof of Theorem \ref{thm:risk_asymp} is to reduce the primal optimization (PO) problem (\ref{def:generic_est}) to a simpler, but probabilistically `equivalent' auxiliary optimization (AO) problem (cf. Theorem \ref{thm:CGMT}). This method of reduction is now well understood; see \cite{thrampoulidis2018precise}. Here with the help of the variational characterization of $\lrt(\cdot)$ proved in Lemma \ref{lem:lrt_qual} and an appropriate reparametrization, it can be shown that we only need to deal with the AO problem
\begin{align}\label{ineq:proof_sketch_1}
\Psi^{\textrm{a}}(g,h)&\equiv \frac{1}{2} \min_{\alpha \in [0,L_w]}\max_{\eta\geq 0}\bigg[   \sqrt{ \pnorm{r_n \alpha g+\xi}{}^2/m}-\frac{\alpha }{2\eta}\cdot r_n^2- \frac{1}{2\alpha \eta}\cdot  \frac{1}{n} \lrt\bigg(\frac{\alpha\eta}{\sqrt{m/n}}\bigg)\bigg]_+^2\nonumber\\
& = \frac{1}{2} \min_{\alpha \in [0,L_w]} \max_{\eta \geq 0}  \mathfrak{D}_n(\alpha,\eta)_+^2 \equiv \mathcal{R}(\mathfrak{D}_n;[0,L_w])
\end{align}
for some $L_w>1$ large enough; see Proposition \ref{prop:adjusted_AO} for a formal statement.  Here the randomness on the standard Gaussian vector $h \in \R^n$ is implicit in $\lrt(\cdot)$. 

The goal now is to show that for the $r_n>0$ chosen according to the fixed point equation (\ref{eqn:fixed_pt_eqn}), the minimizer of $\min_{\alpha \in [0,L_w]}$ in the above AO should be very close to $1$ whenever $r_n^2\gg \mathfrak{L}_n/m$. The basic logic to show this is to prove an assertion of the following type: for every small but fixed $\epsilon>0$, 
\begin{align}\label{ineq:proof_sketch_2}
\hbox{$\mathcal{R}(\mathfrak{D}_n;[0,L_w]_\epsilon)$ is `larger' than $\mathcal{R}(\mathfrak{D}_n;[0,L_w])$ w.h.p. in the limit}, 
\end{align}
where $[0,L_w]_\epsilon\equiv [0,1-\epsilon]\cup [1+\epsilon,L_w]$.

To motivate the discussion of our approach, it is useful to have a sense of how (\ref{ineq:proof_sketch_2}) works in the `proportional high dimensional regime' (cf. \cite{thrampoulidis2018precise}). This regime postulates a non-degenerate limit for the objective function in the AO, i.e., $\mathfrak{D}_n\to \mathfrak{D}$ for some non-trivial $\mathfrak{D}$ in an appropriate sense. Then (\ref{ineq:proof_sketch_2}) reduces to essentially a deterministic inequality $\mathcal{R}(\mathfrak{D};[0,L_w]_\epsilon)> \mathcal{R}(\mathfrak{D};[0,L_w])$. Clearly, such an approach will be useful only if $\mathfrak{D}$ is non-degenerate. 

In our setting, for most interesting problem instances $(K,\mu_0)$, in particular those with vanishing risks, the limit $\mathfrak{D}$ is degenerate $\mathfrak{D}(\cdot,\cdot)\equiv \sigma$, so the above method of analysis necessarily fails. This suggests that in order to analyze (\ref{ineq:proof_sketch_1}), we need to study (i) the precise order of the gap between suitable versions of $\mathcal{R}(\mathfrak{D}_n;[0,L_w]_\epsilon), \mathcal{R}(\mathfrak{D}_n;[0,L_w])$ and (ii) their stochastic fluctuations. It turns out this rough idea can be formalized by a conditional argument on a `good event' of $\xi \in \R^m$, on which the aforementioned two intertwined issues can be resolved at the same time all the way down to $r_n^2\gg \mathfrak{L}_n/m$. More concretely:
\begin{enumerate}
	\item The version of $\mathfrak{D}_n(\alpha,\eta)$ we will be working with is
	\begin{align}\label{ineq:proof_sketch_5}
	\mathsf{D}_n(\alpha,\eta)\equiv \sqrt{r_{n,\xi}^2\alpha^2+\pnorm{\xi}{}^2/m}-\frac{\alpha }{2\eta}\cdot r_{n,\xi}^2 - \frac{1}{2\alpha \eta}\cdot  \frac{1}{n} \lrt\bigg(\frac{\alpha\eta}{\sqrt{m/n}}\bigg),
	\end{align}
	where $r_{n,\xi}$ solves the (conditional) fixed point equation (\ref{eqn:fixed_pt_eqn}) with $\sigma^2$ therein replaced by $\pnorm{\xi}{}^2/m$; see (\ref{def:rn_xi}) below for a formal definition. So conditional on the `good event' of $\xi \in \R^m$, the randomness in $\mathsf{D}_n(\alpha,\eta)$ is entirely driven by the Gaussian vector $h \in \R^n$ in $\lrt(\cdot)$. The precise meaning of the `good event' of $\xi \in \R^m$ will be stated in Definition \ref{def:xi_good}.

   \item Using $\mathsf{D}_n$, we will establish the following version of (\ref{ineq:proof_sketch_2}): For any fixed $\epsilon>0$, conditional on the `good event' of $\xi$, there exists $c_\epsilon>0$ such that
   \begin{align}\label{ineq:proof_sketch_3}
   \mathcal{R}(\mathsf{D}_n;[0,L_w]_\epsilon)\geq \mathcal{R}(\mathsf{D}_n;[0,L_w])+ c_\epsilon\cdot r_n^2\quad \hbox{w.h.p}
   \end{align}
   for all $n$ large enough. See Proposition \ref{prop:gap_det_eps} for a formal statement. The gap order $r_n^2$ in (\ref{ineq:proof_sketch_3}) is essential, as will be seen below.
   \item  Next we study the (conditional) stochastic fluctuation problem. For the risk of the constrained LSE $\hat{\mu}(\sigma)$ in PO to be related to AO in a probabilistically `equivalently way', it is necessary that either side of (\ref{ineq:proof_sketch_3}) should be roughly deterministic for large enough $n$. We achieve this conditional `de-stochastization' step by showing that the random variable on the right hand side of (\ref{ineq:proof_sketch_3}) can be replaced by a conditionally deterministic quantity, with fluctuations controlled strictly below the gap order  $r_n^2$ in (\ref{ineq:proof_sketch_3}):
   \begin{align}\label{ineq:proof_sketch_4}
   \mathcal{R}(\mathsf{D}_n;[0,L_w]) = \mathcal{R}(\E^\xi \mathsf{D}_n;[0,L_w])+ \smallopx(r_n^2).
   \end{align}
   The above equality is formally established in Proposition \ref{prop:local_destoc}. We mention that it is important to choose the right hand side, instead of the left hand side, of (\ref{ineq:proof_sketch_3}) for the above conditional de-stochastization step. The conditional stochastic fluctuation of $\mathcal{R}(\mathfrak{D}_n;[0,L_w]_\epsilon)$ seems much harder to control, in particular in the region ${\alpha \in [1+\epsilon, L_w]}$.  
	
\end{enumerate}

Now we may explain the reason for choosing the version $\mathsf{D}_n$ in (\ref{ineq:proof_sketch_5}). The key point therein is to separate the term $\pnorm{\xi}{}^2/m$ apart from the calculations of stochastic fluctuations that are targeted below the gap order $r_n^2$ as in (\ref{ineq:proof_sketch_3}). In fact,  for each fixed $\alpha,\eta>0$, conditionally on $\xi$, an easy calculation shows that $\abs{\mathsf{D}_n(\alpha,\eta)-\mathfrak{D}_n(\alpha,\eta)}=\smallopx(r_n^2)$. If $\pnorm{\xi}{}^2/m$ is replaced by its limit $\sigma^2$ in $\mathsf{D}_n$ and proceed with unconditional arguments, the fluctuation will be necessarily of a much larger order $\max\{\smallo(r_n^2), m^{-1/2}\}$, where the hard threshold $m^{-1/2}$ comes from the fluctuation of $\pnorm{\xi}{}^2/m$. To put this in other words, the main reason for adopting a conditional argument on $\xi$ is that the speed at which $\mathfrak{D}_n$ converges can be much slower than the targeted gap order $r_n^2$, so keeping $\pnorm{\xi}{}^2/m$ in $\mathsf{D}_n$ amounts to decoupling the undesirably large stochastic fluctuations due to $\pnorm{\xi}{}^2/m$. On the other hand, the potentially slower convergence of $\mathfrak{D}_n$ does not cause problems in the risk analysis, as $r_{n,\xi}$ and $r_n$ are asymptotically equivalent as long as $\pnorm{\xi}{}^2/m$ is consistent for $\sigma^2$ (cf. Lemma \ref{lem:R2_deconditioning}).

The key inequalities (\ref{ineq:proof_sketch_3}) and (\ref{ineq:proof_sketch_4}), valid all the way down to almost the parametric rate $r_n^2\gg \mathfrak{L}_n/m$, are proved using very different ideas that will be of more technical nature, so will be detailed in Section \ref{section:proof_main_result} below. Clearly, in view of the form of $\mathfrak{D}_n, \mathsf{D}_n$ in (\ref{ineq:proof_sketch_1}) and (\ref{ineq:proof_sketch_5}), the analytical and probabilistic results on $\err(\cdot),\lrt(\cdot)$ and other related quantities in Section \ref{section:est_err_lrt_process} will be crucial, for obtaining the correct gap order $r_n^2$ in (\ref{ineq:proof_sketch_3}) and the conditional fluctuation order $\smallopx(r_n^2)$ in (\ref{ineq:proof_sketch_4}).

\section{Proof of Theorem \ref{thm:risk_asymp}}\label{section:proof_main_result}

\subsection{Some further notation}

We introduce some further notation. Let
\begin{align}
\sigma_m^2 \equiv \pnorm{\xi}{}^2/m.
\end{align}
Let $r_{n,\xi}>0$ be the solution to the fixed point equation
\begin{align}\label{def:rn_xi}
\E^\xi \err \big(\omega_{m/n}(r,\sigma_m)\big) = n r^2. 
\end{align}
Then for $\sigma_m>0$, which holds on the `good event' in Definition \ref{def:xi_good} below, there exists a unique solution $r_{n,\xi}>0$ if and only if $m>\delta_K$ (cf. Proposition \ref{prop:exist_unique_fixed_pt}).

Fix any slowly growing sequence $\mathfrak{u}_n\uparrow \infty$ with $\mathfrak{u}_n\mathfrak{L}_n/m \to 0$, and let
\begin{align}\label{def:r_vee}
r_\vee^2&\equiv r_\vee^2(\mathfrak{u}_n)\equiv \max\Big\{r_{n,\xi}^2, \frac{\mathfrak{u}_n \mathfrak{L}_n}{m}\Big\},\quad \bar{r}_\vee^2 \equiv \bar{r}_\vee^2(\mathfrak{u}_n)\equiv \max\Big\{r_{n}^2, \frac{2\mathfrak{u}_n \mathfrak{L}_n}{m}\Big\}.
\end{align}
Notational dependence on $\mathfrak{u}_n$ will typically be suppressed for simplicity. In the case $m/\mathfrak{L}_n\geq M_n$ for some $M_n\uparrow \infty$, we choose $\mathfrak{u}_n\equiv M_n^{1/2}$.
\begin{lemma}\label{lem:r_vee_fixed_pt_eqn}
For $r_\vee $ defined in (\ref{def:r_vee}), $
\E^\xi \err\big(\omega_{m/n}(r_\vee,\sigma_m)\big)\leq nr_\vee^2$
with equality if and only if $r_\vee = r_{n,\xi}$. 
\end{lemma}
\begin{proof}
By definition, $r_\vee \geq r_{n,\xi}$, so by monotonicity of $r\mapsto \E^\xi \err(\omega_{m/n}(r,\sigma_m))/r^2$ proved in Proposition \ref{prop:unique_fixed_point_eqn}-(1), we have $\E^\xi \err\big(\omega_{m/n}(r_\vee,\sigma_m)\big)/nr_\vee^2 \leq \E^\xi \err\big(\omega_{m/n}(r_{n,\xi},\sigma_m)\big)/ nr_{n,\xi}^2=1$, as desired.
\end{proof}

With $r_\vee $ defined in (\ref{def:r_vee}), let
\begin{align}
\mathsf{D}(\alpha,\eta)&\equiv \sqrt{r_\vee^2\alpha^2+\sigma_m^2}-\frac{\alpha }{2\eta}\cdot r_\vee^2 - \frac{1}{2\alpha \eta}\cdot  \frac{1}{n} \lrt\bigg(\frac{\alpha\eta}{\sqrt{m/n}}\bigg),\label{def:D_alpha_eta}
\end{align}
and $\overline{\mathsf{D}}(\alpha,\eta)\equiv \E^\xi \mathsf{D}(\alpha,\eta)$. Derivatives of $\mathsf{D}$ are given by:
\begin{align}
\frac{\d \mathsf{D}}{\d \alpha}
& =  \frac{\alpha r_\vee^2 }{ \sqrt{r_\vee^2\alpha^2+\sigma_m^2 } }-\frac{1}{2\eta}\cdot r_\vee^2  - \frac{1}{2\alpha^2 \eta}\cdot \frac{  \err\big(\alpha\eta/\sqrt{m/n}\big) }{n},\label{eq:dD_dalpha}\\
\frac{\d \mathsf{D}}{\d \eta}
& = \frac{\alpha}{2\eta^2}\cdot r_\vee^2- \frac{1}{2\alpha \eta^2}\cdot \frac{  \err\big(\alpha\eta/\sqrt{m/n}\big)  }{n}. \label{eq:dD_deta}
\end{align}	
In the above calculations we used Lemma \ref{lem:lrt_qual}-(3). The derivatives for $\overline{\mathsf{D}}$ take the same form upon changing $\err(\cdot)$ to $\E^\xi \err(\cdot)$ in (\ref{eq:dD_dalpha})-(\ref{eq:dD_deta}).

\subsection{`Good event' of $\xi$}

We now define formally the good event of $\xi \in \R^m$ on which the conditioning arguments will be performed. 

\begin{definition}\label{def:xi_good}
	Fix some slowly growing sequences $M_n^r, M_n^\sigma \uparrow \infty$ with $M_n^r \wedge M_n^\sigma\geq 2$ and $C_0>1$. Define the good event $E(C_0)\equiv E(\{M_n^r\},\{M_n^\sigma\},C_0)$ to be the collection of all $\xi \in \R^m$ such that:
	\begin{enumerate}
		\item $\sigma_m$ and $\sigma$ are close enough in the sense that
		\begin{align}\label{ineq:sigma_m_high_prob}
		\biggabs{ \frac{\sigma_m^2}{\sigma^2}-1}\leq \frac{1}{M_n^\sigma}.
		\end{align}
		\item $r_{n,\xi}$ and $r_n$ are close enough in the sense that
		\begin{align}\label{ineq:rn_xi_high_prob}
		\biggabs{ \frac{r_{n,\xi}^2}{r_n^2}-1}\leq \frac{1}{M_n^r}.
		\end{align}
		\item (R2) holds conditionally in the sense that
		\begin{align}\label{ineq:R2_xi}
		\frac{1}{2n\sigma_m^2}\bigg(\E^\xi \lrt\big(\omega_{m/n}(r_{n,\xi},\sigma_m)\big)-\E^\xi \err\big(\omega_{m/n}(r_{n,\xi},\sigma_m)\big)\bigg)\leq 1-1/C_0.
		\end{align}
		\item  It holds that $r_\vee \leq C_0$.
	\end{enumerate}
\end{definition}

Below we establish that the above definition indeed leads to a `good event'.

\begin{lemma}\label{lem:R2_deconditioning}
	The following hold.
	\begin{enumerate}
		\item (\ref{ineq:sigma_m_high_prob}) holds with probability tending to $1$ for some slowly growing $M_n^\sigma$.
		\item $\min\{\sigma_m^2/\sigma^2,\sigma^2/\sigma_m^2\}\leq r_{n,\xi}^2/r_n^2 \leq \max\{\sigma_m^2/\sigma^2,\sigma^2/\sigma_m^2\}$. So $r_{n,\xi}\pequiv r_n$, and (\ref{ineq:rn_xi_high_prob}) holds with probability tending to $1$ for some slowly growing $M_n^r$. 
		\item Suppose (R2) holds. Then (\ref{ineq:R2_xi}) holds with probability tending to $1$ for sufficiently large $C_0>1$. 
	\end{enumerate}
	Consequently, under (R2) and $r_n\lesssim 1$, for $\{M_n^\sigma\}$ chosen according to (1), $\{M_n^r\}$ according to (2) and large enough $C_0>1$, the event $E(C_0)$ in Definition \ref{def:xi_good} satisfies $\Prob(E(C_0))\to 1$. 
\end{lemma}
\begin{proof}
	\noindent (1). This follows immediately from $\sigma_m^2\pequiv \sigma^2$.
	
	\noindent (2). 	By the stability estimate in Proposition \ref{prop:est_err_stability}, 
	\begin{align*}
	nr_{n,\xi}^2  &= \E^\xi \err\big(\omega_{m/n}(r_{n,\xi},\sigma_m)\big)\leq \bigg(1\vee \frac{r_{n,\xi}^2+\sigma_m^2}{r_n^2+\sigma^2}\bigg) \E \err\big(\omega_{m/n}(r_n)\big).
	\end{align*}
	It is easy to solve that $r_{n,\xi}^2/r_n^2 \leq \max\{\sigma_m^2/\sigma^2,\sigma^2/\sigma_m^2\}$. A reversed inequality replacing $\max$ to $\min$ can be similarly shown.

	\noindent (3). By the stability estimates in Propositions \ref{prop:est_err_stability} and \ref{prop:lrt_stability}, coupled with (1) and (2), we have $
	\E^\xi H\big(\omega_{m/n}(r_{n,\xi},\sigma_m)\big) \pequiv \E H(\omega_{m/n}(r_n))$ for $H \in \{\err,\lrt\}$. So
	\begin{align*}
	\hbox{LHS of (\ref{ineq:R2_xi}) } = \frac{1+\smallop(1)}{2n\sigma^2}\big(\E \lrt(\omega_{m/n}(r_n))-\E \err(\omega_{m/n}(r_n))\big) + \smallop(1). 
	\end{align*}
	The claim follows.
\end{proof}

\subsection{Identifying the PO and AO}
As mentioned in Section \ref{section:proof_sketch}, the general principle of the reduction scheme from the PO problem to an AO problem is now well understood \cite{thrampoulidis2018precise}. Here we spell out some details, with a particular eye on the scaling issue and conditional arguments.

We first rewrite the objective function in (\ref{def:generic_est}).  Let $\mathsf{L}(v)\equiv \pnorm{v}{}^2/2$.

\begin{proposition}\label{prop:PO}
Let $G\equiv \sqrt{n} X$ be the normalized Gaussian matrix so the entries of $G$ are i.i.d. standard normal. The PO is
\begin{align*}
\Phi^{\textrm{p}}(G)&\equiv\min_{\mu \in K } \frac{\pnorm{Y-X\mu}{}^2}{2m} =  \min_{v\in \R^m, w \in \R^n} \max_{u \in \R^m} \bigg[\frac{1}{\sqrt{m}} r_\vee\cdot u^\top G (-w) + Q(u,v,w) \bigg],
\end{align*}
where
\begin{align}\label{def:Q_PO}
Q(u,v,w)&\equiv 
\frac{1}{\sqrt{m}}\big(u^\top \xi - u^\top v\big) + \frac{1}{m}\Big(\mathsf{L}(v)+  \bm{0}_{K}(\mu_0+\sqrt{n}r_\vee w)\Big).
\end{align}
\end{proposition}
\begin{proof}
Note that
\begin{align*}
&\Phi^{\textrm{p}}(G)=\min_{\mu \in \R^n } \frac{1}{m}\cdot \bigg[ \mathsf{L}(Y-X\mu)+\bm{0}_{K}(\mu)\bigg] \\
& = \min_{w \in \R^n } \frac{1}{m}\cdot \bigg[\mathsf{L}\big(Y-X(\mu_0+\sqrt{n}r_\vee w)\big)+ \bm{0}_{K}(\mu_0+\sqrt{n}r_\vee w)\bigg]\\
&\qquad\qquad \qquad\qquad \qquad \hbox{ (change of variable $\mu=\mu_0+\sqrt{n}r_\vee w$)}\\
& = \min_{w \in \R^n } \frac{1}{m}\cdot \bigg[\mathsf{L}\big(\xi-r_\vee Gw\big)+ \bm{0}_{K}(\mu_0+\sqrt{n}r_\vee w)\bigg] \quad \hbox{ (use $Y=X\mu_0+\xi$)}\\
& = \min_{v\in \R^m,w \in \R^n} \frac{1}{m}\cdot\bigg[\mathsf{L}(v)+ \bm{0}_{K}(\mu_0+\sqrt{n}r_\vee w)\bigg]\qquad  \hbox{subject to $v=\xi-r_\vee Gw$.}
\end{align*} 
Now adding dual variable $u$ in the above optimization, $\Phi^{\textrm{p}}(G)$ becomes
\begin{align*}
& \min_{v\in \R^m,w \in \R^n} \max_{u \in \R^m} \bigg[\frac{1}{\sqrt{m}}\Big(u^\top (\xi-r_\vee Gw-v)\Big)+\frac{1}{m}\Big(\mathsf{L}(v)+ \bm{0}_{K}(\mu_0+\sqrt{n}r_\vee w)\Big)\bigg]\\
& =  \min_{v\in \R^m, w \in \R^n} \max_{u \in \R^m} \bigg[\frac{1}{\sqrt{m}} r_\vee\cdot u^\top G (-w) + Q(u,v,w) \bigg],
\end{align*} 
as desired. 
\end{proof}
For any subset $S_u,S_v,S_w \subset [0,\infty)$, define
\begin{align}\label{def:PO_restrict}
\Phi^{\textrm{p}}_{S_u, S_v, S_w}(G)\equiv \min_{\substack{v\in \R^m, w \in \R^n, \\\pnorm{v}{}\in S_v, \\\pnorm{w}{}\in S_w}} \max_{\substack{u \in \R^m, \\\pnorm{u}{}\in S_u}} \bigg[\frac{1}{\sqrt{m}} r_\vee\cdot u^\top G (-w) + Q(u,v,w) \bigg].
\end{align}
In the special case where $S_u=[0,L_u], S_v=[0,L_v], S_w=[0,L_w]$, we simply write $\Phi^{\textrm{p}}_{[0,L_u], [0,L_v], [0,L_w]}(G)=\Phi^{\textrm{p}}_{L_u, L_v, L_w}(G)$.
The constants $(L_u,L_v,L_w)$ will always come with subscript to indicate the variable for which localization is applied. A constant is often omitted when taken $\infty$. For instance,  $\Phi^{\textrm{p}}_{L_w}(G)=\Phi^{\textrm{p}}_{\infty, \infty, L_w}(G)$.

\begin{lemma}\label{lem:localization_w}
Take $L_w>1$. Fix $\xi \in \R^m$ and a sequence $\{\epsilon_n\}$. 
\begin{enumerate}
	\item If any sequence of  $\epsilon_n$-optimizers $w_{\ast,L_w}^{\textrm{p}}$ for $\Phi^{\textrm{p}}_{L_w}(G)$ satisfies $\pnorm{w_{\ast,L_w}^{\textrm{p}} }{}\pequivx 1$, then any original sequence of $\epsilon_n$-optimizers $w_{\ast}^{\textrm{p}}$  for $\Phi^{\textrm{p}}(G)$ also satisfies $\pnorm{w_{\ast}^{\textrm{p}} }{}\pequivx 1$. 
	\item If any sequence of $\epsilon_n$-optimizers $w_{\ast,L_w}^{\textrm{p}}$ for $\Phi^{\textrm{p}}_{L_w}(G)$ satisfies $\pnorm{w_{\ast,L_w}^{\textrm{p}} }{}\leq L'$ with $\Prob^\xi$-asymptotic probability 1 for some $L'<L_w$, then any original sequence of $\epsilon_n$-optimizers $w_{\ast}^{\textrm{p}}$  for $\Phi^{\textrm{p}}(G)$ also satisfies $\pnorm{w_{\ast}^{\textrm{p}} }{}\leq L'$ with $\Prob^\xi$-asymptotic probability 1. 
\end{enumerate}
\end{lemma}
\begin{proof} 
We only prove (1). (2) is completely similar. The proof is similar to \cite[Lemma A.1]{thrampoulidis2018precise}. Fix $\epsilon>0$. By assumption $\pnorm{w_{\ast,L_w}^{\textrm{p}}}{} \in [1-\epsilon,1+\epsilon]$ with $\Prob^\xi$-asymptotic probability $1$. Let
\begin{align*}
\mathsf{O}(w)\equiv \frac{1}{m}\cdot \bigg[\mathsf{L}\big(\xi-r_\vee Gw\big)+ \bm{0}_{K}(\mu_0+\sqrt{n}r_\vee w)\bigg].
\end{align*}
Then $\mathsf{O}(w_{\ast}^{\textrm{p}})\leq \mathsf{O}(w_{\ast,L_w}^{\textrm{p}})+\epsilon_n$ by the global near optimality of $w_{\ast}^{\textrm{p}}$ with respect to $\mathsf{O}$. This means $w_{\ast}^{\textrm{p}}$ either is the optimizer in the range $\pnorm{w}{}\leq L_w$; or $w_{\ast}^{\textrm{p}}$ falls out of the range, i.e., $\pnorm{w_{\ast}^{\textrm{p}}}{}>L_w$. The latter cannot happen for $n$ large: if it happens, then any point $w$ on the line segment of $w_{\ast}^{\textrm{p}}$ and $w_{\ast,L_w}^{\textrm{p}}$ must satisfy $\mathsf{O}(w)\leq \mathsf{O}(w_{\ast,L_w}^{\textrm{p}})+\epsilon_n$. In particular, any such $w$ that are close enough, but distinct to $w_{\ast,L_w}^{\textrm{p}}$ will be an $\epsilon_n$-optimizer for $\Phi^{\textrm{p}}_{L_w}(G)$, a contradiction to the assumption any  sequence of $\epsilon_n$-optimizers of  $\Phi^{\textrm{p}}_{L_w}(G)$ must have $\pnorm{\cdot}{}$ length converging to $1$ in $\Prob^\xi$-probability. 
\end{proof}

Below in the index $u,v \in \R^m, w \in \R^n$ we often suppress indications of the dimension $m,n$. The corresponding AO problem (cf. Theorem \ref{thm:CGMT}) now reads
\begin{align*}
&\tilde{\Phi}^{\textrm{a}}_{L_w}(g,h) =\min_{v, \pnorm{w}{}\leq L_w} \max_{u} \bigg[\frac{1}{\sqrt{m}}r_\vee\cdot \pnorm{w}{}g^\top u+ \frac{1}{\sqrt{m}}r_\vee\cdot \pnorm{u}{} h^\top (-w) + Q(u,v,w) \bigg].
\end{align*}
In similar spirit to (\ref{def:PO_restrict}), we may define $\tilde{\Phi}^{\textrm{a}}_{L_u,L_v,L_w}(g,h)$ by restricting the range of $u$ to $\pnorm{u}{}\leq L_u, \pnorm{v}{}\leq L_v$ in the above definition. To motivate the adjusted AO that will be actually used, note that by using the definition of $Q$ in (\ref{def:Q_PO}) and the duality $\bm{0}_{K}(\mu_0+\sqrt{n}r_\vee w)=\sup_s \big(s^\top (\mu_0+\sqrt{n} r_\vee w)- \bm{0}_{K}^\ast(s)\big)$, $\tilde{\Phi}^{\textrm{a}}_{L_w}(g,h)$ equals
\begin{align*}
& \min_{v, \pnorm{w}{}\leq L_w} \max_{u,s} \bigg[\frac{1}{\sqrt{m}}\Big(r_\vee\pnorm{w}{}g^\top u- r_\vee\pnorm{u}{} h^\top w + u^\top \xi - u^\top v\Big)\\
&\qquad\qquad \qquad\qquad \qquad  + \frac{1}{m}\Big( \mathsf{L}(v)+ s^\top (\mu_0+\sqrt{n} r_\vee w)- \bm{0}_{K}^\ast(s)\Big) \bigg]. 
\end{align*}
Following \cite{thrampoulidis2018precise}, the adjusted version of AO problem that will be used is a min-max flipped version of the above display by further writing $\max_u = \max_{\beta\geq 0, \pnorm{u}{}=\beta}$:
\begin{align}\label{eqn:adjusted_AO}
&\Phi^{\textrm{a}}_{L_w}(g,h) \equiv \max_{\beta \geq 0,s} \min_{v, \pnorm{w}{}\leq L_w} \max_{\pnorm{u}{}=\beta}  \bigg[\frac{1}{\sqrt{m}}\Big(r_\vee\pnorm{w}{}g^\top u- r_\vee\pnorm{u}{} h^\top w + u^\top \xi - u^\top v\Big)\nonumber\\
&\qquad\qquad \qquad\qquad \qquad  + \frac{1}{m}\Big( \mathsf{L}(v)+ s^\top (\mu_0+\sqrt{n} r_\vee w)- \bm{0}_{K}^\ast(s)\Big) \bigg].
\end{align}
In similar spirit to (\ref{def:PO_restrict}), we may define $\Phi^{\textrm{a}}_{L_\beta,L_v, L_w}(g,h)$ by restricting the range of $\beta, v$ to $\beta \leq L_\beta, \pnorm{v}{}\leq L_v$ in the above definition.

\begin{lemma}\label{lem:apriori_localization_u_v}
	Fix $\xi \in \R^m$. For any $L_w>1,\epsilon>0$, there exists some  $C=C(L_w,\epsilon)>0$ such that with $L_u=L_\beta\geq C(r_\vee\sqrt{n/m}+r_\vee+\pnorm{\xi}{}/m^{1/2})$, $L_v\equiv m^{1/2} L_u$, for any closed subset $S_w \subset [0,L_w]$, all probabilities $\Prob^\xi\big(\Phi^{\textrm{p}}_{S_w}(G)=\Phi^{\textrm{p}}_{L_u,L_v,S_w}(G)\big)$, $ \Prob^\xi\big(\tilde{\Phi}^{\textrm{a}}_{S_w}(g,h)=\tilde{\Phi}^{\textrm{a}}_{L_u, L_v, S_w}(g,h)\big)$ and $\Prob^\xi\big(\Phi^{\textrm{a}}_{S_w}(g,h)=\Phi^{\textrm{a}}_{L_\beta, L_v, S_w}(g,h) \big)$ exceed $1-\epsilon$.
\end{lemma}
\begin{proof}
	The proof adapts the idea from \cite[Lemma A.2]{thrampoulidis2018precise}. First consider localization of the PO problem $\Phi^{\textrm{p}}_{L_w}(G)$. By the first order optimality condition with respect to $u,v$, any saddle point $(u_\ast,v_\ast,w_\ast)$ to $\Phi^{\textrm{p}}_{L_w}(G)$ satisfies
	\begin{align*}
	\frac{1}{\sqrt{m}}\big(-r_\vee G w_\ast+\xi-v_\ast\big)=0,\quad \frac{1}{\sqrt{m}}(-u_\ast)+\frac{1}{m}\cdot v_\ast =0,
	\end{align*}
	 which gives $u_\ast = \frac{1}{\sqrt{m}} v_\ast = \frac{1}{\sqrt{m}}(-r_\vee G w_\ast+\xi)$. Consequently,
	 \begin{align*}
	 \pnorm{u_\ast}{}\leq m^{-1/2}\big(r_\vee \pnorm{G}{\op}\pnorm{w_\ast}{}+\pnorm{\xi}{}\big) = \bigopx\big(r_\vee\sqrt{n/m}+r_\vee+\pnorm{\xi}{}/m^{1/2}\big),
	 \end{align*} 
	 and $\pnorm{v_\ast}{}= \sqrt{m}\pnorm{u_\ast}{}$. This proves the claim for the PO. Next consider localization of the AO problem $\tilde{\Phi}^{\textrm{a}}_{L_w}(g,h)$. Again by the first optimality condition with respect to $u,v$, any saddle point $(\tilde{u},\tilde{v},\tilde{w})$ satisfies
	\begin{align*}
	\frac{1}{\sqrt{m}}\big(r_\vee \pnorm{\tilde{w}}{} g - r_\vee (\tilde{u}/\pnorm{\tilde{u}}{}) h^\top \tilde{w}+\xi-\tilde{v}\big)=0,\,\frac{1}{\sqrt{m}}(-\tilde{u})+\frac{1}{m}\cdot \tilde{v}=0,
	\end{align*}
	which give $
	\tilde{u} = \frac{1}{\sqrt{m}} \tilde{v} = \frac{1}{\sqrt{m}} \big(r_\vee \pnorm{\tilde{w}}{} g - r_\vee (\tilde{u}/\pnorm{\tilde{u}}{}) h^\top \tilde{w}+\xi\big)$. 
	Consequently, 
	\begin{align*}
	\pnorm{\tilde{u}}{}\lesssim m^{-1/2}\big(r_\vee(\pnorm{g}{}+\pnorm{h}{})+\pnorm{\xi}{}\big) = \bigopx\big(r_\vee\sqrt{n/m}+r_\vee+\pnorm{\xi}{}/m^{1/2}\big),
	\end{align*}
	and $\pnorm{\tilde{v}}{}=\sqrt{m}\pnorm{\tilde{u}}{}$. 
	A completely similar argument applies to the adjusted AO. 
\end{proof}

\begin{proposition}\label{prop:adjusted_AO}
The adjusted AO problem $\Phi^{\textrm{a}}_{L_w}(g,h)$ in (\ref{eqn:adjusted_AO}) equals
\begin{align*}
\frac{1}{2} \min_{\alpha \in [0,L_w]}\max_{\eta\geq 0}\bigg[   \sqrt{ \pnorm{r_\vee \alpha g+\xi}{}^2/m}-\frac{\alpha }{2\eta}\cdot r_\vee^2- \frac{1}{2\alpha \eta}\cdot  \frac{1}{n} \lrt\bigg(\frac{\alpha\eta}{\sqrt{m/n}}\bigg)\bigg]_+^2.
\end{align*}
Furthermore, on the event $E(C_0)$ (see Definition \ref{def:xi_good}), 
we have
\begin{align*}
\Phi^{\textrm{a}}_{L_w}(g,h) 
&=\frac{1}{2} \min_{\alpha \in [0,L_w]} \max_{\eta\geq 0}\mathsf{D}(\alpha,\eta)_+^2+ \smallopx(r_\vee^2).
\end{align*}
Here the probability estimate in $\smallopx$ is uniform with respect to problem instances for a fixed choice of $C_0, L_w$. 
\end{proposition}
\begin{proof}
First taking maximum over $u$ with $\pnorm{u}{}=\beta$, $\Phi^{\textrm{a}}_{L_w}(g,h)$ equals
\begin{align*}
&  \max_{\beta \geq 0,s} \min_{v, \pnorm{w}{}\leq L_w} \bigg[ \frac{1}{\sqrt{m}}\Big(-r_\vee\beta h^\top w+ \beta \bigpnorm{ r_\vee\pnorm{w}{}g+\xi-v }{}\Big)\\
&\qquad\qquad \qquad\qquad \qquad  + \frac{1}{m}\Big( \mathsf{L}(v)+ s^\top (\mu_0+\sqrt{n} r_\vee w)- \bm{0}_{K}^\ast(s)\Big) \bigg]. 
\end{align*}
Minimizing over $w$ using the same trick as above for $u$, $\Phi^{\textrm{a}}_{L_w}(g,h)$ equals
\begin{align*}
&\max_{s,\beta\geq 0} \min_{v} \bigg[\frac{1}{\sqrt{m}}\cdot \min_{\pnorm{w}{}\leq L_w}\bigg\{w^\top\Big(-r_\vee\beta h+(r_\vee/\sqrt{m/n}) s\Big)\\
&\qquad\qquad + \beta \bigpnorm{ r_\vee\pnorm{w}{}g+\xi-v }{} \bigg\} +\frac{1}{m}\Big(\mathsf{L}(v)  + \big(s^\top \mu_0- \bm{0}_{K}^\ast(s)\big)\Big) \bigg]\\
& = \max_{s,\beta\geq 0} \min_{v, \alpha \leq L_w} \bigg[\frac{1}{\sqrt{m}}\bigg\{-\alpha\bigpnorm{r_\vee\beta h-(r_\vee/\sqrt{m/n}) s}{}+ \beta \bigpnorm{r_\vee \alpha g+\xi-v }{}\bigg\} \\
&\qquad\qquad\qquad\qquad\qquad+ \frac{1}{m}\Big( \mathsf{L}(v)  +  \big(s^\top \mu_0- \bm{0}_{K}^\ast(s)\big)\Big)\bigg].
\end{align*}
As the objective function is jointly convex in $(\alpha,v)$ and concave in $(\beta,s)$ and the range of $\alpha$ is bounded, we may apply Sion's min-max theorem (cf. Lemma \ref{lem:sion_minmax}) so the above $\max_{s,\beta\geq 0} \min_{v, \alpha \leq L_w} = \min_{\alpha \leq L_w} \max_{\beta\geq 0} \max_s \min_{v}$. As the variables $s,v$ are decoupled in the objective function in the above display, $\max_s \min_{v}$ can be freely interchanged. So by the writing $\pnorm{t}{}=\min_{\tau>0}\big(\tau/2+ \pnorm{t}{}^2/(2\tau)\big)$ where $t\in \{r_\vee\beta h- (r_\vee/\sqrt{m/n})s, r_\vee\alpha g+\xi-v\}$,  $\Phi^{\textrm{a}}_{L_w}(g,h)$ becomes
\begin{align*}
& \min_{\alpha \in [0,L_w]}\max_{\beta \geq 0} \max_{\tau_h>0} \min_{\tau_g>0}
\bigg[\frac{1}{\sqrt{m}}\cdot \frac{\beta \tau_g}{2}+\frac{1}{\sqrt{m}}\cdot \frac{\beta}{2\tau_g } \bigpnorm{r_\vee\alpha g+\xi-v}{}^2 + \frac{1}{m}\cdot \mathsf{L}(v)\\
&\qquad - \frac{1}{\sqrt{m}}\cdot \frac{\alpha \tau_h}{2}- \frac{1}{\sqrt{m}}\cdot \frac{\alpha}{2\tau_h} \bigpnorm{r_\vee\beta h-(r_\vee/\sqrt{m/n})s}{}^2+\frac{1}{m}\big( s^\top \mu_0 - \bm{0}_{K}^\ast(s)\big) \bigg].
\end{align*}
Rescaling $\tau_g,\tau_h$ by $\sqrt{m}\tau_g,\sqrt{m}\tau_h$, $\Phi^{\textrm{a}}_{L_w}(g,h)$ becomes
\begin{align*}
& \min_{\alpha \in [0,L_w]}\max_{\beta \geq 0} \max_{\tau_h>0} \min_{\tau_g>0}
 \bigg[\frac{\beta \tau_g}{2}+ \frac{1}{m}\cdot \min_{v \in \R^m} \bigg( \frac{\beta}{2\tau_g } \pnorm{r_\vee\alpha g+\xi-v}{}^2 +  \mathsf{L}(v)\bigg)\\
&\quad -  \frac{\alpha\tau_h}{2}- \frac{1}{m}\cdot \min_{s \in \R^n}\bigg(\frac{\alpha}{2\tau_h} \pnorm{r_\vee\beta h- (r_\vee/\sqrt{m/n})s}{}^2-  \big(s^\top \mu_0 -  \bm{0}_{K}^\ast(s)\big)\bigg) \bigg].
\end{align*}
Now we shall rewrite the inner two minimization problems. The first minimization problem is easy, as
\begin{align*}
\min_{v \in \R^m} \bigg( \frac{\beta}{2\tau_g } \pnorm{r_\vee\alpha g+\xi-v}{}^2 +  \mathsf{L}(v)\bigg)= \frac{\pnorm{r_\vee\alpha g+\xi}{}^2 }{2(\tau_g/\beta+1)}
\end{align*}
by simple calculations. To handle the second minimization problem, using Lemma \ref{lem:lrt_qual}-(2), with some calculations we have
\begin{align*}
&\min_{s \in \R^n}\bigg(\frac{\alpha}{2\tau_h} \pnorm{r_\vee\beta h-(r_\vee/\sqrt{m/n})s}{}^2-  \big(s^\top \mu_0 -  \bm{0}_{K}^\ast(s)\big)\bigg)\\
& = \frac{\tau_h}{2\alpha (r_\vee/\sqrt{m/n})^2}\cdot \lrt\bigg(\frac{\alpha\beta}{\tau_h}\frac{r_\vee^2}{\sqrt{m/n}}\bigg).
\end{align*}
Now combining all these calculations, $\Phi^{\textrm{a}}_{L_w}(g,h)$  becomes
\begin{align*}
&  \min_{\alpha \in [0,L_w]} \max_{\beta\geq 0,\tau_h>0}  \bigg[ \min_{\tau_g}\bigg\{\frac{\beta \tau_g}{2}+ \frac{1}{m}\cdot \frac{\pnorm{r_\vee\alpha g+\xi}{}^2 }{2(\tau_g/\beta+1)}\bigg\}-  \frac{\alpha\tau_h}{2}- \frac{\tau_h}{2\alpha nr_\vee^2}\cdot \lrt\bigg(\frac{\alpha\beta}{\tau_h}\frac{r_\vee^2}{\sqrt{m/n}}\bigg)\bigg]\\
& = \min_{\alpha \in [0,L_w]}\max_{\beta\geq 0,\gamma>0} \bigg[ \beta\sqrt{ \pnorm{r_\vee\alpha g+\xi}{}^2/m}-\frac{\beta^2}{2}-  \frac{\alpha \gamma}{2}\cdot r_\vee^2- \frac{\gamma }{2\alpha }\cdot  \frac{1}{n} \lrt\bigg(\frac{\alpha\beta}{\gamma\sqrt{m/n}}\bigg)\bigg],
\end{align*}
where we changed $(\beta,\tau_h)$ to $(\beta,\gamma)$ with $\gamma = \tau_h/(r_\vee^2)$ in the last equality. Further changing $(\beta,\gamma)$ to $(\eta,\gamma)$ with $\eta = \beta/\gamma$, $\Phi^{\textrm{a}}_{L_w}(g,h)$  becomes
\begin{align*}
& \min_{\alpha \in [0,L_w]}\max_{\eta\geq 0,\gamma>0} \bigg[ \eta \gamma \sqrt{ \pnorm{r_\vee\alpha g+\xi}{}^2/m}-\frac{\eta^2\gamma^2 }{2}-  \frac{\alpha \gamma}{2}\cdot r_\vee^2- \frac{\gamma }{2\alpha }\cdot  \frac{1}{n} \lrt\bigg(\frac{\alpha\eta}{\sqrt{m/n}}\bigg)\bigg]\\
& = \min_{\alpha \in [0,L_w]}\max_{\eta\geq 0} \max_{\gamma>0}\bigg[\gamma\bigg\{ \eta \sqrt{ \pnorm{r_\vee\alpha g+\xi}{}^2/m}-\frac{\alpha }{2}\cdot r_\vee^2 - \frac{1}{2\alpha }\cdot  \frac{1}{n} \lrt\bigg(\frac{\alpha\eta}{\sqrt{m/n}}\bigg) \bigg\}-\eta^2\cdot \frac{\gamma^2}{2}  \bigg].
\end{align*}
 Now computing the inner most maximum with respect to $\gamma$,  $\Phi^{\textrm{a}}_{L_w}(g,h)$  equals
\begin{align*}
\frac{1}{2} \min_{\alpha \in [0,L_w]}\max_{\eta\geq 0}\bigg[   \sqrt{ \pnorm{r_\vee\alpha g+\xi}{}^2/m}-\frac{\alpha }{2\eta}\cdot r_\vee^2 - \frac{1}{2\alpha \eta}\cdot  \frac{1}{n} \lrt\bigg(\frac{\alpha\eta}{\sqrt{m/n}}\bigg)\bigg]_+^2.
\end{align*}
The claim now follows by Lemma \ref{lem:random_error_replace_D} below.
\end{proof}

\begin{lemma}\label{lem:random_error_replace_D}
On the event $E(C_0)$ (see Definition \ref{def:xi_good}), 
\begin{align*}
&\min_{\alpha \in [0,L_w]}\max_{\eta\geq 0}\bigg[   \sqrt{ \pnorm{r_\vee\alpha g+\xi}{}^2/m}-\frac{\alpha }{2\eta}\cdot r_\vee^2 - \frac{1}{2\alpha \eta}\cdot  \frac{1}{n} \lrt\bigg(\frac{\alpha\eta}{\sqrt{m/n}}\bigg)\bigg]_+^2 \nonumber\\
& = \min_{\alpha \in [0,L_w]}\max_{\eta\geq 0}\mathsf{D}(\alpha,\eta)_+^2 +\smallopx(r_\vee^2).
\end{align*}
Here the probability estimate in $\smallopx$ is uniform with respect to problem instances for a fixed choice of $C_0, L_w$. 
\end{lemma}
\begin{proof}
Note that for generic functions $F_1(\alpha),F_2(\alpha)$ and $G(\alpha,\eta)$,
\begin{align*}
&\bigabs{\min_\alpha \max_\eta \big(F_1(\alpha)+G(\alpha,\eta)\big)_+^2-\min_\alpha \max_\eta \big(F_2(\alpha)+G(\alpha,\eta)\big)_+^2 }\\
&\leq \max_\alpha \max_\eta \bigabs{ \big(F_1(\alpha)+G(\alpha,\eta)\big)_+^2- \big(F_2(\alpha)+G(\alpha,\eta)\big)_+^2 }\\
&\leq 2 \bigg[\max_{\alpha}\max_\eta \big(F_1(\alpha)+G(\alpha,\eta)\big)_+\\
&\qquad\qquad \vee \max_{\alpha}\max_\eta \big(F_2(\alpha)+G(\alpha,\eta)\big)_+\bigg] \max_\alpha \bigabs{F_1(\alpha)-F_2(\alpha)}.
\end{align*}
So the difference between the LHS and the first term of the RHS in the equality stated in the lemma is at most the product of
\begin{align*}
(I)\equiv 2\max_{\alpha \in [0,L_w]} \biggabs{  \sqrt{ \pnorm{r_\vee\alpha g+\xi}{}^2/m}- \sqrt{ r_\vee^2 \alpha^2+\sigma_m^2} }
\end{align*}
and
\begin{align*}
(II)& = \max_{\alpha \in [0,L_w]}\max_{\eta\geq 0}\bigg[   \sqrt{ \pnorm{r_\vee\alpha g+\xi}{}^2/m}-\frac{\alpha }{2\eta}\cdot r_\vee^2 - \frac{1}{2\alpha \eta}\cdot  \frac{1}{n} \lrt\bigg(\frac{\alpha\eta}{\sqrt{m/n}}\bigg)\bigg]_+\\
&\qquad\qquad \vee \max_{\alpha \in [0,L_w]}\max_{\eta\geq 0}\bigg[   \sqrt{ r_\vee^2\alpha^2+\sigma_m^2}-\frac{\alpha }{2\eta}\cdot r_\vee^2 - \frac{1}{2\alpha \eta}\cdot  \frac{1}{n} \lrt\bigg(\frac{\alpha\eta}{\sqrt{m/n}}\bigg)\bigg]_+.
\end{align*}
For $(I)$, note that
\begin{align*}
(I)& \leq 2\sigma_m^{-1} \max_{\alpha \in [0,L_w]}\bigabs{m^{-1}\pnorm{r_\vee \alpha g+\xi}{}^2- (r_\vee^2\alpha^2+\sigma_m^2)} \\
&\lesssim \sigma_m^{-1} \max_{\alpha \in [0,L_w]}\bigg[r_\vee^2\alpha^2 \biggabs{\frac{\pnorm{g}{}^2}{m}-1}+r_\vee \alpha\cdot \frac{1}{m}  \abs{\iprod{g}{\xi}}\bigg] = \bigopx \big(r_\vee/m^{1/2}\big).
\end{align*}
For $(II)$, we only handle the first term therein, as the second term is actually simpler. To this end, note that
\begin{align*}
& \max_{\alpha \in [0,L_w]}\max_{\eta\geq 0}\bigg[   \sqrt{ \pnorm{r_\vee\alpha g+\xi}{}^2/m}-\frac{\alpha }{2\eta}\cdot r_\vee^2 - \frac{1}{2\alpha \eta}\cdot  \frac{1}{n} \lrt\bigg(\frac{\alpha\eta}{\sqrt{m/n}}\bigg)\bigg]_+\\
&\leq \max_{\alpha \in [0,L_w]}  \sqrt{ \pnorm{r_\vee\alpha g+\xi}{}^2/m}+ \max_{\alpha \in [0,L_w]}\min_{\eta\geq 0}\bigg[\frac{\alpha }{2\eta}\cdot r_\vee^2 + \frac{1}{2\alpha \eta}\cdot  \frac{1}{n} \lrt\bigg(\frac{\alpha\eta}{\sqrt{m/n}}\bigg)\bigg]\\
&\equiv (II)_1+(II)_2. 
\end{align*}
It is easy to see $(II)_1=\bigopx(1)$. For $(II)_2$, by choosing $\eta=1$,  we have
\begin{align*}
(II)_2&\lesssim  \bigopx(1)+ \frac{1}{n}\cdot \max_{\alpha \in [0,L_w]}  \frac{1}{\alpha} \lrt\bigg(\frac{\alpha}{\sqrt{m/n}}\bigg)\\
&\leq \bigopx(1)+ \frac{1}{n}\cdot  \frac{1}{L_w} \lrt\bigg(\frac{L_w}{\sqrt{m/n}}\bigg) \quad \hbox{(by Lemma \ref{lem:lrt_qual}-(3))}\\
&\lesssim_{L_w} \bigopx(1)+\frac{1}{n}\cdot \frac{L_w^2}{r_{n,\xi}^2+\sigma_m^2}\cdot \lrt\big(\omega_{m/n}(r_{n,\xi},\sigma_m)\big).
\end{align*} 
As $\lrt(\cdot)\geq 0$, 
\begin{align*}
&\lrt\big(\omega_{m/n}(r_{n,\xi},\sigma_m)\big) = \bigopx\Big(\E^\xi \lrt\big(\omega_{m/n}(r_{n,\xi},\sigma_m)\big)\Big)\\
& = \bigopx\Big(\E^\xi \err\big(\omega_{m/n}(r_{n,\xi},\sigma_m)\big)\Big)+\bigopx(n) \quad \hbox{(using (\ref{ineq:R2_xi}))}\\
& = \bigopx(nr_{n,\xi}^2+n) = \bigopx(n).
\end{align*}
Combining the above two displays, we have $(II) = \bigopx(1)$. The claim follows by noting that $r_\vee/m^{1/2}\ll r_\vee^2$. 
\end{proof}

\subsection{Conditional localization and de-stochastization}

\begin{proposition}\label{prop:local_destoc}
Fix a sequence $M_n\uparrow \infty$ such that $m/\mathfrak{L}_n\geq M_n$. The following hold on the event $E(C_0)$ (see Definition \ref{def:xi_good}).
\begin{enumerate}
	\item It holds that
	\begin{align*}
	\Prob^\xi\bigg(\min_{\alpha \in [0,L_w]} \max_{\eta\geq 0}\mathsf{D}(\alpha,\eta)_+^2=\min_{\alpha \in [0,L_w]} \max_{\eta\geq 0}\mathsf{D}(\alpha,\eta)^2\bigg) = 1-\smallopx(1).
	\end{align*}
     The terms in the above probability equal 
	\begin{align}\label{eqn:local_destoc_1}
	&\frac{1}{r_\vee^2+\sigma_m^2 }\bigg[\sigma_m^2-\frac{1+\smallopx(1)}{2n} \nonumber\\
	&\qquad\qquad \times \Big(\E^\xi \lrt\big(\omega_{m/n}(r_{n,\xi},\sigma_m)\big)-\E^\xi\err\big(\omega_{m/n}(r_{n,\xi},\sigma_m)\big) \Big) \bigg]_+^2.
	\end{align}
	\item 
    It holds that
	\begin{align*}
	\min_{\alpha \in [0,L_w]} \max_{\eta\geq 0}\overline{\mathsf{D}}(\alpha,\eta)_+^2&=\min_{\alpha \in [0,L_w]} \max_{\eta\geq 0}\overline{\mathsf{D}}(\alpha,\eta)^2.
	\end{align*}
	These terms equal (\ref{eqn:local_destoc_1}) with $1+\smallopx(1)$ replaced by $1+\smallo(1)\bm{1}_{r_\vee>r_{n,\xi}}$ therein, where $\smallo(1)$ depends only on the choice of $\{M_n\}$.
	\item It holds that 
	\begin{align*}
	\min_{\alpha \in [0,L_w]} \max_{\eta\geq 0}\mathsf{D}(\alpha,\eta)^2= \min_{\alpha \in [0,L_w]} \max_{\eta\geq 0}\overline{\mathsf{D}}(\alpha,\eta)^2+\smallopx(\bar{r}_\vee^2).
	\end{align*}
\end{enumerate}
The probability estimates in the $\smallopx(1)$ terms are uniform with respect to problem instances for a fixed choice of $C_0, L_w$ and $\{M_n\}$. 
\end{proposition}
\begin{proof}
\noindent (1a). First consider the case $r_n^2 \geq 2 \mathfrak{u}_n \mathfrak{L}_n/m $. On $E(C_0)$, we have $r_\vee =r_{n,\xi}$. We only need to prove that
\begin{align}\label{ineq:local_destoc_1}
\Prob^\xi\Big(\min_{\alpha \in [0,L_w]} \max_{\eta\geq 0}\mathsf{D}(\alpha,\eta)\leq 0\Big) = \smallopx(1).
\end{align}
By (\ref{eq:dD_dalpha})-(\ref{eq:dD_deta}), an inner saddle point $(\alpha_\ast,\eta_\ast)$ to the above min-max problem satisfies the first order optimality condition
\begin{align}\label{ineq:local_destoc_0}
\alpha_\ast \eta_\ast = \sqrt{\alpha_\ast^2 r_\vee^2+\sigma_m^2},\quad \alpha_\ast^2 n r_\vee^2 = \err\big(\alpha_\ast\eta_\ast/\sqrt{m/n}\big).
\end{align}
By the stability estimate in Proposition \ref{prop:est_err_stability},
\begin{align}\label{ineq:local_destoc_2}
\alpha_\ast^2 n r_\vee^2 = \err\big(\omega_{m/n}(\alpha_\ast r_\vee,\sigma_m)\big)\leq \frac{\alpha_\ast^2 r_\vee^2+\sigma_m^2}{r_\vee^2+\sigma_m^2}\cdot  \err\big(\omega_{m/n}(r_\vee,\sigma_m)\big).
\end{align}
Using the variance bound in Proposition \ref{prop:est_err_variance}, we have
\begin{align}\label{ineq:local_destoc_4}
&\err\big(\omega_{m/n}(r_\vee,\sigma_m)\big) \nonumber\\
&= \E^\xi\err\big(\omega_{m/n}(r_\vee,\sigma_m)\big) + \bigopx\bigg(\omega_{m/n}(r_\vee,\sigma_m)\cdot \E^{1/2,\xi}\err\big(\omega_{m/n}(r_\vee,\sigma_m)\big)\bigg)\nonumber\\
& = nr_\vee^2\cdot \bigg[1 + \bigopx\bigg( \omega_{m/n}(r_\vee,\sigma_m)\Big/\sqrt{nr_\vee^2} \bigg)\bigg] \pequivx nr_\vee^2.
\end{align}
The last equivalence in probability uses the fact $mr_\vee^2\geq \mathfrak{L}_n\gg 1$. Combined with (\ref{ineq:local_destoc_2}), we obtain
\begin{align*}
\alpha_\ast^2 n r_\vee^2 \leq \frac{\alpha_\ast^2 r_\vee^2+\sigma_m^2}{r_\vee^2+\sigma_m^2}\cdot nr_\vee^2\big(1+\smallopx(1)\big).
\end{align*} 
On the event $E_{+\epsilon}\equiv \{\alpha_\ast> 1+\epsilon\}$, the above display implies
\begin{align*}
2\epsilon \sigma_m^2\leq (\alpha_\ast^2-1)\sigma_m^2=\smallopx(1).
\end{align*}
As $\sigma_m^2\geq \sigma^2/2>0$ on $E(C_0)$, the above inequality must be violated for $n$ large with $\Prob^\xi$-high probability, so $\lim_n\Prob^\xi(E_{+\epsilon})= 0$ for a fixed $\epsilon>0$. Similarly, the event $E_{-\epsilon}\equiv \{\alpha_\ast< 1-\epsilon\}$ satisfies $\lim_n\Prob^\xi(E_{-\epsilon})=0$ for fixed $\epsilon>0$, upon using the reserved version of the inequality (\ref{ineq:local_destoc_2}) that holds for $\alpha_\ast<1$. Consequently, we have proved the $\Prob^\xi$-asymptotically probability $1$ existence of an inner saddle point $(\alpha_\ast,\eta_\ast)$ with
\begin{align}\label{ineq:local_destoc_3}
\alpha_\ast\pequivx 1,\quad \eta_\ast \pequivx \sqrt{r_\vee^2+\sigma_m^2},\quad \hbox{when $r_{n}\geq 2\mathfrak{u}_n\mathfrak{L}_n/m$ and $\xi \in E(C_0)$}.
\end{align}
Now using (\ref{ineq:local_destoc_0}) and the above display (\ref{ineq:local_destoc_3}), we may calculate
\begin{align}\label{ineq:local_destoc_5}
&\min_{\alpha \in [0,L_w]} \max_{\eta\geq 0}\mathsf{D}(\alpha,\eta) = \mathsf{D}(\alpha_\ast,\eta_\ast) \nonumber\\
& = \sqrt{\alpha_\ast^2r_\vee^2+\sigma_m^2}-\frac{\alpha_\ast}{2\sqrt{\alpha_\ast^2r_\vee^2+\sigma_m^2}   }\bigg[\alpha_\ast r_\vee^2+\frac{1}{\alpha_\ast n}\lrt\big(\omega_{m/n}(\alpha_\ast r_\vee,\sigma_m)\big)\bigg]\nonumber\\
& = \frac{\sigma_m^2}{\sqrt{\alpha_\ast^2 r_\vee^2+\sigma_m^2}  }- \frac{1}{2\sqrt{\alpha_\ast^2r_\vee^2+\sigma_m^2}   }\bigg[\frac{1}{n}\lrt\big(\omega_{m/n}(\alpha_\ast r_\vee,\sigma_m)\big)-\alpha_\ast^2 r_\vee^2\bigg] \nonumber\\
& = \frac{1}{ \sqrt{\alpha_\ast^2 r_\vee^2+\sigma_m^2}  }\bigg[\sigma_m^2-\frac{1}{2n}\Big(\lrt\big(\omega_{m/n}(\alpha_\ast r_\vee,\sigma_m)\big)-\err\big(\omega_{m/n}(\alpha_\ast r_\vee,\sigma_m)\big) \Big) \bigg].
\end{align}
Using the stability estimates in Propositions \ref{prop:est_err_stability} and \ref{prop:lrt_stability}, and the proven fact that $\alpha_\ast \pequivx 1$, we have  $H\big(\omega_{m/n}(\alpha_\ast r_\vee,\sigma_m)\big)\pequivx H\big(\omega_{m/n}(r_\vee,\sigma_m)\big)$ for $H \in \{\err,\lrt\}$. Now using the variance bounds in Proposition \ref{prop:est_err_variance}-(2) and Proposition \ref{prop:lrt_prop}-(3), and the fact that $\lrt(\cdot)\geq \err(\cdot)$ as proved in Proposition \ref{prop:lrt_prop}-(2), we have for $H \in \{\err,\lrt\}$,
\begin{align*}
&H\big(\omega_{m/n}(r_\vee,\sigma_m)\big) \\
&= \E^\xi H\big(\omega_{m/n}(r_\vee,\sigma_m)\big) + \bigopx\bigg(\omega_{m/n}(r_\vee,\sigma_m)\cdot \E^{1/2,\xi}\err\big(\omega_{m/n}(r_\vee,\sigma_m)\big)\bigg)\nonumber\\
& =\E^\xi H\big(\omega_{m/n}(r_\vee,\sigma_m)\big) \cdot \bigg[1 + \bigopx\bigg( \omega_{m/n}(r_\vee,\sigma_m)\Big/\sqrt{nr_\vee^2} \bigg)\bigg] \\
&\pequivx \E^\xi H \big(\omega_{m/n}(r_\vee,\sigma_m)\big).
\end{align*}
Combining these calculations and using (\ref{ineq:R2_xi}), we have
\begin{align*}
&\min_{\alpha \in [0,L_w]} \max_{\eta\geq 0}\mathsf{D}(\alpha,\eta) = \mathsf{D}(\alpha_\ast,\eta_\ast)\\
& = \frac{1}{ \sqrt{r_\vee^2+\sigma_m^2}  }\bigg[\sigma_m^2-\frac{1+\smallopx(1)}{2n}\Big(\E^\xi \lrt\big(\omega_{m/n}(r_\vee,\sigma_m)\big)-\E^\xi\err\big(\omega_{m/n}(r_\vee,\sigma_m)\big) \Big) \bigg]\\
&\geq C_0^{-1}(1+\smallopx(1))\sigma_m^2\Big/ \sqrt{r_\vee^2+\sigma_m^2},
\end{align*}
proving the claim in (\ref{ineq:local_destoc_1}). 

\noindent (1b).  Next consider the case $r_n^2<2\mathfrak{u}_n\mathfrak{L}_n/m$. First we show $\alpha_\ast>0$. As $\sigma\mapsto \lrt(\sigma)/\sigma$ is non-decreasing by Lemma \ref{lem:lrt_qual}-(3), the map
\begin{align*}
\eta \mapsto  \frac{\alpha }{2\eta}\cdot r_\vee^2 + \frac{1}{2\alpha \eta}\cdot  \frac{1}{n} \lrt\bigg(\frac{\alpha\eta}{\sqrt{m/n}}\bigg)
\end{align*}
cannot be minimized at $0$. So $\eta_\ast >0$. If $\alpha_\ast =0$, the first-order optimality condition for $\alpha$ becomes
\begin{align*}
\frac{\d \mathsf{D}}{\d \alpha} (\alpha_\ast,\eta_\ast) \geq 0 \quad \Rightarrow \quad  \eta_\ast\alpha_\ast  \geq \sqrt{r_\vee^2\alpha_\ast^2+\sigma_m^2}.
\end{align*}
This leads to a contradiction as $\sigma_m^2>0$ on $E(C_0)$. So $\alpha_\ast >0$. 

Now repeating (\ref{ineq:local_destoc_0}), (\ref{ineq:local_destoc_2}) and (\ref{ineq:local_destoc_4}) (where the second equality in (\ref{ineq:local_destoc_4}) becomes $\leq$ due to Lemma \ref{lem:r_vee_fixed_pt_eqn}), we have established a $\Prob^\xi$-high probability existence result of an inner saddle point $(\alpha_\ast,\eta_\ast)$ satisfying (\ref{ineq:local_destoc_0}) with $(\alpha_\ast-1)_+\pequiv 0$. 

The calculations in (\ref{ineq:local_destoc_5}) remain valid. Under $\mathfrak{u}_n\mathfrak{L}_n/m\to 0$ (recall here $\mathfrak{u}_n=M_n^{1/2})$, we have $\omega_{m/n}(\alpha_\ast r_\vee,\sigma_m)\simeq \omega_{m/n}(\alpha_\ast r_{n,\xi},\sigma_m) \simeq \omega_{m/n}(0,\sigma_m)$, so by the stability estimates for $H \in \{\err,\lrt\}$ in Propositions \ref{prop:est_err_stability} and \ref{prop:lrt_stability},
\begin{align*}
&H\big(\omega_{m/n}(\alpha_\ast r_\vee,\sigma_m)\big)\pequivx \E^\xi H\big(\omega_{m/n}(\alpha_\ast r_\vee,\sigma_m)\big)\simeq \E^\xi H\big(\omega_{m/n}(\alpha_\ast r_{n,\xi},\sigma_m)\big).
\end{align*}
Consequently, 
\begin{align}\label{ineq:local_destoc_6}
&\min_{\alpha \in [0,L_w]} \max_{\eta\geq 0}\mathsf{D}(\alpha,\eta) = \mathsf{D}(\alpha_\ast,\eta_\ast) = \frac{1}{ \sqrt{r_\vee^2+\sigma_m^2}  } \nonumber\\
&\quad \times\bigg[\sigma_m^2-\frac{1+\smallopx(1)}{2n}\Big(\E \lrt\big(\omega_{m/n}(r_{n,\xi},\sigma_m)\big)-\E\err\big(\omega_{m/n}(r_{n,\xi},\sigma_m)\big) \Big) \bigg].
\end{align}
Now invoking (\ref{ineq:R2_xi}) to conclude. 

\noindent (2). The proof is essentially a deterministic version (conditional on $\xi \in E(C_0)$) of (1). We only sketch some key steps. First, the optimality condition for an inner saddle point $(\bar{\alpha},\bar{\eta})$ is  
\begin{align}\label{ineq:local_destoc_7}
\bar{\alpha} \bar{\eta} = \sqrt{\bar{\alpha}^2 r_\vee^2+\sigma_m^2},\quad \bar{\alpha}^2 n r_\vee^2 = \E  \err\big(\bar{\alpha}\bar{\eta}/\sqrt{m/n}\big).
\end{align}
So using the fixed point equation (\ref{def:rn_xi}), we may solve
\begin{align*}
\bar{\alpha}=r_{n,\xi}/r_\vee, \quad \bar{\eta} = \sqrt{r_\vee^2+\sigma_m^2 r_\vee^2/r_{n,\xi}^2}.
\end{align*}
Using the same calculations as in (\ref{ineq:local_destoc_5}) we conclude that 
\begin{align*}
&\min_{\alpha \in [0,L_w]} \max_{\eta\geq 0}\overline{\mathsf{D}}(\alpha,\eta)_+^2 
=\min_{\alpha \in [0,L_w]} \max_{\eta\geq 0}\overline{\mathsf{D}}(\alpha,\eta)^2= \mathsf{D}(\bar{\alpha},\bar{\eta})^2\\
& = \frac{1}{ r_\vee^2+\sigma_m^2 }\bigg[\sigma_m^2-\frac{1}{2n}\Big(\E^\xi \lrt\big(\omega_{m/n}(r_\vee,\sigma_m)\big)-\E^\xi\err\big(\omega_{m/n}(r_\vee,\sigma_m)\big) \Big) \bigg]_+^2,
\end{align*}
when $r_n^2\geq 2\mathfrak{u}_n\mathfrak{L}_n/m$. If $r_n^2<2\mathfrak{u}_n\mathfrak{L}_n/m$, the factor $1/2n$ in front of the big bracket in the above display is replaced by $(1+\smallo(1))/2n$ using the same replacements as done in (\ref{ineq:local_destoc_6}).

\noindent (3). 
Note that the proof in (1a)-(1b), in particular the first-order optimality conditions (\ref{ineq:local_destoc_0}) and (\ref{ineq:local_destoc_7}) yield that $\alpha_\ast\eta_\ast \vee \bar{\alpha}\bar{\eta} =\bigo(1)$. So by (1) we may find some large $C_1=C_1(L_w,C_0)>0$ such that for $n$ large enough, it holds with $\Prob^\xi$-probability at least $1-\epsilon$ that
\begin{align*}
\min_{\alpha \in [0,L_w]} \max_{\eta\geq 0}\mathsf{D}(\alpha,\eta)_+^2 = \min_{\alpha \in [0,L_w]} \max_{\substack{\eta\geq 0,\\\alpha\eta\leq C_1 }} \mathsf{D}(\alpha,\eta)^2,
\end{align*}
and the same deterministic equality holds for $\overline{\mathsf{D}}$. 
On the other hand, with $\Delta(\alpha,\eta)\equiv \mathsf{D}(\alpha,\eta)-\overline{\mathsf{D}}(\alpha,\eta)$, we have
\begin{align*}
&\biggabs{\min_{\alpha \in [0,L_w]} \max_{\substack{\eta\geq 0,\\\alpha\eta\leq C_1 }} \mathsf{D}(\alpha,\eta)_+^2 - \min_{\alpha \in [0,L_w]} \max_{\substack{\eta\geq 0,\\\alpha\eta\leq C_1 }}   \overline{\mathsf{D}}(\alpha,\eta)_+^2}\\
&\leq 2\bigg[\max_{\alpha \in [0,L_w]} \max_{\substack{\eta\geq 0,\\\alpha\eta\leq C_1 }} \Big({\mathsf{D}}(\alpha,\eta)_+ \vee \overline{\mathsf{D}}(\alpha,\eta)_+\Big)\bigg]\cdot \max_{\alpha \in [0,L_w]} \max_{\substack{\eta\geq 0,\\\alpha\eta\leq C_1 }} \abs{\Delta(\alpha,\eta)}.
\end{align*}
For the first term, we have
\begin{align*}
&\max_{\alpha \in [0,L_w]} \max_{\substack{\eta\geq 0,\\\alpha\eta\leq C_1 }} \Big({\mathsf{D}}(\alpha,\eta)_+ \vee \overline{\mathsf{D}}(\alpha,\eta)_+\Big) \leq \max_{\alpha \in [0,L_w]} \Big({\mathsf{D}}(\alpha,1)_+\vee\overline{\mathsf{D}}(\alpha,1)_+\Big)\\
&\lesssim 1+ n^{-1} \sup_{w\leq L_w} w^{-1} \Big(\lrt(w/\sqrt{m/n})+\E \lrt(w/\sqrt{m/n})\Big)\\
&\lesssim 1+ n^{-1}\Big(\lrt(L_w/\sqrt{m/n})+\E \lrt(L_w/\sqrt{m/n})\Big)\quad \hbox{(using Lemma \ref{lem:lrt_qual}-(3))}\\
& \lesssim 1 + n^{-1} \bigopx\Big( \E^\xi \lrt\big(\omega_{m/n}(r_{n,\xi},\sigma_m)\big)\Big)\quad \hbox{(using Proposition \ref{prop:lrt_stability})}\\
& \leq 1+ n^{-1} \bigopx\bigg( \E^\xi \err\big(\omega_{m/n}(r_{n,\xi},\sigma_m)\big) + 2n\sigma_m^2\bigg) \quad \hbox{(using (\ref{ineq:R2_xi}))}\\
&\asymp 1+ n^{-1}\cdot \bigop( nr_{n,\xi}^2) = \bigopx(1).
\end{align*}
Next we handle the second term by the uniform concentration inequality proved in Proposition \ref{prop:est_err_lrt_sup}. To do so, note that
\begin{align*}
\max_{\alpha \in [0,L_w]} \max_{\substack{\eta\geq 0,\\\alpha\eta\leq C_1}}\abs{\Delta(\alpha,\eta)}&\lesssim \bigg[\sup_{w \in [0,n^{-2}]}+\sup_{w \in (n^{-2},C_1]} \bigg]\\
&\qquad \frac{1}{nw}\biggabs{\lrt\bigg(\frac{w}{\sqrt{m/n}} \bigg) -\E \lrt\bigg(\frac{w}{\sqrt{m/n}} \bigg)}\equiv (I)+(II). 
\end{align*}
For the first term $(I)$, using a simple bound in Proposition \ref{prop:lrt_prop}-(1), we have
\begin{align*}
(I)&\lesssim \sup_{w \in [0,n^{-2}]} w\cdot \frac{n}{m} \leq \frac{1}{m}\ll r_\vee^2.
\end{align*}
For the second term $(II)$, using Proposition \ref{prop:est_err_lrt_sup}, with 
\begin{align*}
\mathfrak{L}_n'\equiv \mathfrak{u}_n^{1/2} \log \Big[(1+\delta_{T_K(\mu_0)})\cdot  \log_+(C_1 n^2)\Big],
\end{align*}
$(II)$ can be bounded with $\Prob^\xi$-high probability by
\begin{align*}
(II)&\lesssim n^{-1}\sup_{w \in (n^{-2},C_1]}w^{-1}\bigg[\frac{w}{\sqrt{m/n}}\cdot \E^{1/2} \err\bigg(\frac{w}{\sqrt{m/n}}\bigg)\cdot \sqrt{\mathfrak{L}_n'}+ \frac{w^2}{m/n}\cdot \mathfrak{L}_n'\bigg]\\
&\stackrel{(\ast)}{\lesssim} n^{-1}\bigg[\frac{1}{\sqrt{m/n}}\cdot \E^{1/2,\xi} \err\big(\omega_{m/n}(r_\vee,\sigma_m)\big)\cdot \sqrt{\mathfrak{L}_n'}+ \frac{\mathfrak{L}_n'}{m/n}  \bigg]\\
& \leq n^{-1}\bigg[\frac{1}{\sqrt{m/n}}\cdot n^{1/2} r_\vee\cdot \sqrt{\mathfrak{L}_n'}+ \frac{\mathfrak{L}_n'}{m/n}  \bigg]\stackrel{(\ast\ast)}{\lesssim}  r_\vee \sqrt{\mathfrak{L}_n'/m} \stackrel{(\ast\ast)}{\ll} r_\vee^2. 
\end{align*} 
In $(\ast)$ we used the stability estimate in Proposition \ref{prop:est_err_stability} and $r_\vee\leq C_0$, while in $(\ast\ast)$ we used (i) $\mathfrak{L}_n'\asymp \mathfrak{u}_n^{1/2}\mathfrak{L}_n$, and (ii) $r_\vee^2\geq \mathfrak{u}_n\mathfrak{L}_n/m\gg \mathfrak{L}_n'/m$. Combining all the above displays proves the claim by noting $r_\vee \asymp \bar{r}_\vee$ on $E(C_0)$.
\end{proof}

\subsection{Conditional gap analysis}

\begin{proposition}\label{prop:gap_det_eps}
Fix a sequence $M_n\uparrow \infty$ such that $m/\mathfrak{L}_n\geq M_n$. The following holds on the event $E(C_0)$ (see Definition \ref{def:xi_good}): For $\epsilon \in (0,1)$, with
\begin{align*}
\Gamma_\epsilon\equiv 
\begin{cases}
[0,1-\epsilon]\cup [1+\epsilon,L_w], & r_n^2\geq 2\mathfrak{u}_n\mathfrak{L}_n/m;\\
 [1+\epsilon,L_w], & r_n^2< 2\mathfrak{u}_n\mathfrak{L}_n/m,
\end{cases}
\end{align*}
we have
\begin{align*}
\min_{\alpha \in  [0,L_w]} \max_{\eta\geq 0}{\mathsf{D}}(\alpha,\eta)_+^2\leq \min_{\substack{\alpha \in  \Gamma_\epsilon} } \max_{\eta\geq 0 }{\mathsf{D}}(\alpha,\eta)_+^2-c_0(1+\smallopx(1))\cdot \epsilon \bar{r}_\vee^2
\end{align*}
holds for some $c_0=c_0(C_0,L_w,\sigma)>0$. The probability estimate in the $\smallopx(1)$ term is uniform with respect to problem instances for a fixed choice of $C_0, L_w$ and $\{M_n\}$. 
\end{proposition}
\begin{proof}
Let $(\alpha_\ast,\eta_\ast)$ be the solution to the first order optimality condition (\ref{ineq:local_destoc_0}), and  $\alpha_{\ast,\epsilon} \in  [0,1-\epsilon]\cup [1+\epsilon,L_w]$ be a minimizer for the min-max problem $\min_{\alpha \in  [0,1-\epsilon]\cup [1+\epsilon,L_w]} \max_{\eta\geq 0}\mathsf{D}(\alpha,\eta)_+^2$.

\noindent (\textbf{Step 1}). In this step we will estimate the derivative $\frac{\d {\mathsf{D}}}{\d \alpha}$ at $(\cdot,{\eta}_\ast)$. To this end, by Lemma \ref{lem:monotone_est_err}-(2), 
\begin{align*}
J(\alpha)\equiv \frac{\d {\mathsf{D}}}{\d \alpha}(\alpha,{\eta}_\ast)
& =  \frac{\alpha r_\vee^2 }{ \sqrt{r_\vee^2\alpha^2+\sigma_m^2 } }-\frac{1}{2{\eta}_\ast}\cdot r_\vee^2  - \frac{1}{2\alpha^2 {\eta}_\ast}\cdot \frac{ \err\big(\alpha {\eta}_\ast/\sqrt{m/n}\big) }{n}
\end{align*}
is a strictly increasing function of $\alpha$ with $J(\alpha_\ast)=0$. With $\zeta(\alpha)\equiv \alpha{\eta}_\ast/\sqrt{m/n}$, 
\begin{align*}
J'(\alpha)/r_\vee^2 
& = \frac{\sigma_m^2}{(r_\vee^2\alpha^2+\sigma_m^2)^{3/2} }- \frac{{\eta}_\ast}{ 2n r_\vee^2(m/n)}\cdot \frac{\d}{\d \zeta}\bigg(\frac{  \err\big(\zeta\big) }{\zeta^2 }\bigg)\cdot \frac{\d \zeta}{\d \alpha}.
\end{align*}
Using Lemma \ref{lem:monotone_est_err}-(2) which entails $\frac{\d}{\d \zeta}(\err(\zeta)/\zeta^2)\leq 0$ and the easy fact that $\d\zeta/\d \alpha>0$, we have
\begin{align}\label{ineq:gap_det_2}
J'(\alpha)/r_\vee^2\geq  \frac{\sigma_m^2}{(r_\vee^2\alpha^2+\sigma_m^2)^{3/2} }.
\end{align}

\noindent (\textbf{Step 2}). 
Using (\ref{ineq:gap_det_2}), for some $c_1=c_1(C_0,\sigma,L_w)>0$ we have
\begin{align}\label{ineq:gap_det_3}
{\mathsf{D}}({\alpha}_{\ast,\epsilon},{\eta}_\ast)-{\mathsf{D}}(\alpha_\ast,{\eta}_\ast)\geq c_1\cdot \epsilon r_\vee^2.
\end{align}
 By Proposition \ref{prop:local_destoc}-(1), there exists $c_2=c_2(C_0,\sigma,L_w)>0$ such that 
\begin{align}\label{ineq:gap_det_4}
{\mathsf{D}}({\alpha}_{\ast,\epsilon},{\eta}_\ast)\geq{\mathsf{D}}(\alpha_\ast,{\eta}_\ast)\geq  c_2(1+\smallopx(1)). 
\end{align}
Combining (\ref{ineq:gap_det_3}) and (\ref{ineq:gap_det_4}), for some $c_3=c_3(C_0,\sigma,L_w)>0$ we have
\begin{align*}
&{\mathsf{D}}({\alpha}_{\ast,\epsilon},{\eta}_\ast)^2-{\mathsf{D}}(\alpha_\ast,{\eta}_\ast)^2\\
&=\big({\mathsf{D}}({\alpha}_{\ast,\epsilon},{\eta}_\ast)+{\mathsf{D}}(\alpha_\ast,{\eta}_\ast)\big) \big({\mathsf{D}}({\alpha}_{\ast,\epsilon},{\eta}_\ast)-{\mathsf{D}}(\alpha_\ast,{\eta}_\ast\big)\big) \\
&\geq  c_3 (1+\smallopx(1))\cdot \epsilon r_\vee^2.
\end{align*}
Using Proposition \ref{prop:local_destoc}-(1), both ${\mathsf{D}}({\alpha}_{\ast,\epsilon},{\eta}_\ast)^2$, ${\mathsf{D}}(\alpha_\ast,{\eta}_\ast)^2$ can be replaced with ${\mathsf{D}}({\alpha}_{\ast,\epsilon},{\eta}_\ast)_+^2$, $ {\mathsf{D}}(\alpha_\ast,{\eta}_\ast)_+^2$ with $\Prob^\xi$-high probability. Consequently, 
\begin{align*}
\min_{\alpha \in  [0,L_w]} \max_{\eta\geq 0}{\mathsf{D}}(\alpha,\eta)_+^2&= {\mathsf{D}}(\alpha_\ast,\eta_\ast)_+^2 \leq {\mathsf{D}}(\alpha_{\ast,\epsilon},\eta_\ast)_+^2-c_3(1+\smallopx(1))\cdot \epsilon r_\vee^2 \\
&\leq \max_{\eta\geq 0} {\mathsf{D}}(\alpha_{\ast,\epsilon},\eta)_+^2-c_3(1+\smallopx(1))\cdot \epsilon r_\vee^2 \\
& = \min_{\alpha \in \Gamma_\epsilon} \max_{\eta\geq 0} {\mathsf{D}}(\alpha,\eta)_+^2-c_3(1+\smallopx(1))\cdot \epsilon r_\vee^2.
\end{align*}
The proof is complete by noting $r_\vee \asymp \bar{r}_\vee$ on $E(C_0)$. 
\end{proof}

\subsection{Proof of Theorem \ref{thm:risk_asymp}}

	Without loss of generality, we assume that $\hat{\mu}(\sigma)$ satisfies (\ref{def:approximate_minimizer}) with $\smallop(r_n^2)$ therein replaced by $\delta_n r_n^2$ for some $\delta_n \downarrow 0$.
	
	\noindent (\textbf{Step 1}). 
		Fix $\epsilon \in (0,1)$ small and $L_w>1$ large. Let 	$S_w\equiv S_{w,\epsilon}\equiv \{w: \pnorm{w}{}\in \Gamma_\epsilon\}$, where $\Gamma_\epsilon$ is defined in Proposition \ref{prop:gap_det_eps}. Let $w_{\ast}^{\textrm{p}},w_{\ast,L_w}^{\textrm{p}}$ be $(\delta_n r_n^2)$-optimizers for $\Phi^{\textrm{p}}(G), \Phi^{\textrm{p}}_{L_w}(G)$ respectively. We will establish that on the event $E(C_0)$,
	\begin{align}\label{ineq:sensitivity_anal_0}
	\lim_n \Prob^\xi\big( w_{\ast}^{\textrm{p}} \in S_w\big)=0,\quad \forall \epsilon \in (0,1) \hbox{ small enough}.
	\end{align}
	 By Lemma \ref{lem:localization_w}, we only need to prove that on $E(C_0)$,
	\begin{align}\label{ineq:sensitivity_anal_1}
	\lim_n \Prob^\xi\big( w_{\ast,L_w}^{\textrm{p}} \in S_w\big)=0,\quad \forall \epsilon \in (0,1) \hbox{ small enough}.
	\end{align}
	We claim that for any pair of constants $\phi,\phi_{S_w}$ such that $\phi<\phi_{S_w}-\delta_n r_n^2$,
	\begin{align}\label{ineq:sensitivity_anal_2}
	\Prob^\xi\big( w_{\ast,L_w}^{\textrm{p}} \in S_w\big)\leq 2\Big(\Prob^\xi\big(\Phi^{\textrm{a}}_{S_w} (g,h)\leq \phi_{S_w}\big) +\Prob^\xi\big(\Phi^{\textrm{a}}_{L_w} (g,h)\geq \phi\big)\Big).
	\end{align}
	Fix $\epsilon'>0$ small enough. We choose $ \textbf{L}_{\epsilon'}\equiv (L_u,L_\beta,L_v)$ as in Lemma \ref{lem:apriori_localization_u_v} (with $\epsilon'$ replacing $\epsilon$ therein). For notational simplicity, for any closed set $S_w \subset [0,L_w]$, we write $\Phi^{\textrm{p}}_{L_u,L_v,S_w}(G) = \Phi^{\textrm{p}}_{\textbf{L}_{\epsilon'},S_w}(G)$ and $\Phi^{\textrm{a}}_{L_\beta, L_v, S_w}(g,h)=\Phi^{\textrm{a}}_{\textbf{L}_{\epsilon'}, S_w}(g,h)$.
	
	To see (\ref{ineq:sensitivity_anal_2}), \cite[Eqns (83) and (84)]{thrampoulidis2018precise} (which are simple consequences of the Convex Gaussian Min-Max Theorem stated in Theorem \ref{thm:CGMT}) yield that
	\begin{align}\label{ineq:sensitivity_anal_3}
	\Prob^\xi\big(\Phi^{\textrm{p}}_{\textbf{L}_{\epsilon'},S_w} (G)\leq \phi_{S_w}\big) &\leq 2\Prob^\xi\big(\Phi^{\textrm{a}}_{\textbf{L}_{\epsilon'},S_w} (g,h)\leq \phi_{S_w}\big),\nonumber \\
	\Prob^\xi\big(\Phi^{\textrm{p}}_{\textbf{L}_{\epsilon'},L_w} (G)\geq \phi\big)&\leq 2\Prob^\xi\big(\Phi^{\textrm{a}}_{\textbf{L}_{\epsilon'},L_w} (g,h)\geq \phi\big).
	\end{align}
	By Lemma \ref{lem:apriori_localization_u_v}, 
	\begin{align*}
	\Prob^\xi\big(E_1\equiv \big\{ \Phi^{\textrm{p}}_{S_w} (G)\leq \phi_{S_w}\big\}\big) &\leq 2\Prob^\xi\big(\Phi^{\textrm{a}}_{S_w} (g,h)\leq \phi_{S_w}\big)+2\epsilon',\\
	\Prob^\xi\big(E_2\equiv \big\{\Phi^{\textrm{p}}_{L_w} (G)\geq \phi\big\}\big)&\leq 2\Prob^\xi\big(\Phi^{\textrm{a}}_{L_w} (g,h)\geq \phi\big)+2\epsilon'.
	\end{align*}
	Clearly the event
	\begin{align*}
	\big\{\Phi^{\textrm{p}}_{S_w} (G)> \phi_{S_w}, \Phi^{\textrm{p}}_{L_w} (G)< \phi \big\} = E_1^c \cap E_2^c
	\end{align*}
	implies that $w_{\ast,L_w}^{\textrm{p}} \notin S_w$, so
	\begin{align*}
	&\Prob^\xi\big( w_{\ast,L_w}^{\textrm{p}} \in S_w\big)\leq \Prob^\xi(E_1)+\Prob^\xi(E_2)\\
	&\leq 2\Big(\Prob^\xi\big(\Phi^{\textrm{a}}_{S_w} (g,h)\leq \phi_{S_w}\big)+ \Prob^\xi\big(\Phi^{\textrm{a}}_{L_w} (g,h)\geq \phi\big)\Big)+4\epsilon'.
	\end{align*}
	Letting $\epsilon'\downarrow 0$ proves (\ref{ineq:sensitivity_anal_2}).

	Now by Proposition \ref{prop:adjusted_AO} and Proposition \ref{prop:local_destoc}-(3),
	\begin{align*}
	\Phi^{\textrm{a}}_{L_w} (g,h)& = \frac{1}{2}\min_{\alpha \in  [0,L_w]} \max_{\eta\geq 0}\overline{\mathsf{D}}(\alpha,\eta)_+^2+\smallopx(\bar{r}_\vee^2)\equiv \bar{\phi}+ \smallopx(\bar{r}_\vee^2).
	\end{align*}
	By Proposition \ref{prop:gap_det_eps}, for some $c_0>0$ small, 
	\begin{align*}
	\Phi^{\textrm{a}}_{S_w} (g,h)\geq \Phi^{\textrm{a}}_{L_w} (g,h)+c_0(1+\smallopx(1))\bar{r}_\vee^2 = \bar{\phi}+c_0 \cdot \bar{r}_\vee^2+ \smallopx(\bar{r}_\vee^2). 
	\end{align*}
	Consequently, by choosing $\phi\equiv \bar{\phi}+(c_0/4)\bar{r}_\vee^2$ and $\phi_{S_w}\equiv \bar{\phi}+(3c_0/4)\bar{r}_\vee^2$ in (\ref{ineq:sensitivity_anal_2}), which is a valid pair $\phi<\phi_{S_w}-\delta_n r_n^2$ for $n$ large, the two conditional probability terms in the RHS of (\ref{ineq:sensitivity_anal_2}) vanish as $n \to \infty$.  The proof of (\ref{ineq:sensitivity_anal_0}) is complete. \qed
	
	\noindent (\textbf{Step 2}). By Lemma \ref{lem:R2_deconditioning}, $\Prob(E(C_0))\to 1$ for large enough $C_0>1$ and well chosen $\{M_n^r\},\{M_n^\sigma\}$. So for $\epsilon>0$ small,
	\begin{align*}
	\Prob\big( w_{\ast}^{\textrm{p}} \in S_w\big) \leq  \E \bm{1}_{E(C_0)} \Prob^\xi\big( w_{\ast}^{\textrm{p}} \in S_w\big) + \Prob(E(C_0)^c). 
	\end{align*}
	The first term on the RHS of the above display vanishes by dominated convergence theorem. The proof is complete.  
	
	For the claim in (\ref{eqn:risk_uniform}) , we may modify (\ref{ineq:sensitivity_anal_0}) to show that on the event $E(C_0)$, 
	\begin{align*}
	\lim_n \sup_{(n,m,\mu_0,K) \in \mathscr{C}(M_n,L)} \Prob^\xi\big( w_{\ast}^{\textrm{p}} \in S_w\big)=0,\quad \forall \epsilon \in (0,1) \hbox{ small enough}.
	\end{align*}
	Then using
	\begin{align*}
	\sup_{(n,m,\mu_0,K) \in \mathscr{C}(M_n,L)}\Prob\big( w_{\ast}^{\textrm{p}} \in S_w\big) \leq  \E \bm{1}_{E(C_0)} \sup_{(n,m,\mu_0,K) \in \mathscr{C}(M_n,L)} \Prob^\xi\big( w_{\ast}^{\textrm{p}} \in S_w\big) + \Prob(E(C_0)^c)
	\end{align*}
	to conclude. Details are omitted. \qed

\section{Remaining proofs for Section \ref{section:risk_asymp}}\label{section:proof_remain_main}

\subsection{Proof of Proposition \ref{prop:exist_unique_fixed_pt}}

	\noindent (1). This follows directly from Proposition \ref{prop:unique_fixed_point_eqn}.
	
	\noindent (2). We first prove consistency. To see this, by the monotonicity of $r\mapsto \E \err(\omega_{m/n}(r))/r^2$ as in Proposition \ref{prop:unique_fixed_point_eqn}-(1), we have $\E \err(\omega_{m/n}(r))/(nr^2)\geq 1$ for $r \in [0,r_n]$, so the iterations $\{r_{n,t}\}$ is a monotone sequence: $0=r_{n,0}\leq r_{n,1}\leq \cdots \leq r_n$. Consequently $r_{n,t}\to r^\ast$ for some $r^\ast \in [0,r_n]$. Taking limit as $t\to \infty$ on both sides of (\ref{def:iteration_fixed_point}) and using continuity of $r\mapsto \E \err(r)$, we conclude that $r^\ast$ is a solution to the fixed point equation (\ref{eqn:fixed_pt_eqn}). By uniqueness this necessarily implies $r^\ast = r_n$. This proves that $\lim_{t \to \infty} r_{n,t}=r_n$. 
	
	Next we prove the announced error bound. Note that with $G(r) \equiv \sqrt{\E \err(\omega_{m/n}(r))/n}$, we have by Lemma \ref{lem:monotone_est_err}-(2), 
	\begin{align*}
	\abs{G'(r)}& = n^{-1/2}\cdot \biggabs{\frac{ \E \err'(\omega_{m/n}(r))}{2\sqrt{\E \err'(\omega_{m/n}(r))}}\cdot \frac{\d \omega_{m/n}(r)}{\d r}}\\
	&\leq \sqrt{ \frac{\E \err(\omega_{m/n}(r))}{n(\omega_{m/n}(r))^2}\cdot \frac{r^2}{(m/n)(r^2+\sigma^2)} } = \sqrt{ \frac{ \E \err(\omega_{m/n}(r)) }{nr^2}   }\cdot \frac{r^2}{(r^2+\sigma^2)}. 
	\end{align*}
	Using the monotonicity of $r\mapsto \E \err(\omega_{m/n}(r))/r^2$ again, we may estimate $G'$ on $[r_{n,t},r_n]$ by
	\begin{align*}
	\sup_{r \in [r_{n,t},r_n]} \abs{G'(r)}\leq  \sqrt{ \frac{\E \err(\omega_{m/n}(r_{n,t})) }{nr_{n,t}^2} }\cdot \frac{r_n^2}{r_n^2+\sigma^2}.
	\end{align*}
	Consequently, by the consistency proven above, for any $\rho \in \big(r_n^2/(r_n^2+\sigma^2),1\big)$, we may find large enough $T_\rho \in \N$ so that $\sup_{r \in [r_{n,t},r_n]} \abs{G'(r)}\leq \rho$ for $t\geq T_\rho$. This means for $t\geq T_\rho$,
	\begin{align*}
	\abs{r_{n,t+1}-r_n}&= \bigabs{G(r_{n,t})-G(r_n)}\leq \sup_{r \in [r_{n,t},r_n]} \abs{G'(r)} \abs{r_{n,t}-r_n}\leq \rho \abs{r_{n,t}-r_n},
	\end{align*}
	proving the claim. 
	
	\noindent (3). By Proposition \ref{prop:est_err_variance}-(1), we have
	\begin{align*}
	r_n^2  = n^{-1} \E \err\big(\omega_{m/n}(r_n)\big)\geq n^{-1}\cdot \omega_{m/n}^2 (r_n)\delta_K = \frac{r_n^2+\sigma^2}{m}\cdot \delta_K.
	\end{align*}
	Solving $r_n^2$ yields the lower bound. For the upper bound, we replace $\delta_K$ by $\delta_{T_K(\mu_0)}$ and $\geq$ by $\leq $ in the above display to conclude.
	 
	Now we examine the behavior of $r_n=r_n(\sigma)$ in the low noise limit $\sigma \downarrow 0$ when $m>\delta_{T_K(\mu_0)}$. Write $\bar{r}_n\equiv \bar{r}_n(\sigma)=r_n(\sigma)/\sigma$ and $\tau\equiv m/n$, the fixed point equation (\ref{eqn:fixed_pt_eqn}) becomes
	\begin{align*}
	\frac{\E \err\Big(\sqrt{\sigma^2\big(1+\bar{r}_n^2\big)/\tau}\Big)  }{\sigma^2(1+\bar{r}_n^2)/\tau} = \frac{m \bar{r}_n^2}{1+\bar{r}_n^2}.
	\end{align*}
	As $m>\delta_{T_K(\mu_0)}$, $\bar{r}_n(\sigma)=\bigo(1)$ as $\sigma \downarrow 0$. Taking $\sigma \downarrow 0$ on both sides of the above identity, $\bar{r}_\ast\equiv \lim_{\sigma \downarrow 0} \bar{r}_n(\sigma)$ exists and satisfies $\delta_{T_K(\mu_0)} = m \bar{r}_\ast/(1+\bar{r}_\ast^2)$. Solving the equation yields that $\bar{r}_\ast^2 = \delta_{T_K(\mu_0)}/(m-\delta_{T_K(\mu_0)})$. \qed

 \subsection{Proof of Proposition \ref{prop:R2}}

 	\noindent (1a).	Suppose (\ref{cond:R2}) holds, and we wish to prove that $m^{-1} \pnorm{Y-X\hat{\mu}(\sigma)}{}^2$ is bounded away from $0$ in probability, or equivalently, $\Phi^{\textrm{p}} (G)$ is bounded away from $0$ in probability. Note that using Proposition \ref{prop:local_destoc}-(2) and $
 	\E^\xi H\big(\omega_{m/n}(r_{n,\xi},\sigma_m)\big) \pequiv \E H(\omega_{m/n}(r_n))$ for $H \in \{\err,\lrt\}$, 
 	\begin{align}\label{ineq:prop_R2_0}
 	&\min_{\alpha \in  [0,L_w]} \max_{\eta\geq 0}\overline{\mathsf{D}}(\alpha,\eta)_+^2 \nonumber\\
& = \frac{1}{r_\vee^2+\sigma_m^2 }\bigg[\sigma_m^2-\frac{1+\smallo(1)}{2n}\Big(\E^\xi \lrt\big(\omega_{m/n}(r_{n,\xi},\sigma_m)\big)-\E^\xi\err\big(\omega_{m/n}(r_{n,\xi},\sigma_m)\big) \Big) \bigg]_+^2 \nonumber \\
 	&= \frac{1}{r_\vee^2+\sigma_m^2}\cdot \bigg(\sigma_m^2- (1+\smallop(1))\frac{ \E \lrt(\omega_n)-\E \err(\omega_n) }{2n}\bigg)_+^2\nonumber\\
 	&=\frac{1+\smallop(1)}{r_n^2+\sigma^2}\cdot \bigg(\sigma^2- \frac{ \E \lrt(\omega_n)-\E \err(\omega_n) }{2n}\bigg)_+^2+\smallop(1).
 	\end{align}
 	So under (\ref{cond:R2}), we have 
 	\begin{align*}
 	\min_{\alpha \in  [0,L_w]} \max_{\eta\geq 0}\overline{\mathsf{D}}(\alpha,\eta)_+^2\geq \epsilon_0(1+\smallop(1))
 	\end{align*} 
 	for some $\epsilon_0>0$. By (\ref{ineq:sensitivity_anal_3}) and using Lemma \ref{lem:apriori_localization_u_v} followed by letting $\epsilon' \downarrow 0$ therein, and finally taking expectation for $\xi$, we obtain that for any $\phi \in \R$, 
 	\begin{align}\label{ineq:prop_R2_2}
 	\Prob\big(\Phi^{\textrm{p}}_{L_w} (G)\leq \phi\big)&\leq 2\Prob\big(\Phi^{\textrm{a}}_{L_w} (g,h)\leq \phi\big),\nonumber\\
 	\Prob\big(\Phi^{\textrm{p}}_{L_w} (G)\geq \phi\big)&\leq 2\Prob\big(\Phi^{\textrm{a}}_{L_w} (g,h)\geq \phi\big).
 	\end{align}
 	By Proposition \ref{prop:adjusted_AO}, and using $\phi\equiv \epsilon_0/4$ in the first inequality of the above display, we have proven that $\Phi^{\textrm{p}}_{L_w} (G)\geq \epsilon_0/4$ with asymptotic probability $1$. On the other hand, under (\ref{cond:R2}), Theorem \ref{thm:risk_asymp} applies so localization of $w$ in the PO problem $\Phi^{\textrm{p}} (G)$ can be done within $\pnorm{w}{}=\bigop(1)$. In other words, by choosing $L_w$ large enough, $\Phi^{\textrm{p}}_{L_w} (G)= \Phi^{\textrm{p}} (G)$ holds with high enough probability for $n$ large. This proves that $\Phi^{\textrm{p}} (G)$ is bounded away from $0$ in probability under (\ref{cond:R2}).

 	\noindent (1b). Suppose the LHS of (\ref{cond:R2}) is $>1$, and we wish to prove that $m^{-1} \pnorm{Y-X\hat{\mu}(\sigma)}{}^2=\smallop(r_n^2)$, or equivalently, $\Phi^{\textrm{p}} (G)=\smallop(r_n^2)$. To this end, note that
 	\begin{align*}
 	&\min_{\alpha \in  [0,L_w]} \max_{\eta\geq 0}\overline{\mathsf{D}}(\alpha,\eta)_+^2 = 0\quad 
 	\Leftrightarrow \quad  \min_{\alpha \in  [0,L_w]} \max_{\eta\geq 0}\overline{\mathsf{D}}(\alpha,\eta)\leq 0. 
 	\end{align*}
 	Using the first-order optimality condition, an inner saddle point $(\bar{\alpha},\bar{\eta})$ for the right hand side min-max problem in the above display satisfies $\bar{\alpha}=1$ and $\bar{\eta}=\sqrt{r_{n,\xi}^2+\sigma_m^2}$ with high probability. Consequently, again using Proposition \ref{prop:local_destoc}-(2) and the assumption that the LHS of (\ref{cond:R2}) is $>1$, it follows that
 	\begin{align*}
 	&\min_{\alpha \in  [0,L_w]} \max_{\eta\geq 0}\overline{\mathsf{D}}(\alpha,\eta) = \overline{\mathsf{D}}(1,\bar{\eta}) \\
 	&= \frac{1+\smallop(1)}{\sqrt{r_{n}^2+\sigma^2}}\cdot \bigg(\sigma^2- \frac{ \E \lrt(\omega_n)-\E \err(\omega_n) }{2n}\bigg)+\smallop(1)<0
 	\end{align*}
 	with high probability for $n$ large. This means that 
 	\begin{align}\label{ineq:prop_R2_1}
 	\min_{\alpha \in  [0,L_w]} \max_{\eta\geq 0}\overline{\mathsf{D}}(\alpha,\eta)_+^2 = 0
 	\end{align}
 	with high probability for $n$ large. By the second inequality of (\ref{ineq:prop_R2_2}), Proposition \ref{prop:adjusted_AO} and (\ref{ineq:prop_R2_1}), for every slowly decreasing $\epsilon_n \downarrow 0$, 
 	\begin{align*}
 	\Prob\big(\Phi^{\textrm{p}}_{L_w} (G)> \epsilon_{n} r_{n}^2\big)\leq 2\Prob\big(\smallop(r_n^2)\geq \epsilon_{n} r_{n}^2\big)= 0. 
 	\end{align*}
 	This means $\Phi^{\textrm{p}} (G)\leq \Phi^{\textrm{p}}_{L_w} (G)=\smallop(r_n^2)$.

 	\noindent (2). The calculations in (\ref{ineq:prop_R2_0}) show that
 	\begin{align}\label{ineq:prop_R2_3}
 	\lim_n \min_{\alpha \in  [0,L_w]} \max_{\eta\geq 0}\overline{\mathsf{D}}(\alpha,\eta)_+^2  = \sigma^2 \hbox{ in probability}
 	\end{align}
 	if and only if 
 	\begin{align*}
 	\lim_n \frac{1}{2n}\big(\E \lrt(\omega_n)-\E \err(\omega_n)\big)=0,\quad \lim_n r_n =0. 
 	\end{align*}
 	When (\ref{ineq:prop_R2_3}) holds, the arguments in (1a) show that  $\Phi^{\textrm{p}} (G)\pequiv \sigma^2$. When (\ref{ineq:prop_R2_3}) fails, say, the limsup of LHS of (\ref{ineq:prop_R2_3}) is less or equal than $(1-\epsilon)\sigma^2$ for some $\epsilon \in (0,1]$, the arguments in (1b) show that $\Phi^{\textrm{p}} (G)\leq (1+\smallop(1))(1-\epsilon)\sigma^2$. 
 	
 	\noindent (3). We choose $K=K_+=\{\mu\geq 0: \mu \in \R^n\}$, $\xi\equald \mathcal{N}(0,I_m)$ (so $\sigma=1$) and $m=(1/2+\epsilon)\cdot n$ for some $\epsilon \in (0,1/2)$. Let $\mu_0=u_n\bm{1}$, where $u_n\uparrow \infty$ will be chosen later. Using the fixed point equation (\ref{eqn:orthant_fixed_point_eqn}), we must have $r_n\asymp u_n \uparrow \infty$. This holds only if $\epsilon<1/2$. So with $\omega_{\epsilon,n}\equiv \omega_{1/2+\epsilon}(r_n)$, we have $u_n/\omega_{\epsilon,n}\asymp 1$ whereas $\omega_{\epsilon,n}\uparrow \infty$. Consequently,
 	\begin{align*}
 	\frac{1}{2n\sigma^2}\big(\E \lrt(\omega_{\epsilon,n})-\E \err(\omega_{\epsilon,n})\big) = \omega_{\epsilon,n}^2  \mathsf{H}(u_n/\omega_{\epsilon,n})\uparrow \infty.
 	\end{align*}
 	So the RHS of the above display will strictly exceed $1$ for $n$ large. As $m\leq n$, $(XX^\top)^{-1}$ is a.s. well-defined and therefore $\tilde{\mu}=\mu_0+X^\top (XX^\top)^{-1}\xi$ is also a.s. well-defined with $m^{-1}\pnorm{Y-X\tilde{\mu}}{}^2=0$ a.s..  Furthermore, as
 	\begin{align*}
 	\pnorm{X^\top (XX^\top)^{-1}\xi}{\infty}&=\max_{1\leq j\leq n} \abs{X_{\cdot j}^\top (XX^\top)^{-1}\xi}\leq \max_{1\leq j\leq n} \pnorm{X_{\cdot j}}{}\cdot \pnorm{(X X^\top)^{-1}}{\op} \cdot \pnorm{\xi}{},
 	\end{align*} 
 	we have $\pnorm{X^\top (XX^\top)^{-1}\xi}{\infty}\lesssim n^2$ with overwhelmingly high probability. Consequently, by choosing $u_n=n^3$, for $n$ large, $\tilde{\mu}\in K_+$ and $\min_{1\leq j\leq n} \tilde{\mu}_j\geq n^3/2 $ with overwhelmingly high probability. This means with high probability $\tilde{\mu}$ satisfies (\ref{def:approximate_minimizer}), and for any $v \in \mathrm{null}(X)$ with $\pnorm{v}{}^2\leq n$, where $\mathrm{null}(X)\equiv \{\mu\in \R^n: X\mu =0\}$ be the (random) null space of $X$, we have $\tilde{\mu}_v\equiv \tilde{\mu}+v \in K_+$ and therefore satisfies (\ref{def:approximate_minimizer}) with high probability. On the other hand, for any such prescribed $v$, 
 	\begin{align*}
 	n^{-1}\pnorm{\tilde{\mu}_v-\mu_0}{}^2 = n^{-1}\big(\pnorm{X^\top (XX^\top)^{-1}\xi}{}^2 + \pnorm{v}{}^2\big),
 	\end{align*}
 	where the intersection term $2n^{-1}\iprod{X^\top (XX^\top)^{-1}\xi }{v}=0$ vanishes due to $Xv=0$ by the choice of $v \in \mathrm{null}(X)$. As a result,
 	\begin{align*}
 	n^{-1}\sup_{\substack{v \in \mathrm{null}(X),\\ \pnorm{v}{}^2\leq n}}\pnorm{\tilde{\mu}_v-\mu_0}{}^2 = n^{-1}\inf_{\substack{v \in \mathrm{null}(X),\\ \pnorm{v}{}^2\leq n}}\pnorm{\tilde{\mu}_v-\mu_0}{}^2+1,
 	\end{align*}
 	and therefore a deterministic probabilistic limit does not exist for $n^{-1}\pnorm{\tilde{\mu}_v-\mu_0}{}^2$ any choice of valid $v$'s (so that $\tilde{\mu}_v\in K_+$ satisfies (\ref{def:approximate_minimizer})).	\qed

\subsection{Proof of Proposition \ref{prop:R2_verify}}

	\noindent (1). By Proposition \ref{prop:lrt_prop}-(2) 
	\begin{align*}
	0&\leq \hbox{LHS of (\ref{cond:R2})} \leq \frac{1}{n\sigma^2}\cdot \bigg[\omega_n \int_0^{\omega_n} \frac{\E\err(\tau)}{\tau^2}\,\d{\tau}-\E \err(\omega_n)\bigg]\\
	& \leq \frac{1}{n\sigma^2} \big(\omega_n^2 \delta_{T_K(\mu_0)} - nr_n^2\big) = \frac{\delta_{T_K(\mu_0)}}{m}-\frac{1}{\sigma^2}\bigg(1-\frac{\delta_{T_K(\mu_0)}}{m}\bigg)r_n^2\leq \frac{\delta_{T_K(\mu_0)}}{m}.
	\end{align*}
	So under $\liminf_n(m/\delta_{T_K(\mu_0)})>1$, the limsup of the right most side of the above display is $<1$, verifying (\ref{cond:R2}). Furthermore, if $m\gg \delta_{T_K(\mu_0)}$, then the right most side of the above display converges to $0$, i.e., (\ref{cond:R2}) holds with limit $0$.

	\noindent(2). Take any $\nu \in L(K)$. Then $K\pm \nu \subset K$. This means $\nu$ satisfies the conditions in Proposition \ref{prop:lrt_prop}-(2), and therefore
	\begin{align*}
	\hbox{LHS of (\ref{cond:R2})}&\leq \frac{1}{n\sigma^2}\cdot \E \inf_{\nu \in L(K)} \iprod{\hat{\mu}^{\seq}_K(\omega_n)-\mu_0-\omega_n h}{\mu_0-\nu}\\
	&\leq  \frac{1}{n\sigma^2}\cdot\inf_{\nu \in L(K)}\E \iprod{\hat{\mu}^{\seq}_K(\omega_n)-\mu_0}{\mu_0-\nu}\\
	&\leq  \frac{1}{n\sigma^2}\cdot\E^{1/2} \err(\omega_n)\cdot \inf_{\nu \in L(K)} \pnorm{\mu_0-\nu}{}\\
	&\leq  \frac{1}{n\sigma^2}\cdot \sqrt{n r_n^2}\cdot \inf_{\nu \in L(K)} \pnorm{\mu_0-\nu}{} = \hbox{LHS of (\ref{cond:R2_suff})}.
	\end{align*}
	Consequently, (\ref{cond:R2}) holds provided the limsup of the RHS is $<1$, proving the claim. When $r_n\to 0$ and $\inf_{\nu \in L(K)}\pnorm{\mu_0-\nu}{}/\sqrt{n}=\mathcal{O}(1)$, the RHS vanishes, so  (\ref{cond:R2}) holds with limit $0$. \qed

\subsection{Proof of Proposition \ref{prop:regime_low_m}}

	Let $\mathrm{null}(X)\equiv \{\mu\in \R^n: X\mu =0\}$ be the (random) null space of $X$. Then using e.g. \cite[Step 2, pp. 11]{goldstein2017gaussian}, with $Q$ denoting a uniformly random $n\times n$ orthogonal matrix, we have
	\begin{align*}
	\Prob\big(K\cap \mathrm{null}(X)=\{0\}\big)& = \Prob \big(K\cap Q(\R^{n-m}\times \{0\})=\{0\}\big).
	\end{align*}
	Using \cite[(22)-(23), pp. 12]{goldstein2017gaussian}, with $V_K$ denoting the integer-valued random variable associated with the intrinsic volumes of $K$ (see e.g. \cite[Definition 2.3]{han2022high}), the RHS of the above display can be bounded by $\Prob(V_K\leq m)$, so we arrive at
	\begin{align*}
	\Prob\big(K\cap \mathrm{null}(X)=\{0\}\big)\leq \Prob(V_K\leq m).
	\end{align*}
	Now using that $\E V_K =\delta_K$ and $\var(V_K)\leq 4\delta_K$ (see e.g. \cite[Lemma 2.4]{han2022high}), we may continuing bounding the RHS of the above display:
	\begin{align*}
	\Prob\big(K\cap \mathrm{null}(X)=\{0\}\big)&\leq \Prob\big(\abs{V_K-\E V_K}\geq (\delta_K-m)_+\big)\\
	&\leq \frac{4\delta_K}{ (\delta_K-m)_+^2 } = \frac{4}{\delta_K (1-m/\delta_K)_+^2} \to 0. 
	\end{align*}
	Consequently, on an event $\omega \in E$ with $\Prob(E)\to 1$, $K\cap \mathrm{null}(X)$ contains a nontrivial element, say, $0\neq \mu(\omega) \in K$. This means that for this fixed $\omega \in E$ due to the randomness of $X$, if $\hat{\mu}(\sigma) \in \argmin_{\mu \in K}\pnorm{Y-X\mu}{}^2$ (which depends on $\omega$ and also the randomness of $\xi\in \R^m$) is a minimizer, then elements in the  family $\{\hat{\mu}(\sigma)+c \mu(\omega):c\geq 0\}\subset K$ are also minimizers. So for $\omega \in E$, 
	\begin{align*}
	\sup_{\tilde{\mu}(\sigma) \in \argmin\limits_{\mu \in K} \pnorm{Y-X\mu}{}^2} n^{-1}\pnorm{\tilde{\mu}(\sigma)-\mu_0}{}^2\geq \sup_{c\geq 0} n^{-1}\pnorm{\hat{\mu}(\sigma)+c \mu(\omega)-\mu_0}{}^2 = \infty. 
	\end{align*}
	The proof is complete.\qed

\subsection{Proof of Theorem \ref{thm:risk_asymp_consist}}

	We first prove that
	\begin{align}\label{ineq:risk_asymp_consist_1}
	m/\delta_K \to \infty.
	\end{align}
	To see this, note that for any $\epsilon\in (0,1)$, there exists some $C_\epsilon>0$ such that
	\begin{align*}
	\err(\sigma)&\geq \big(\pnorm{\Pi_K(\sigma h)}{}-2\pnorm{\mu_0}{}\big)_+^2\geq (1-\epsilon) \pnorm{\Pi_K(\sigma h)}{}^2-C_\epsilon \pnorm{\mu_0}{}^2.
	\end{align*}
	Using Proposition \ref{prop:est_err_stability}, we have for any $M>1$ and $\epsilon \in (0,1)$,
	\begin{align*}
	\mathfrak{o}(1)& = n^{-1} \E \err\big(\omega_{m/n}(0)\big) \geq \frac{1}{nM^2}\cdot \E \err\big(M\omega_{m/n}(0)  \big)\\
	&\geq \frac{1}{nM^2} \bigg[(1-\epsilon)\cdot \frac{M^2 \sigma^2}{m/n}\cdot \delta_K-C_\epsilon \pnorm{\mu_0}{}^2\bigg] = (1-\epsilon)\sigma^2\cdot \frac{\delta_K}{m}-C_\epsilon\frac{\pnorm{\mu_0}{}^2}{n M^2}.
	\end{align*}
	Now taking $M\equiv \pnorm{\mu_0}{}$ and letting $n \to \infty$ we arrive at the claim (\ref{ineq:risk_asymp_consist_1}). By Proposition \ref{prop:exist_unique_fixed_pt}, the fixed point equation (\ref{eqn:fixed_pt_eqn}) has a unique solution for $n$ large, which we denote as $r_n$. We next prove that 
	\begin{align}\label{ineq:risk_asymp_consist_2}
	r_n \to 0.
	\end{align}
	To see this, again by Proposition \ref{prop:est_err_stability}, we have
	\begin{align*}
	nr_n^2 =  \E \err\big(\omega_{m/n}(r_n)\big)\leq \bigg(1+\frac{r_n^2}{\sigma^2}\bigg) \E \err\big(\omega_{m/n}(0) \big) = \bigg(1+\frac{r_n^2}{\sigma^2}\bigg)\cdot n\bar{r}_n^2. 
	\end{align*}
	Equivalently, we have $
	r_n^2\big(1-{\bar{r}_n^2}/{\sigma^2}\big)\leq \bar{r}_n^2$. 
	Now (\ref{ineq:risk_asymp_consist_2}) follows as $\bar{r}_n\to 0$ by the assumption. Finally, we shall use (\ref{ineq:risk_asymp_consist_2}) to prove $r_n\simeq \bar{r}_n$. To see this, using Lemma \ref{lem:monotone_est_err}-(1) and Proposition \ref{prop:est_err_stability} again, we have
	\begin{align*}
	\E \err\big(\omega_{m/n}(0)\big)&\leq \E \err\big( \omega_{m/n}(r_n) \big)\leq \big(1+\mathfrak{o}(1)\big) \E \err\big(\omega_{m/n}(0)\big).
	\end{align*}
	By using the definitions of $r_n$ and $\bar{r}_n$, we conclude that  $r_n\simeq \bar{r}_n$. The claim now follows by an application of Theorem \ref{thm:risk_asymp}.\qed

\subsection{Proof of Theorem \ref{thm:risk_asymp_const}}
 
 	We only need to prove that it is valid to replace (R2) by (R2-c) under $m/n=\tau$. The limit version follows from minor notational modifications. Recall $\omega_n = \omega_{m/n}(r_n)$ and let $\omega\equiv \omega_{\tau}(\mathsf{r})$. First, by the stability estimate in Proposition \ref{prop:est_err_stability}, we have
 	\begin{align*}
 	&\biggabs{\frac{\err(\omega_n)}{\err(\omega)}-1}\vee \biggabs{\frac{\err(\omega)}{\err(\omega_n)}-1} 
    \lesssim_\sigma \abs{r_n^2-\mathsf{r}^2}.
 	\end{align*}
 	This means
 	\begin{align*}
 	\frac{1}{n\sigma^2}\bigabs{\E \err(\omega_n)-\E \err(\omega)} \lesssim_\sigma \frac{\abs{r_n^2-\mathsf{r}^2}\cdot \E \err(\omega_n)}{n} = \abs{r_n^2-\mathsf{r}^2}\cdot  r_n^2.
 	\end{align*}
 	Next, using that for any $\sigma,\sigma'>0$, $\pnorm{ \hat{\mu}^{\seq}_K(\sigma)-\hat{\mu}^{\seq}_K(\sigma')  }{}\leq \abs{\sigma-\sigma'}\pnorm{h}{}$,
 	\begin{align*}
 	&\frac{1}{n\sigma^2} \abs{\dof(\omega_n)-\dof(\omega)}=\frac{1}{n\sigma^2}\cdot \bigabs{\omega_n \E \iprod{ \hat{\mu}^{\seq}_K(\omega_n)-\mu_0 }{h} -\omega \E \iprod{ \hat{\mu}^{\seq}_K(\omega)-\mu_0 }{h}    }\\
 	&\leq \frac{\abs{\omega_n-\omega} }{n\sigma^2}\cdot \bigabs{ \E \iprod{ \hat{\mu}^{\seq}_K(\omega_n)-\mu_0 }{h} }+\frac{\omega}{n\sigma^2}\cdot \bigabs{  \E \iprod{ \hat{\mu}^{\seq}_K(\omega_n)-\hat{\mu}^{\seq}_K(\omega)}{h} }\\
 	&\lesssim \abs{\omega_n-\omega} \cdot n^{-1/2} \E^{1/2} \err(\omega_n)+ \omega n^{-1/2}\cdot \E^{1/2}\pnorm{ \hat{\mu}^{\seq}_K(\omega_n)-\hat{\mu}^{\seq}_K(\omega)  }{}^2\\
 	& \lesssim \abs{\omega_n-\omega} (r_n + \omega). 
 	\end{align*}
 	Using that $\omega_n^2-\omega^2 = (r_n^2-\mathsf{r}^2)/(m/n)$, and that
 	\begin{align*}
 	(\omega_n-\omega)^2 = \frac{1}{m/n}\Big(\sqrt{r_n^2+\sigma^2}-\sqrt{\mathsf{r}^2+\sigma^2}\Big)^2\lesssim \frac{1}{m/n} \Big[(r_n^2-\mathsf{r}^2)^2\wedge \abs{r_n^2-\mathsf{r}^2}\Big].
 	\end{align*}
 	These estimates show that under $r_n\to \mathsf{r} \in (0,\infty)$,
 	\begin{align}\label{ineq:risk_asymp_const_1}
 	& \frac{1}{n\sigma^2}\Big(-\E \err(\omega_n)+ \E\dof(\omega_n)\Big)- \frac{1}{n\sigma^2}\Big(-\E \err(\omega)+ \E\dof(\omega)\Big) \nonumber\\
 	&= \bigo\Big( \big(1\vee (m/n)^{-2}\big) \abs{r_n^2-\mathsf{r}^2}\Big),
 	\end{align}
 	as desired. \qed

\section{Proofs for Section \ref{section:examples}}\label{section:proof_example}

\subsection{Proofs for Section \ref{section:orthant}}

The following proposition summarizes some basic properties of $\mathsf{G},\mathsf{H}$.
\begin{proposition}\label{prop:orthant_GH}
	The following hold for $\mathsf{G},\mathsf{H}$ defined in (\ref{def:orthant_GH}).
	\begin{enumerate}
		\item 
		\begin{enumerate}
			\item $\mathsf{G}$ is smooth, non-negative, strictly increasing with $\mathsf{G}(x)-1/2\sim x^2/2$ as $x\downarrow 0$, and $\mathsf{G}(x)-1\sim -2\varphi(x)/x$ as $x\uparrow \infty$.
			\item For any $x\geq 0, \delta\geq 0$, $
			\mathsf{G}(x)\leq \mathsf{G}\big((1+\delta)x\big)\leq \mathsf{G}(x)(1+8\delta)$. 
			\item $x \mapsto x^2 \mathsf{G}(1/x)$ is non-decreasing on $(0,\infty)$ with $x^2 \mathsf{G}(1/x)\sim x^2$ as $x\downarrow 0$, and $x^2 \mathsf{G}(1/x)\sim x^2/2$ as $x\uparrow \infty$. 
		\end{enumerate}

		\item 
		\begin{enumerate}
			\item $\mathsf{H}$ is smooth, non-negative with $\mathsf{H}(x)\sim x/\sqrt{2\pi}$ as $x\downarrow 0$, and $\mathsf{H}(x)\sim \varphi(x)/x$ as $x\uparrow \infty$.
			\item  The uniform bound $\sup_{x\geq 0} \mathsf{H}(x)<0.13$ holds.
			\item  $x \mapsto x^2 \mathsf{H}(1/x)$ is non-decreasing on $(0,\infty)$ with $x^2 \mathsf{H}(1/x)\sim x^3\varphi(1/x)$ as $x \downarrow 0$, and $x^2\mathsf{H}(1/x)\sim x/\sqrt{2\pi}$ as $x \uparrow \infty$. 
		\end{enumerate}
		
		\item For any $\sigma>0$ and $\mu_0 \in K_+$, we have $\E \err(\sigma) = \sigma^2 \sum_{i=1}^n \mathsf{G}(\mu_{0,i}/\sigma)$ and $\E \lrt(\sigma)=\E \err(\sigma)+ 2\sigma^2 \sum_{i=1}^n \mathsf{H}(\mu_{0,i}/\sigma)$. 
	\end{enumerate}
\end{proposition}

\begin{proof}[Proof of Proposition \ref{prop:orthant_GH}]
	\noindent (1). $\mathsf{G}\geq 0$ by definition. The first and second derivatives of $\mathsf{G}$ are given by 
	\begin{align*}
	\mathsf{G}'(x)&= \varphi(x)-\big(\varphi(x)-x^2\varphi(x)\big)+2x(1-\Phi(x))-x^2\varphi(x)\\
	&=2x(1-\Phi(x))> 0,\\
	\mathsf{G}''(x)&=2\big(1-\Phi(x)-x\varphi(x)\big), \quad x>0.
	\end{align*}
	So $\mathsf{G}$ is strictly increasing, and $\mathsf{G}(x)-1/2\sim x^2/2$ as $x\downarrow 0$. On the other hand, by Lemma \ref{lem:mill_ratio}, as $x\uparrow \infty$, 
	\begin{align*}
	\mathsf{G}(x)- 1&= (x^2-1)(1-\Phi(x))-x\varphi(x)\\
	&\sim \varphi(x)(x^2-1)\bigg(\frac{1}{x}-\frac{1}{x^3}+\frac{\bigo(1)}{x^5}\bigg)-x\varphi(x) \\
	&= \bigg[\frac{(x^2-1)^2}{x^3}-x\bigg]\varphi(x)+\frac{\bigo(1)}{x^3}\varphi(x)\sim -2 \varphi(x)/x. 
	\end{align*}
	This proves (a).

	For (b), as the function $u\mapsto u(1-\Phi(u))$ is decreasing on $[2,\infty)$, and $u(1-\Phi(u))\leq u\cdot \varphi(u)/u <1$, $u^2(1-\Phi(u))\leq u\varphi(u)\leq \varphi(1)<1$ on $[0,\infty)$, we have
	\begin{align*}
	&\biggabs{\frac{ \mathsf{G}\big((1+\delta) x\big)}{\mathsf{G}(x)} -1} = \frac{1}{\mathsf{G}(x)} \int_x^{(1+\delta)x} 2u(1-\Phi(u))\,\d{u}\\
	&\leq 4 \delta x\cdot \sup_{u \in [x,(1+\delta)x]} u(1-\Phi(u))\leq 4\delta\big[x\cdot \bm{1}_{x\leq 2}+x^2(1-\Phi(x))\bm{1}_{x> 2}\big]\leq 8\delta. 
	\end{align*}

	For (c), let $\mathsf{G}_0(x)\equiv x^{-2} \mathsf{G}(x)$. Then
	\begin{align*}
	\mathsf{G}_0'(x)& = \frac{1}{x^4}\big(\mathsf{G}'(x)x^2-2x\mathsf{G}(x)\big)=-\frac{2}{x^3}\big(\Phi(x)-x\varphi(x)\big)\leq 0,
	\end{align*}
	where in the last inequality we used the easily verified fact that $\Phi(x)\geq x\varphi(x)$ for all $x\geq 0$. This means that $x\mapsto x^{-2} \mathsf{G}(x)$ is non-increasing. Combining (1a) to conclude.
	
	\noindent (2). For (a), that $\mathsf{H}\geq 0$ follows from the standard bound $\Phi(-x)=1-\Phi(x)\leq \varphi(x)/x$ for $x>0$. Clearly $\mathsf{H}(x)\sim x/\sqrt{2\pi}$ as $x \downarrow 0$. On the other hand, by Lemma \ref{lem:mill_ratio},
	\begin{align*}
	\mathsf{H}(x) = x\varphi(x)-x^2\varphi(x)\bigg[\frac{1}{x}-\frac{1}{x^3}+\frac{\bigo(1)}{x^5}\bigg]\sim \varphi(x)/x,\quad x\uparrow \infty.
	\end{align*}

	For (b), first note that $\mathsf{H}'(x) = \varphi(x)-\mathsf{G}'(x)=\varphi(x)-2x \Phi(-x)$, so $\mathsf{H}$ attains maximum at $x_0 >0$ such that $\mathsf{H}'(x_0)=0$, i.e., $\Phi(-x_0)=\varphi(x_0)/(2x_0)$. Consequently 
	\begin{align*}
	\sup_{x\geq 0}\mathsf{H}(x)&=\mathsf{H}(x_0) = x_0\varphi(x_0)-x_0^2 \Phi(-x_0)\\
	&=\frac{x_0\varphi(x_0)}{2} \leq \frac{1}{2}\sup_{x\geq 0} x\varphi(x) = \frac{\varphi(1)}{2}<0.13.
	\end{align*}
	In the last equality we used the easily verified fact that $x\mapsto x\varphi(x)$ attains maximum at $x=1$ over $[0,\infty)$. This proves the desired bound for $\mathsf{H}$. 
	
	For (c), let $\mathsf{H}_0(x)=x^{-2}\mathsf{H}(x)$. Then 
	\begin{align*}
	\mathsf{H}_0'(x)&=\frac{1}{x^4}\big(\mathsf{H}'(x)x^2-2x\mathsf{H}(x)\big)=-\frac{\varphi(x)}{x^2}<0.
	\end{align*}
	This means that $x\mapsto x^{-2} \mathsf{H}(x)$ is non-increasing. Combining with (2a) to conclude.

	\noindent (3).
	The convex constrained LSE in the Gaussian sequence model enjoys a closed form: $\hat{\mu}_{K_+}^{\seq}(\sigma)=\big((\mu_{0,i}+\sigma h_i)_+\big)$. Consequently, using Lemma \ref{lem:gaussian_trunc},
	\begin{align*}
	\E \err(\sigma) &= \E \pnorm{\hat{\mu}_{K_+}^{\seq}(\sigma)-\mu_0}{}^2 = \sum_i \E \Big((\mu_{0,i}+\sigma h_i)_+-\mu_{0,i} \Big)^2\\
	& = \sum_i \bigg[\E (\sigma h_i)^2 \bm{1}_{h_i\geq -\mu_{0,i}/\sigma}+ \E (\mu_{0,i}^2) \bm{1}_{h_i<-\mu_{0,i}/\sigma}\bigg]\\
	& = \sigma^2 \sum_i \Big(\Phi(\mu_{0,i}/\sigma)- (\mu_{0,i}/\sigma)\varphi(\mu_{0,i}/\sigma)\Big)+ \sum_i \mu_{0,i}^2  \Phi(-\mu_{0,i}/\sigma)\\
	& =\sigma^2 \sum_i \mathsf{G}(\mu_{0,i}/\sigma),
	\end{align*}
	and
	\begin{align*}
	\E \dof(\sigma)&=\E\iprod{ \hat{\mu}_{K_+}^{\seq}(\sigma)-\mu_0 }{\sigma h} = \sigma \sum_i \E (\mu_{0,i}+\sigma h_i)_+ h_i \\
	&= \sigma \sum_i  \E (\mu_{0,i}+\sigma h_i) h_i \bm{1}_{h_i\geq -\mu_{0,i}/\sigma} \\
	& = \sigma \sum_i \mu_{0,i} \E h_i \bm{1}_{h_i\geq -\mu_{0,i}/\sigma}+ \sigma^2 \sum_i \E h_i^2 \bm{1}_{h_i\geq -\mu_{0,i}/\sigma}\\
	& =  \sigma \sum_i \mu_{0,i} \varphi(\mu_{0,i}/\sigma)+ \sigma^2 \sum_i \Big(\Phi(\mu_{0,i}/\sigma)- (\mu_{0,i}/\sigma)\varphi(\mu_{0,i}/\sigma)\Big)\\
	& = \sigma^2 \sum_i \Phi(\mu_{0,i}/\sigma).
	\end{align*}
	Now using the relationship
	\begin{align*}
	\E \lrt(\sigma) &= 2\E \dof(\sigma)- \E \err(\sigma) = \E \err(\sigma)+2 \sigma^2 \sum_i \mathsf{H}(\mu_{0,i}/\sigma)
	\end{align*}
	to conclude. 
\end{proof}

We need two technical lemmas before the proof of Theorem \ref{thm:orthant}.

\begin{lemma}\label{lem:uniform_conv_G}
	Let $U\geq 0$ be any non-negative random variable. It holds that
	\begin{align*}
	\sup_{\omega > 0}\biggabs{\frac{ \Prob_n \mathsf{G}\big(U/\omega\big) }{ \E \mathsf{G}\big(U/\omega\big) }-1}=\smallop(1). 
	\end{align*}
\end{lemma}
\begin{proof}
	As $\mathsf{G}\geq 1/2$, it suffices to prove that
	\begin{align*}
	\sup_{\omega > 0} \bigabs{\Prob_n \mathsf{G}\big(U/\omega \big)- \E \mathsf{G}\big(U/\omega \big) }=\smallop(1). 
	\end{align*} 
	Note that the class $\{\mathsf{G}(\cdot/\omega): \omega >0\}\subset \mathcal{M}$ containing all monotone functions on $[0,\infty)$ taking value in $[0,1]$, so by \cite[Theorem 2.7.5]{van1996weak} (bracketing entropy for $\mathcal{M}$) and the Glivenko-Cantelli theorem (cf. \cite[Theorem 2.4.1]{van1996weak}), we conclude
	\begin{align*}
	\sup_{\omega > 0} \bigabs{\Prob_n \mathsf{G}\big(U/\omega \big)- \E \mathsf{G}\big(U/\omega \big) }\leq \sup_{g \in \mathcal{M}} \abs{(\Prob_n-P)g}=\smallop(1),
	\end{align*}
	as desired.
\end{proof}

\begin{lemma}\label{lem:orthant_prob_deter}
	Suppose the conditions in the beginning of Theorem \ref{thm:orthant} hold. 
	\begin{enumerate}
		\item The fixed point equation (\ref{eqn:orthant_fixed_point_eqn}) has a unique solution $r_{n,+} \in (0,\infty)$.
		\item Let $\Prob_n\equiv n^{-1}\sum_{i=1}^n \delta_{\mu_{0,i}}$. Then almost surely, the fixed point equation 
		\begin{align}\label{eqn:orthant_fixed_point_eqn_1}
		\omega^2_{m/n}(r)\cdot \Prob_n \mathsf{G}\bigg(\frac{U}{\omega_{m/n}(r)}\bigg) = r^2
		\end{align}
		has a unique solution $r_n(\mu_0)$ in $(0,\infty)$. 
		\item Suppose further $m/n>1/2+\epsilon$ for some $\epsilon>0$. Then $r_n(\mu_0)\pequiv r_{n,+}$.
		\item It holds that
		\begin{align*}
		\frac{\sigma^2}{m/n-1/2}\leq r_{n,+}^2\leq \frac{\sigma^2}{(m/n-1)_+}.
		\end{align*}
		If furthermore $m/n>1/2+\epsilon$ for some small $\epsilon>0$, the right hand side of the above display can be replaced by $C_{\epsilon,\sigma,U}>0$. 
	\end{enumerate} 
\end{lemma}
\begin{proof}
	\noindent (1). It is well known that the statistical dimension of $K_+$ is $\delta_{K_+}=n/2$, cf. \cite[Table 3.1]{amelunxen2014living}. Using the same proof as in Proposition \ref{prop:unique_fixed_point_eqn}-(2) but now take further expectation with respect to $\mu_0$ to conclude.
	
	\noindent (2). This is a direct consequence of Proposition \ref{prop:exist_unique_fixed_pt} as $m>\delta_{K_+}=n/2$.
	
	\noindent (3). (\textbf{Step 1}). We show that
	\begin{align}\label{ineq:orthant_prob_deter_1}
	r_n(\mu_0)=\bigop(1),\quad r_{n,+}=\mathcal{O}(1). 
	\end{align}
	We start with $r_{n,+}$. By Proposition \ref{prop:orthant_GH}-(1),  $\delta\mapsto \omega^2_\delta(r) \mathsf{G}(U/\omega_\delta(r))$ is non-increasing, so a solution $\bar{r}_{n,+}$ to the fixed point equation
	\begin{align*}
	\omega^2_{1/2+\epsilon}(r)\cdot \E \mathsf{G}\bigg(\frac{U}{\omega_{1/2+\epsilon}(r)}\bigg) = r^2
	\end{align*}
	provides an upper bound for $r_{n,+}$. This means that we only need to show $\bar{r}_{n,+}=\mathcal{O}(1)$. Suppose the contrary that $\bar{r}_{n,+}\to \infty$. Then
	\begin{align}\label{ineq:orthant_prob_deter_2}
	1&=\limsup_n \frac{ \omega^2_{1/2+\epsilon}(\bar{r}_{n,+}) }{\bar{r}_{n,+}^2}\E \mathsf{G}\bigg(\frac{U}{\omega_{1/2+\epsilon}(\bar{r}_{n,+})}\bigg)\nonumber\\
	& = \limsup_n \frac{1+\sigma^2/\bar{r}_{n,+}^2}{m/n}\cdot \mathsf{G}(0) \leq \frac{1}{2(1/2+\epsilon)}=\frac{1}{1+2\epsilon}<1,
	\end{align}
	leading to a contradiction. This proves $\bar{r}_{n,+}=\mathcal{O}(1)$. Next we prove $r_n(\mu_0)=\bigop(1)$. The idea is similar, but a bit more technical work is needed. In fact, we only need to prove that $\bar{r}_n(\mu_0)=\bigop(1)$, where $\bar{r}_n(\mu_0)$ is a fixed point solution to 
	\begin{align*}
	\omega^2_{1/2+\epsilon}(r)\cdot \Prob_n \mathsf{G}\bigg(\frac{U}{\omega_{1/2+\epsilon}(r)}\bigg) = r^2.
	\end{align*}
	Suppose for some $M_n \to \infty$, $\bar{r}_n(\mu_0)\geq M_n$ with probability at least $\epsilon_0>0$ for $n$ large. On this event,
	\begin{align*}
	1&= \frac{ \omega^2_{1/2+\epsilon}(\bar{r}_n(\mu_0)) }{\bar{r}_n(\mu_0)^2}\Prob_n \mathsf{G}\bigg(\frac{U}{\omega_{1/2+\epsilon}(\bar{r}_n(\mu_0))}\bigg)\\
	&\leq \frac{ \omega^2_{1/2+\epsilon}(\bar{r}_n(\mu_0)) }{\bar{r}_n(\mu_0)^2}\Prob_n \mathsf{G}\bigg(\frac{U}{\omega_{1/2+\epsilon}(M_n)}\bigg) = \frac{ \omega^2_{1/2+\epsilon}(\bar{r}_n(\mu_0)) }{\bar{r}_n(\mu_0)^2}\E \mathsf{G}\bigg(\frac{U}{\omega_{1/2+\epsilon}(M_n)}\bigg)+\smallop(1). 
	\end{align*}
	The last equality follows by Lemma \ref{lem:uniform_conv_G} and the fact that ${ \omega^2_{1/2+\epsilon}(\bar{r}_n(\mu_0)) }/{\bar{r}_n(\mu_0)^2}$ remains bounded for $\bar{r}_n(\mu_0)\geq 1$. For $n$ large enough, the above display reduces to the inequality $1\leq 1/(1+3\epsilon)+\smallop(1)$ that holds on an event with probability at least $\epsilon_0>0$ for all $n$ large, a contradiction. This means $\bar{r}_n(\mu_0)=\bigop(1)$ and therefore $r_n(\mu_0)=\bigop(1)$, completing the proof of (\ref{ineq:orthant_prob_deter_1}).

	\noindent (\textbf{Step 2}).
	By (\ref{ineq:orthant_prob_deter_1}), we may assume without loss generality $r_{n,+}\to \mathsf{r}_+ \in [0,\infty)$ and $r_n(\mu_0)$ converges weakly to a tight random variable $\mathsf{r}(\mu_0)$  along a proper subsequence of $\{n\}$. Due to the fixed point equation, $\mathsf{r}(\mu_0)$ must be degenerate that charges mass at one point, and we denote this point, with slight abuse of notation, again by $\mathsf{r}(\mu_0) \in [0,\infty)$. In other words, $r_n(\mu_0)\pequiv \mathsf{r}(\mu_0)$. By working with further proper subsequence, we replace the convergence in probability to a.s., i.e., $r_n(\mu_0)\to_{a.s.} \mathsf{r}(\mu_0)$, and that $m/n\to \delta \in (1/2,\infty]$. Consequently, by taking limits to the fixed point equations (\ref{eqn:orthant_fixed_point_eqn}) and (\ref{eqn:orthant_fixed_point_eqn_1}) along the aforementioned subsequence, upon using Lemma \ref{lem:uniform_conv_G} to replace $\Prob_n$ by $\E$ in (\ref{eqn:orthant_fixed_point_eqn_1}), we find that both $\mathsf{r}_+$ and $\mathsf{r}(\mu_0)$ is a solution to the fixed point equation 
	\begin{align*}
	\omega^2_{\delta}(r)\cdot \E \mathsf{G}\bigg(\frac{U}{\omega_{\delta}(r)}\bigg) = r^2.
	\end{align*}
	This equation has a unique solution for all $\delta \in (1/2,\infty]$ so $\mathsf{r}_+=\mathsf{r}(\mu_0)$. This means in particular 
	\begin{align}\label{ineq:orthant_prob_deter_3}
	r_n^2(\mu_0) \pequiv r_{n,+}^2,\quad\hbox{ if } r_{n,+}\gtrsim 1. 
	\end{align}
	Next we consider the regime $r_{n,+}\to 0$. Combining Lemma \ref{lem:uniform_conv_G} and Proposition \ref{prop:orthant_GH}-(1), we have
	\begin{align*}
	r_n^2(\mu_0) &= \omega^2_{m/n}(r_n(\mu_0))\cdot \Prob_n \mathsf{G}\bigg(\frac{U}{\omega_{m/n}(r_n(\mu_0))}\bigg)\\
	&\pequiv \omega^2_{m/n}(r_n(\mu_0))\cdot \E \mathsf{G}\bigg(\frac{U}{\omega_{m/n}(r_n(\mu_0))}\bigg)\\
	&\simeq  \omega^2_{m/n}(r_{n,+})\cdot  \E \mathsf{G}\bigg(\frac{U}{\omega_{m/n}(r_{n,+})}\bigg)\cdot \big(1+\mathcal{O}(\abs{r_n^2(\mu_0)-r_{n,+}^2})\big)\\
	& = r_{n,+}^2 \big(1+\mathcal{O}(\abs{r_n^2(\mu_0)-r_{n,+}^2}). 
	\end{align*}
	By (\ref{ineq:orthant_prob_deter_1}), $r_n^2(\mu_0)=r_{n,+}^2\cdot \bigop(1)$. Plugging this back in the above display, we obtain
	\begin{align*}
	r_n^2(\mu_0) \pequiv r_{n,+}^2 \big(1+r_{n,+}^2\cdot \bigop(1)\big).
	\end{align*}
	Consequently
	\begin{align}\label{ineq:orthant_prob_deter_4}
	r_n^2(\mu_0) \pequiv r_{n,+}^2,\quad\hbox{ if } r_{n,+}\to 0. 
	\end{align}
	The proof is complete by combining two cases considered in (\ref{ineq:orthant_prob_deter_3}) and (\ref{ineq:orthant_prob_deter_4}). 
	
	\noindent (4). 	By Proposition \ref{prop:orthant_GH}-(1), $1/2\leq \mathsf{G}\leq 1$, so (\ref{eqn:orthant_fixed_point_eqn}) leads to the inequality
	\begin{align*}
	\frac{1}{2}\cdot \omega_{m/n}^2(r_{n,+}^2)\leq    r_{n,+}^2\leq \omega_{m/n}^2(r_{n,+}^2).
	\end{align*}
	Solving the above display yields the desired inequality. In the regime $1/2+\epsilon<m/n\leq 1+\epsilon$,  (\ref{ineq:orthant_prob_deter_2}) shows that $r_{n,+}^2\leq C_\epsilon$, completing the proof. 
\end{proof}

\begin{proof}[Proof of Theorem \ref{thm:orthant}]
	\noindent (1).  By Lemma \ref{lem:orthant_prob_deter}, $r_n(\mu_0)\pequiv r_{n,+}=\bigo(1)$. On the other hand, $m\gg \mathfrak{L}_n$ automatically holds, so (R1) is satisfied. By Proposition \ref{prop:orthant_GH}-(3), condition (R2) reads 
	\begin{align}\label{ineq:orthant_1}
	\limsup_n \frac{1}{n\sigma^2} \sum_{i=1}^n \omega_{m/n}^2(r_n(\mu_0)) \mathsf{H}\bigg(\frac{\mu_{0,i}}{\omega_{m/n}(r_n(\mu_0))}  \bigg)<1,
	\end{align}
	which, by (\ref{ineq:risk_asymp_const_1}), is equivalent to
	\begin{align}\label{ineq:orthant_2}
	\limsup_n \frac{1}{n\sigma^2} \sum_{i=1}^n \omega_{m/n}^2(r_{n,+}) \mathsf{H}\bigg(\frac{\mu_{0,i}}{\omega_{m/n}(r_{n,+})}  \bigg)<1
	\end{align}
	in probability. On the other hand, using $\omega_{m/n}^2(r_{n,+})\leq 2(r_{n,+}^2+\sigma^2)\leq C_\epsilon$, and the monotonicity of the map $x\mapsto x^2 H(u/x)$ as proved in Proposition \ref{prop:orthant_GH}-(2), the variance of the left hand side of the above display is
	\begin{align*}
	&\frac{1}{n\sigma^4}\cdot \omega_{m/n}^4(r_{n,+})\var\bigg[\mathsf{H}\bigg(\frac{U}{\omega_{m/n}(r_{n,+})}\bigg)  \bigg]\\
	&\leq \frac{1}{n\sigma^4} \E\bigg[ \omega_{m/n}^2(r_{n,+})\mathsf{H}\bigg(\frac{U}{\omega_{m/n}(r_{n,+})}\bigg)    \bigg]^2 \leq n^{-1}\cdot (C_\epsilon^2/\sigma^4)\cdot \E \mathsf{H}^2(U/C_\epsilon^{1/2}).
	\end{align*}
	This means (\ref{ineq:orthant_2}) holds in probability if and only if
	\begin{align}\label{ineq:orthant_3}
	\frac{1}{\sigma^2} \cdot \limsup_n \omega_{m/n}^2(r_{n,+}) \E \mathsf{H}\bigg(\frac{U}{\omega_{m/n}(r_{n,+})}  \bigg)<1.
	\end{align}
	Consequently, under (\ref{ineq:orthant_3}), for any small $\epsilon>0$, there exists an event $E$ on which (\ref{ineq:orthant_1}) holds, and $\Prob_U(E)\geq 1-\epsilon$. As $r_n(\mu_0)^2\pequiv r_{n,+}^2\gtrsim 1/(m/n-1/2)\gg  \mathfrak{L}_n/m$, on the event $E$ we may apply Theorem \ref{thm:risk_asymp} to obtain
	\begin{align}\label{ineq:orthant_4}
	\lim_n \Prob_{X,\xi}\bigg(\biggabs{ \frac{ \pnorm{\hat{\mu}(\sigma)-\mu_0}{}^2 }{nr_n^2(\mu_0)}-1}>\epsilon\bigg)=0.
	\end{align}
	Note that
	\begin{align*}
	\Prob\bigg(\biggabs{\frac{ \pnorm{\hat{\mu}(\sigma)-\mu_0}{}^2 }{nr_n^2(\mu_0)}-1}>\epsilon\bigg)& \leq \E \bm{1}\bigg(\biggabs{\frac{ \pnorm{\hat{\mu}(\sigma)-\mu_0}{}^2   }{nr_n^2(\mu_0)}-1}>\epsilon\bigg) \bm{1}_E + \epsilon.
	\end{align*}
	By (\ref{ineq:orthant_4}) and dominated convergence theorem, the first term in the right hand side of the above display vanishes as $n \to \infty$, so we conclude that 
	\begin{align*}
	n^{-1} \pnorm{\hat{\mu}(\sigma)-\mu_0}{}^2\pequiv r_n^2(\mu_0)\pequiv r_{n,+}^2,
	\end{align*}
	as desired. 
	
	\noindent (2). Consider the fixed point equation
	\begin{align}\label{ineq:orthant_5}
	\omega_\delta^2(r)\cdot \E \mathsf{G}\bigg(\frac{U}{\omega_\delta(r)}\bigg) = r^2
	\end{align}
	which admits a unique solution $\mathsf{r}_+(\delta)$. Fix $1/2<\delta_1\leq \delta_2$. As $r\mapsto r^2 \E \mathsf{G}(U/r)$ is  non-decreasing, we have
	\begin{align*}
	\mathsf{r}_+^2(\delta_1)&= \omega_{\delta_1}^2(\mathsf{r}_+(\delta_1))\cdot \E \mathsf{G}\bigg(\frac{U}{\omega_{\delta_1}(\mathsf{r}_+(\delta_1))}\bigg) \geq \omega_{\delta_2}^2(\mathsf{r}_+(\delta_1))\cdot \E \mathsf{G}\bigg(\frac{U}{\omega_{\delta_2}(\mathsf{r}_+(\delta_1)) }\bigg) \\
	&= n^{-1} \E \bigpnorm{\hat{\mu}^{\seq}_{K_+}\big(\omega_{\delta_2}(\mathsf{r}_+(\delta_1))\big)-\mu_0}{}^2 = n^{-1}\cdot \E \err\big(\omega_{\delta_2}(\mathsf{r}_+(\delta_1))\big). 
	\end{align*}
	By the monotonicity of $r\mapsto \E \err_{K_+}\big(\omega_{\delta_2}(r)\big)/r^2$ (cf. Proposition \ref{prop:unique_fixed_point_eqn}-(1)), we conclude that $\mathsf{r}_+(\delta_2)\leq \mathsf{r}_+(\delta_1)$. This proves that $\delta \mapsto \mathsf{r}_+(\delta)$ is non-increasing on $(1/2,\infty)$. Now separately letting $\delta \uparrow \infty$ and $\delta \downarrow 1/2$ on both sides of (\ref{ineq:orthant_5}), it is easy to see that $\lim_{\delta \uparrow \infty} \mathsf{r}_+(\delta) = 0$ whereas $\lim_{\delta \downarrow 1/2} \mathsf{r}_+(\delta) = \infty$. The remaining claims follow from (1).
	
	\noindent (3). As $m/n \to \infty$, by (\ref{eqn:orthant_fixed_point_eqn}) and that $\mathsf{G}\leq 1$, we have the apriori estimate:
	\begin{align*}
	r_{n,+}^2 \leq \omega_{m/n}^2(r_{n,+}) = \frac{r_{n,+}^2+\sigma^2}{m/n},
	\end{align*}
	which leads to $r_{n,+}\vee \omega_{m/n}(r_{n,+})\to 0$. So we have
	\begin{align*}
	\omega_{m/n}^2(r_{n,+}) \E \mathsf{H}\bigg(\frac{U}{\omega_{m/n}(r_{n,+})}  \bigg) \to 0,
	\end{align*} 
	and by dominated convergence theorem, we have
	\begin{align*}
	0\leq 1- \E \mathsf{G}\bigg(\frac{U}{ \omega_{m/n}(r_{n,+})  }\bigg) &= 1- \frac{1}{2}\cdot \Prob(U=0)-\E \mathsf{G}\bigg(\frac{U}{ \omega_{m/n}(r_{n,+})  }\bigg)\bm{1}_{U>0}\\
	&\to 1-\frac{1}{2}\cdot \Prob(U=0)-\Prob(U>0) = \frac{p_0}{2}.
	\end{align*}
	The above display implies that under $m/n\to \infty$, the fixed point equation (\ref{eqn:orthant_fixed_point_eqn}) reduces to
	\begin{align*}
	r_{n,+}^2=\omega_{m/n}^2(r_{n,+})\bigg(1-\frac{1+\smallo(1)}{2}\cdot p_0 \bigg)= \frac{r_{n,+}^2+\sigma^2}{m/n}\bigg(1-\frac{1+\smallo(1)}{2}\cdot p_0\bigg).
	\end{align*}
	Solving the equation we obtain $r_{n,+}^2 \simeq (1-p_0/2)\sigma^2/(m/n)$. The claim now follows by an application of (1). 
\end{proof}

\subsection{Proofs for Section \ref{section:shape}}

\begin{proof}[Proof of Theorem \ref{thm:orac_ineq_shape}]
	We claim that for any $k \in [1:n]$, any $\mu \in K_S\cap K \in \mathscr{K}_k$ with $S=\{S_\ell\}_{\ell=1}^k$, we have
	\begin{align}\label{ineq:orac_ineq_shape_1}
	\E \err(\sigma)\leq \pnorm{\mu-\mu_0}{}^2+ \sigma^2 \sum_{\ell=1}^k \delta_{K|_{S_\ell}}.
	\end{align}
	First by projection, $\iprod{y-\Pi_K(y)}{\mu-\Pi_K(y)}\leq 0$ for all $\mu \in K$. This means that $\pnorm{y-\mu}{}^2=\pnorm{y-\Pi_K(y)+\Pi_K(y)-\mu}{}^2\geq \pnorm{y-\Pi_K(y)}{}^2+\pnorm{\Pi_K(y)-\mu}{}^2$ for all $\mu \in K$. Expanding the square with $y=\mu_0+\sigma h$, it is easy to see
	\begin{align}\label{ineq:orac_ineq_shape_2}
	\pnorm{\Pi_K(y)-\mu_0}{}^2&\leq \pnorm{\mu-\mu_0}{}^2+\pnorm{\sigma h}{}^2-\pnorm{\sigma h-(\Pi_K(y)-\mu)}{}^2\nonumber\\
	&\stackrel{(\ast)}{\leq} \pnorm{\mu-\mu_0}{}^2+ \pnorm{\sigma h}{}^2 - \pnorm{\sigma h-\Pi_{T_K(\mu)}(\sigma h) }{}^2\nonumber\\
	&\stackrel{(\ast\ast)}{\leq} \pnorm{\mu-\mu_0}{}^2 + \sigma^2 \pnorm{  \Pi_{T_K(\mu)}(h)   }{}^2,
	\end{align}
	where in $(\ast)$ we used that $\Pi_K(y)-\mu \in T_K(\mu)$, and in $(\ast\ast)$ we used Lemma \ref{lem:moreau_thm}. The above inequality appears already in  \cite{bellec2018sharp} and \cite{chatterjee2018matrix}; we have reproduced some details for the convenience of the reader. On the other hand, for any $\mu \in K_S\cap K$ and $\nu \in K$, as $\mu|_{S_\ell} \in K_{S_\ell}$, $K|_{S_\ell}-K_{S_\ell}\subset K|_{S_\ell}$, we have $(\nu-\mu)|_{S_\ell}\subset K|_{S_\ell}$. So $T_K(\mu)\subset  \oplus_{\ell=1}^k T_{K|_{S_\ell}}(0)$. By \cite[Proposition 3.1]{amelunxen2014living}, we have $\E\pnorm{\Pi_{T_K(\mu)}(\sigma h)   }{}^2= \delta_{T_K(\mu)}\leq \sum_{\ell=1}^k \delta_{K|_{S_\ell}}$. Now taking expectation on both sides of (\ref{ineq:orac_ineq_shape_2}) to conclude (\ref{ineq:orac_ineq_shape_1}). Finally invoking Theorem \ref{thm:risk_asymp_consist} to complete the proof. 
\end{proof}

\begin{proof}[Proof of Corollary \ref{cor:risk_asymp_iso}]
	We first prove the main inequality. It is well-known that $\delta_{K_{\uparrow}}=\sum_{j=1}^n (1/j) \leq \log (en)$ (cf. \cite[Eqn. (D.12)]{amelunxen2014living}), so for any partition $S=\{S_\ell\}_{\ell=1}^k \in \mathcal{P}_n$ where $S_\ell$ contains consecutive integers, we have $\sum_{\ell=1}^k \delta_{K|_{S_\ell}}\leq \sum_{\ell=1}^k \log (e\abs{S_\ell})\leq k\log (en/k)$, where the last inequality follows by an application of Jensen's inequality. 
	
	Next we verify that $\mathcal{E}_{K_{\uparrow}}(\mu_0)=\smallo(1)$ for $\mu_{0,n}-\mu_{0,1}=\bigo(1)$. To this end, for given $1\leq k\leq n$, let $\{n_\ell\equiv \ell \floor{n/k}\}_{\ell=0}^k$ and $n_{k+1}\equiv n$. Let $\{S_\ell\equiv (n_{\ell-1}:n_\ell]\}_{\ell=1}^k$  and $S_{k+1}\equiv (n_k:n_{k+1}]$ ($S_{k+1}$ can possibly be $\emptyset$). Now for the given $\mu_0$, define $\mu_0^{k+1} \in \mathcal{M}_{k+1}$ by $\mu_0^{k+1}|_{S_\ell}\equiv \mu_{0,n_{\ell-1}+1}$, $1\leq \ell \leq k+1$. Then
	\begin{align*}
	&\inf_{\mu \in \mathcal{M}_{k+1}} \pnorm{\mu-\mu_0}{}^2 \leq \pnorm{\mu_0^{k+1}-\mu_0}{}^2 \leq \sum_{\ell=1}^{k+1} \big(\mu_{0,n_{\ell}}-\mu_{0,n_{\ell-1}+1}\big)^2\cdot \floor{n/k}\\
	& \leq (\mu_{0,n}-\mu_{0,1})\cdot \sum_{\ell=1}^{k+1} \big(\mu_{0,n_{\ell}}-\mu_{0,n_{\ell-1}+1}\big) (n/k) \leq (\mu_{0,n}-\mu_{0,1})^2(n/k). 
	\end{align*}
	This means
	\begin{align*}
	\mathcal{E}_{K_{\uparrow}}(\mu_0)\lesssim \inf_{1\leq k\leq n}\bigg(\frac{1}{k}+\frac{k \log (en/k)}{m}\bigg)=\smallo(1)
	\end{align*}
	under $m\gg \log n$. The proof of the main inequality is now complete.
	
	Finally we verify that $\bigop(\mathfrak{L}_n/m)$ can be assimilated into the bound $\sigma^2 k\log (en/k)/m$ when $\mu_0 \in \mathcal{M}_k$. To this end, note that  $\mathfrak{L}_n\lesssim \log\big(k \log (en/k)\log n\big)$. As $\min_{1\leq k\leq n} k\log (en/k)=\log (en)\gg 1$, we have $k\log (en/k)\gg \log(k \log (en/k))\vee \log \log n$. So $\mathfrak{L}_n \ll k\log (en/k)$, as desired.
\end{proof}

\subsection{Proofs for Section \ref{section:lasso}}

\begin{proof}[Proof of Theorem \ref{thm:risk_asymp_lasso}]
	We first prove the upper bound for $r_n^2$. By Proposition \ref{prop:est_err_variance}-(1), $\E \err(\sigma)\leq \sigma^2 \delta_0$, so the fixed point equation (\ref{eqn:fixed_pt_eqn}) yields that
	\begin{align}\label{ineq:proof_lasso_1}
	nr_n^2 = \E \err\big(\omega_{m/n}(r_n)\big)\leq \omega_{m/n}^2(r_n) \cdot \delta_0 = \frac{r_n^2+\sigma^2}{m/n}\cdot \delta_0.
	\end{align}
	Solving the inequality yields the desired bound for $r_n^2\leq \bar{r}_n^2$, where $\bar{r}_n^2\equiv \sigma^2\delta_0/(m-\delta_0)_+$. Now we consider (R1). Under $\liminf_n (m/\delta_0)>1$, $r_n^2\leq \bar{r}_n^2=\bigo(1)$ by the upper bound just proven. On the other hand, as $\mathfrak{L}_n \asymp  \log_+\delta_{0}+\log\log (16n)$, we have $\bar{r}_n^2+\bigop(\mathfrak{L}_n/m)=(1+\smallop(1))\bar{r}_n^2+\bigop((\log_+\delta_0+\log\log n)/m)$ and $m\gg \log_+\delta_{0}+\log \log n$ is required. If $\delta_{0}\to \infty$, $\log_+\delta_0/m$ can be assimilated in $\bar{r}_n^2$, and $\liminf_n m/\delta_{0}>1$ already entails $m\gg \log_+ \delta_{0}$. If $\delta_{0}=\bigo(1)$, then $m\gg \log \log n\gg \log_+\delta_{0}$. Consequently, we only need to require $m\gg \log \log n$ to ensure (R1). (R2) is satisfied by Proposition \ref{prop:R2_verify}-(1). Now applying Theorem \ref{thm:risk_asymp} to conclude.
	
	For the last claim, if $K_{(\mathsf{f},\mu_0)}-\mu_0$, then the inequality in (\ref{ineq:proof_lasso_1}) takes equality, and (R2) is degenerate. 
\end{proof}

\section{Proofs for Section \ref{section:est_err_lrt_process}}\label{section:proof_est_lrt_dof}

\subsection{Proofs for Section \ref{section:est_process}}

\begin{proof}[Proof of Lemma \ref{lem:monotone_est_err}]
	\noindent (1). By the arguments up to \cite[Eqn. (10), pp. 2352-2353]{chatterjee2014new}, we have 
	\begin{align*}
	E(\sigma)&\equiv \err^{1/2}(\sigma)= \argmax_{t \geq 0} \bigg\{\sigma\cdot \sup_{\mu \in K: \pnorm{\mu-\mu_0}{}\leq t} \iprod{h}{\mu-\mu_0}-\frac{t^2}{2}\bigg\}\\
	& \equiv \argmax_{t \geq 0} \bigg\{\sigma\cdot M(t)-\frac{t^2}{2}\bigg\}\equiv \argmax_{t \geq 0} V(t;\sigma).
	\end{align*}
	The arguments around \cite[Eqn. (8)-(9), pp. 2352]{chatterjee2014new} showed that the map $t\mapsto V(t;\sigma)$ is strictly concave and drops to $-\infty$ as $t\to\infty$, so admits a unique maximizer. In other words, $E(\sigma)$ is well-defined in this representation. Clearly $M(t)$ is non-decreasing with $M(0)=0$. By \cite[Eqn. (8), pp. 2352]{chatterjee2014new}, $M(t)$ is also concave. Fix $0\leq \sigma_1<\sigma_2<\infty$. Then by definition of $E(\sigma_1),E(\sigma_2)$,
	\begin{align*}
	V\big(E(\sigma_2);\sigma_1\big)&<V\big(E(\sigma_1);\sigma_1\big),\\
	V\big(E(\sigma_2);\sigma_2\big)&>V\big(E(\sigma_1);\sigma_2\big).
	\end{align*}
	This means that
	\begin{align}\label{ineq:monotone_est_err_1}
	V\big(E(\sigma_2);\sigma_2\big)-V\big(E(\sigma_2);\sigma_1\big)>V\big(E(\sigma_1);\sigma_2\big)-V\big(E(\sigma_1);\sigma_1\big).
	\end{align}
	On the other hand, the map
	\begin{align*}
	t\mapsto V(t;\sigma_2)-V(t;\sigma_1) = (\sigma_2-\sigma_1) M(t)
	\end{align*}
	is non-decreasing on $[0,\infty)$ due to the choice $\sigma_1< \sigma_2$. This combined with the comparison inequality in (\ref{ineq:monotone_est_err_1}) necessarily implies that $E(\sigma_2)\geq E(\sigma_1)$, as  desired. 
	
	\noindent (2). That $\err'\geq 0$  follows by (1).
	
	\noindent (\textbf{Step 1}). First we show the result for a convex polytope $K$. 	$\err$ is absolutely continuous with derivative
	\begin{align}\label{ineq:est_err_var_1}
	\err'(\sigma)&=(2/\sigma)\iprod{\Pi_K(\mu_0+\sigma h)-\mu_0}{\J_{\Pi_K}(\mu_0+\sigma h)(\sigma h)}.
	\end{align}
	The above identity holds for any closed convex set $K$. Now as $K$ is a convex polytope, $\J_{\Pi_K}$ is a.e. an orthogonal projection matrix, i.e., $\J_{\Pi_K}=\J_{\Pi_K}^\top$ with $\J_{\Pi_K}=\J_{\Pi_K}^2$.  By \cite[Lemma 2.1-(2)]{han2022high}, we have
	\begin{align}\label{ineq:est_err_var_3}
	\J_{\Pi_K}(\mu_0+\sigma h)^\top (\hat{\mu}_K^{\seq}(\sigma))=\J_{\Pi_K}(\mu_0+\sigma h)^\top (\mu_0+\sigma h).
	\end{align}
	The above identity holds for any closed convex set $K$.
	So using $\J_{\Pi_K}=\J_{\Pi_K}^\top$ a.e., we have a.e.
	\begin{align*}
	\J_{\Pi_K}(\mu_0+\sigma h)(\sigma h)= \J_{\Pi_K}(\mu_0+\sigma h)(\hat{\mu}_K^{\seq}(\sigma)-\mu_0).
	\end{align*}
	Using $\pnorm{\J_{\Pi_K}}{\op}\leq 1$, it follows that a.e.
	\begin{align}\label{ineq:est_err_var_2}
	\pnorm{\J_{\Pi_K}(\mu_0+\sigma h)(\sigma h)}{}
	&\leq  \pnorm{\hat{\mu}_K^{\seq}(\sigma)-\mu_0}{}.
	\end{align}
	Combining (\ref{ineq:est_err_var_1}) and (\ref{ineq:est_err_var_2}), we have a.e.
	\begin{align*}
	\abs{\err'(\sigma)}&\leq \frac{2}{\sigma} \pnorm{\hat{\mu}_K^{\seq}(\sigma)-\mu_0}{}\cdot \pnorm{\J_{\Pi_K}(\mu_0+\sigma h)(\sigma h)}{} \nonumber\\
	& \leq \frac{2\pnorm{\hat{\mu}_K^{\seq}(\sigma)-\mu_0}{}^2 }{\sigma} = \frac{2\err(\sigma)}{\sigma}.
	\end{align*}
	Consequently, 
	\begin{align*}
	\frac{\d}{\d\sigma} \bigg(\frac{\err(\sigma)}{\sigma^2}\bigg) = \frac{\err'(\sigma)-2\err(\sigma)/\sigma}{\sigma^2}\leq 0\quad \textrm{a.e.},
	\end{align*}
	and by the absolute continuity of $\err(\cdot)$, the above display is equivalent to that of $\sigma\mapsto \err(\sigma)/\sigma^2$ being non-increasing. 
	
	\noindent (\textbf{Step 2}). Now we will extend the result to a general closed convex set $K$. Let $\err_K\equiv \err_{(K,\mu_0)}$ for notational convenience. We want to prove that $\sigma\mapsto \err_K(\sigma)/\sigma^2$ is non-increasing, or equivalently
	\begin{align}\label{ineq:est_err_var_4}
	\frac{\err_K(M\sigma)}{(M\sigma)^2}\leq \frac{\err_K(\sigma)}{\sigma^2},\quad \forall \sigma>0, M\geq 1.
	\end{align}
	For any fixed $\sigma>0,M\geq 1$ and $h \in \R^n$ (hidden in the notation $\err_K$), we may find sufficiently large $R>0$ such that $K$ may be replaced by the convex body $K\cap B(R)$ in the above display, where $B(R)\equiv \{x \in \R^n: \pnorm{x}{}\leq R\}$. Now the convex body $K\cap B(R)$ can be approximated in Hausdorff distance by convex polytopes, say, $K_\ell$, cf. \cite[Theorem 1.8.16]{schneider2014convex}, so by Step 1, (\ref{ineq:est_err_var_4}) holds for $K_\ell$. By (an easy consequence of) \cite[Theorem 1.8.8-(2)]{schneider2014convex}, $\Pi_{K_\ell}\to \Pi_{K\cap B(R)}$ pointwise, so passing limits in $\ell$ proves (\ref{ineq:est_err_var_4}) for $K\cap B(R)$, and therefore for a general closed convex set $K$.
\end{proof}

\begin{proof}[Proof of Proposition \ref{prop:est_err_variance}]
	
	\noindent (1). 
	The inequality $\err(\sigma) = \pnorm{\Pi_K(\mu_0+\sigma h)-\Pi_K(\mu_0)}{}^2\leq \sigma^2 \pnorm{h}{}^2$ follows by the fact that $\Pi_K$ is a contraction map. 
	The inequality $\err(\sigma)/\sigma^2\geq \lim_{\sigma \uparrow \infty}\pnorm{\Pi_K(\sigma h)}{}^2/\sigma^2$ follows from the limiting claim by the monotonicity of $\sigma \mapsto \err(\sigma)/\sigma^2$ in Lemma \ref{lem:monotone_est_err}. To prove the limit, using that 
	\begin{align*}
	\big(\pnorm{\Pi_K(\sigma h)}{}-2 \pnorm{\mu_0}{}\big)_+\leq \bigpnorm{\hat{\mu}_K^{\seq}(\sigma)-\mu_0}{}\leq \pnorm{\Pi_K(\sigma h)}{}+2 \pnorm{\mu_0}{},
	\end{align*}
	we have for any $\epsilon>0$, 
	\begin{align*}
	&\frac{1}{\sigma^2}\cdot \Big[(1-\epsilon) \pnorm{\Pi_K(\sigma h)}{}^2-C_\epsilon\cdot \pnorm{\mu_0}{}^2\Big]\\
	&\leq \frac{ \err(\sigma)}{\sigma^2}\leq \frac{1}{\sigma^2}\cdot \Big[(1+\epsilon) \pnorm{\Pi_K(\sigma h)}{}^2+C_\epsilon\cdot \pnorm{\mu_0}{}^2\Big]
	\end{align*}
	holds for some large $C_\epsilon>0$. Now taking $\sigma \uparrow \infty$ followed by $\epsilon\downarrow 0$ to conclude the high noise limit. The low noise limit is proved in \cite[Theorem 1.1]{oymak2016sharp}.
	
	\noindent (2). Note that
	\begin{align*}
	\nabla_h \err(\sigma) 
	= 2\sigma\cdot \J_{\Pi_K}(\mu_0+\sigma h)^\top \big(\Pi_K(\mu_0+\sigma h)-\mu_0\big).
	\end{align*}
	Now Gaussian-Poincar\'e inequality yields that
	\begin{align*}
	&\var\big(\err(\sigma)\big)\leq \E \bigpnorm{  \nabla_h \err(\sigma)  }{}^2 \\
	&\leq 4\sigma^2\cdot \E \bigpnorm{ \J_{\Pi_K}(\mu_0+\sigma h)^\top \big(\Pi_K(\mu_0+\sigma h)-\mu_0\big)  }{}^2 \leq 4\sigma^2 \E\err(\sigma).
	\end{align*}
	The last inequality follows by using $\pnorm{\J_{\Pi_K}}{\op}\leq 1$.

	\noindent (3). 	We apply the Gaussian log-Sobolev inequality in the form stated in Lemma \ref{lem:gaussian_log_sobolev} with $G(h) \equiv \lambda Z\equiv  \lambda \big(\err(\sigma)-\E \err(\sigma) \big)$. Then using the calculations in (2),  $\pnorm{\nabla_h G}{}^2=\lambda^2 \pnorm{\nabla_h \err(\sigma)}{}^2 \leq 4 \sigma^2\lambda^2 \err(\sigma) =4 \sigma^2\lambda^2 (Z+ \E \err(\sigma)) $. This means 
	\begin{align*}
	\lambda \E[Z e^{\lambda Z}]- \E e^{\lambda Z}\cdot \log \E e^{\lambda Z} \leq \frac{1}{2} \E \big[4\sigma^2 \lambda^2 (Z+\E \err(\sigma)) e^{\lambda Z}\big].
	\end{align*}
	Let $m_Z(\lambda) = \E e^{\lambda Z}$ be the moment generating function of $Z$. Then
	\begin{align*}
	\lambda m_Z'(\lambda)-m_Z(\lambda) \log m_Z(\lambda)\leq 2 \sigma^2 \lambda^2 m_Z'(\lambda) +2\sigma^2 \lambda^2 \E \err(\sigma) m_Z(\lambda). 
	\end{align*}
	Now we may use Herbst's argument to conclude a concentration inequality. Some details are given below. Divide on both sides $\lambda^2 m_Z(\lambda)$, we arrive at
	\begin{align*}
	\bigg(\frac{\log m_Z(\lambda)}{\lambda}\bigg)'&=\frac{1}{\lambda} \frac{m_Z'(\lambda)}{m_Z(\lambda)}-\frac{1}{\lambda^2} \log m_Z(\lambda) \\
	&\leq 2\sigma^2 \bigg[ \frac{m_Z'(\lambda)}{m_Z(\lambda)}+\E \err(\sigma)\bigg] = 2\sigma^2 \big(\log m_Z(\lambda) + \lambda \E \err(\sigma)\big)'.
	\end{align*}
	As $\lim_{\lambda \downarrow 0}\log m_Z(\lambda)/\lambda = 0$ and $\log m_Z(0)=0$, integrating the above display yields
	\begin{align*}
	\log m_Z(\lambda)  \leq 2 \sigma^2 \lambda \log m_Z(\lambda)+ 2\sigma^2 \lambda^2 \E \err(\sigma).
	\end{align*}
    This gives the desired gamma behavior of the moment generating function. Conversion to tail bounds is standard, see cf. \cite[pp. 29]{boucheron2013concentration}.
\end{proof}

\subsection{Proofs for Section \ref{section:lrt_process}}

\begin{proof}[Proof of Lemma \ref{lem:lrt_qual}]
	\noindent (1). Note that
	\begin{align*}
	& \pnorm{y-\Pi_K(y)}{}^2  =  \pnorm{\mu_0+\sigma h-\hat{\mu}_K^{\seq}(\sigma)}{}^2 = \err(\sigma)-2\dof(\sigma)+ \sigma^2\pnorm{h}{}^2,
	\end{align*}
	and
	\begin{align*}
	\frac{\d}{\d\sigma} \pnorm{y-\Pi_K(y)}{}^2 &= 2 \iprod{y-\Pi_K(y)}{h} = 2\big(\sigma\pnorm{h}{}^2- \dof(\sigma)/\sigma\big).
	\end{align*}
	So the definition of $\lrt(\cdot)$ in (\ref{def:lrt}) yields
	\begin{align}\label{ineq:lrt_qual_1}
	\lrt(\sigma)& =  2\dof(\sigma)-\err(\sigma),
	\end{align}
	with its derivative $\lrt'(\cdot)$ given by
	\begin{align*}
	\lrt'(\sigma)& = \frac{\d}{\d\sigma}  \pnorm{y-\Pi_{\{\mu_0\}}(y) }{}^2 - \frac{\d}{\d\sigma}  \pnorm{y-\Pi_K(y)}{}^2  = 2\dof(\sigma)/\sigma.
	\end{align*}
	The claim follows by simple algebra combining the above two displays. 
	
	\noindent (2).  As $\bm{0}_K^\ast(s)=\sup_{t \in K} s^\top t$, we have
	\begin{align*}
	&\min_{s \in \R^n} \big\{\pnorm{\sigma h-s}{}^2-2\big(s^\top (\mu_0-\nu)-\bm{0}_K^\ast(s)\big)\big\} \\
	& = \sigma^2 \pnorm{h}{}^2+ \min_{s \in \R^n}\sup_{t \in K}\big\{- 2\iprod{\sigma h+\mu_0-\nu}{s}+\pnorm{s}{}^2+2s^\top t\big\}\\
	& \stackrel{(\ast)}{=}\sigma^2 \pnorm{h}{}^2+ \sup_{t \in K} \min_{s \in \R^n}\big\{- 2\iprod{\sigma h+\mu_0-\nu-t}{s}+\pnorm{s}{}^2\big\}\\
	& = \sigma^2 \pnorm{h}{}^2- \inf_{t \in K} \pnorm{\sigma h+\mu_0-\nu-t}{}^2\\
	& = \sigma^2 \pnorm{h}{}^2 -\dis^2\big(\sigma h+\mu_0-\nu,K\big)\\
	& \stackrel{(\ast\ast)}{=} \sigma^2 \pnorm{h}{}^2 -\dis^2\big(\sigma h+\mu_0,K\big)=\lrt(\sigma),
	\end{align*}
	Here (a) in $(\ast)$ we used Sion's min-max theorem (cf. Lemma \ref{lem:sion_minmax}), as the range of $s$ can be restricted to a compact set due to the fact that objective function is quadratic in $s$ given all other quantities; (b) in $(\ast\ast)$ we used Lemma \ref{lem:proj_split} with $(\mu,\xi)$ therein replaced by $(-\nu,\sigma h+\mu_0)$ under the condition that (i) $K+\nu\subset K$ and (ii) $-\nu +\Pi_K(\mu_0+\sigma h) \in K$. 
	
	\noindent (3). Note that by (1),
	\begin{align*}
	\frac{\d}{\d\sigma}\bigg(\frac{\lrt(\sigma)}{\sigma}\bigg)& = \frac{\lrt'(\sigma)\sigma-\lrt(\sigma)}{\sigma^2}=\frac{\err(\sigma)}{\sigma^2}.
	\end{align*}
	By Lemma \ref{lem:monotone_est_err}-(2), the right hand side of the above display is non-increasing. 
	
	\noindent (4). Note that by (1) and Proposition \ref{prop:lrt_prop}-(2) below,
	\begin{align*}
	\frac{\d}{\d\sigma}\bigg(\frac{\lrt(\sigma)}{\sigma^2}\bigg)& =\frac{\lrt'(\sigma)\sigma- 2 \lrt(\sigma)   }{\sigma^3}= \frac{\err(\sigma)-\lrt(\sigma)}{\sigma^3} \leq 0.
	\end{align*}
	The claim follows. 
\end{proof}

\begin{proof}[Proof of Proposition \ref{prop:lrt_prop}]
	\noindent (1). The claim follows by using that
	\begin{align*}
	\E \lrt(\sigma) = 2\sigma^2 \E \dv \hat{\mu}_K^{\seq}(\sigma)- \E \err(\sigma)
	\end{align*}
	and $\E \dv \hat{\mu}_K^{\seq}(\sigma)\leq n$ (cf. \cite[Corollary 1]{meyer2000degrees}).

	\noindent (2). The lower bound follows by (\ref{ineq:lrt_qual_1}) and the left hand side of the first inequality in Proposition \ref{prop:dof}-(3). For the upper bound, by (\ref{ineq:lrt_qual_1}) and Proposition \ref{prop:dof}-(3), 
	\begin{align*}
	\lrt(\sigma) -\err(\sigma) &= 2\big(\dof(\sigma)- \err(\sigma)\big)\leq 2\sigma \int_0^\sigma \bigg(\frac{\err(\tau)}{\tau^2}-\frac{\err(\sigma)}{\sigma^2}\bigg)\,\d{\tau}.
	\end{align*}
	Now suppose further that $K$ is a closed convex cone and we prove the announced upper bound in this setting.  Choosing
	\begin{align*}
	s=\sigma h+\mu_0- \Pi_K(\mu_0+\sigma h) = \Pi_{K^\ast}(\mu_0+\sigma h) \in K^\ast
	\end{align*}
	in Lemma \ref{lem:lrt_qual}-(2), we obtain 
	\begin{align*}
	\lrt(\sigma)& \leq \pnorm{\Pi_K(\mu_0+\sigma h)- \mu_0}{}^2+2\iprod{ \Pi_K(\mu_0+\sigma h) -\mu_0-\sigma h  }{\mu_0-\nu}\\
	&= \err(\sigma)+2\iprod{ \hat{\mu}_K^{\seq}(\sigma)-\mu_0-\sigma h}{\mu_0-\nu}.
	\end{align*}
	The claimed inequality follows.

	\noindent (3). By \cite[Lemma 2.1-(1)]{han2022high}, $\nabla_y \pnorm{y-\Pi_K(y)}{}^2 = 2(y-\Pi_K(y))$, so $
	\nabla_h \pnorm{y-\Pi_K(y)}{}^2 = 2\sigma (y-\Pi_K(y))$.
	Hence
	\begin{align}\label{ineq:lrt_prop_1}
	\nabla_h \lrt(\sigma) &= 2\sigma\Big[\big(y-\mu_0\big)-\big(y-\Pi_K(y)\big)\Big]=2\sigma \big(\mu_K^{\seq}(\sigma)-\mu_0\big).
	\end{align}
	By Gaussian-Poincar\'e inequality,
	\begin{align*}
	\var\big(\lrt(\sigma)\big)\leq \E \pnorm{   \nabla_h \lrt(\sigma)  }{}^2\leq 4\sigma^2 \E \err(\sigma). 
	\end{align*}
	\noindent (4). Fix $\lambda \in \R$. Let 
	\begin{align*}
	G(h)\equiv \lambda \cdot \frac{\lrt(\sigma)-\E \lrt (\sigma)}{\sigma\cdot \E^{1/2} \err(\sigma)}\equiv \lambda \cdot Z.
	\end{align*}
	Then
	\begin{align*}
	\nabla_h G(h)& = \frac{\lambda}{\sigma \cdot \E^{1/2} \err(\sigma)}\nabla_h \lrt (\sigma),
	\end{align*}
	so by (\ref{ineq:lrt_prop_1}),
	\begin{align*}
	\pnorm{\nabla_h G(h)}{}^2&\leq   \frac{\lambda^2}{\sigma^2 \E \err(\sigma)}\cdot (4\sigma^2 \err(\sigma)) = 4\lambda^2\cdot  (\err(\sigma)/\E \err(\sigma)). 
	\end{align*}
	By Gaussian log-Sobolev inequality (cf. Lemma \ref{lem:gaussian_log_sobolev}), we have for any $t\geq 0$, 
	\begin{align*}
	\mathrm{Ent}(e^{t G})\leq \frac{1}{2}\E\big[\pnorm{\nabla_h \big(tG(h)\big) }{}^2 e^{tG (h)}\big] = \frac{t^2}{2}\cdot \E\big[\Gamma\cdot e^{t G}\big],
	\end{align*}
	where $\Gamma\equiv 4\lambda^2 \big(\err(\sigma)/\E \err(\sigma)\big)$. By the `exponential Poincar\'e inequality' proved in \cite{bobkov1999exponential} (cf. Lemma \ref{lem:expo_poincare_ineq}),
	\begin{align*}
	\E e^G\leq \E e^{\Gamma} = \E \exp\big(4\lambda^2 \err(\sigma)/\E \err(\sigma)\big).
	\end{align*}
	By Proposition \ref{prop:est_err_variance}-(3), with $u \equiv 4\lambda^2/\E \err(\sigma)<1/2\sigma^2$, we have
	\begin{align*}
	\E \exp\big(\lambda\cdot Z\big)& \leq \E \exp\big(u\cdot \err(\sigma)\big)\\ &\leq  \exp\big(u\cdot \E \err(\sigma)\big)\cdot \exp\bigg(\frac{2\sigma^2 u^2 \E \err(\sigma)}{1-2\sigma^2 u}\bigg)\\
	& =\exp\bigg(\frac{u \E \err(\sigma)}{1-2\sigma^2 u}\bigg) = \exp\bigg(\frac{4\lambda^2}{1-8 \big(\sigma^2/\E \err(\sigma)\big) \lambda^2 }\bigg).
	\end{align*} 
	Apply Lemma \ref{lem:tail_bound_generic} with $C_0\equiv 4$ and $a\equiv 8(\sigma^2/\E \err(\sigma))$ therein to conclude. 
\end{proof}

\subsection{Proofs for Section \ref{section:dof}}

\begin{proof}[Proof of Proposition \ref{prop:dof}]
	\noindent (1). The claim follows from Stein's identity.
	
	\noindent (2). By (\ref{ineq:lrt_qual_1}), $\dof(\sigma) = (\lrt(\sigma)+\err(\sigma))/2$. The first claim now follows from Lemmas \ref{lem:monotone_est_err}-(2) and \ref{lem:lrt_qual}-(4). For the second claim, $\dof(M\sigma)\leq M^2 \dof(\sigma^2)$ follows immediately from the monotonicity of $\sigma \mapsto \dof(\sigma)/\sigma^2$, so we are left to prove $\dof(\sigma)\leq \dof(M\sigma)$. This can be seen as
	\begin{align*}
	\dof(\sigma) = \frac{1}{2}\big(\lrt(\sigma)+\err(\sigma)\big)\leq  \frac{1}{2}\big(\lrt(M\sigma)+\err(M\sigma)\big) = \dof(M\sigma),
	\end{align*}
	where the inequality follows by (i) $\err(\sigma)\leq \err(M\sigma)$ due to Proposition \ref{prop:est_err_stability}, and (ii) $\lrt(\sigma)\leq M\lrt(\sigma)\leq \lrt(M\sigma)$ due to Proposition \ref{prop:lrt_stability}.

	\noindent (3). 	We prove the inequality first. The left side inequality follows from basic properties of projection, as
	\begin{align*}
	0\leq \iprod{y-\hat{\mu}_K^{\seq}(\sigma)}{\hat{\mu}_K^{\seq}(\sigma)-\mu_0}= \dof(\sigma)-\err(\sigma).
	\end{align*}
	For the right side inequality, note that
	\begin{align*}
	\frac{\d}{\d\sigma} \iprod{\hat{\mu}_K^{\seq}(\sigma)-\mu}{h} &= \frac{1}{\sigma^2}\iprod{\J_{\Pi_K}^\top(\mu_0+\sigma h)(\sigma h)}{\sigma h}\\
	& = \frac{1}{\sigma^2}\iprod{\J_{\Pi_K}^\top(\mu_0+\sigma h)(\hat{\mu}_K^{\seq}(\sigma)-\mu_0)}{\sigma h}\quad (\hbox{by (\ref{ineq:est_err_var_3})})\\
	& = \frac{\err'(\sigma)}{2\sigma}\quad (\hbox{by (\ref{ineq:est_err_var_1})}),
	\end{align*}
	so by the absolute continuity of $\sigma\mapsto \hat{\mu}_K^{\seq}(\sigma)$, Lebesgue's fundamental theorem of calculus yields that
	\begin{align*}
	\iprod{\hat{\mu}_K^{\seq}(\sigma)-\mu}{h} &= \int_0^\sigma \frac{\d}{\d\tau} \iprod{\hat{\mu}_K^{\seq}(\tau)-\mu}{h}\,\d{\tau} = \int_0^\sigma \frac{\err'(\tau)}{2\tau}\,\d{\tau} \leq \int_0^\sigma \frac{\err(\tau)}{\tau^2}\,\d{\tau}.
	\end{align*}
	Here the last inequality follows from Lemma \ref{lem:monotone_est_err}-(2), completing the proof of the announced inequality. 
	
	For the second claim, by taking expectation, we have 
	\begin{align*}
	&\frac{\E \err(\sigma)}{\sigma^2}\leq \E \dv\hat{\mu}_K^{\seq}(\sigma)\leq \frac{1}{\sigma} \int_0^\sigma \frac{\E\err(\tau)}{\tau^2}\,\d{\tau} \leq \delta_{T_K(\mu_0)}.
	\end{align*}
	The last inequality in the above display follows from Proposition \ref{prop:est_err_variance}-(1) and the monotonicity of $\sigma \mapsto \E \err(\sigma)/\sigma^2$ proved in Lemma \ref{lem:monotone_est_err}-(2). Now taking $\sigma \downarrow 0$, the left most side of the inequality converges to $\delta_{T_K(\mu_0)}$. 
	
	\noindent (4). Both the variance bound and the exponential inequality is a simple consequence of $\dof(\sigma) = (\lrt(\sigma)+\err(\sigma))/2$ combined with the variance bounds and exponential inequalities for $\err,\lrt$ proved in Propositions \ref{prop:est_err_variance} and \ref{prop:lrt_prop}.
\end{proof}

\subsection{Proofs for Section \ref{section:exp_ineq_err_lrt}}

\begin{proof}[Proof of Proposition \ref{prop:est_err_lrt_sup}]
	We only prove the inequality for $H \in \{\err,\lrt\}$, as the inequality for $H=\dof$ follows from the identity $\dof(\sigma) = (\lrt(\sigma)+\err(\sigma))/2$.

	\noindent \textbf{(Step 1)}. We will show that there exists a universal constant $C>0$ such that for $H\in \{\err,\lrt \}$, and any $u>0,t\geq 1$,
	\begin{align}\label{ineq:est_err_lrt_sup_1}
	&\Prob\bigg[\sup_{\sigma \in [u,2u]}\big(\abs{H(\sigma)-\E H(\sigma)}\big)\geq C\big( u\cdot \E^{1/2} \err(u)\cdot \sqrt{t}+ u^2\cdot t \big)\bigg]\nonumber\\
	&\qquad\qquad \leq C\cdot(1+\delta_{T_K(\mu_0)}) \cdot e^{-t/C},
	\end{align}
	
	To prove (\ref{ineq:est_err_lrt_sup_1}), for some $N\in \N$ to be determined later, let $\sigma(j)\equiv  u(1+j/N)$ for $1\leq j\leq N$. Then for some absolute constant $L_1>0$,
	\begin{align*}
	&\sup_{\sigma \in [u,2u]}\abs{H(\sigma)-\E H(\sigma)}\leq \max_{1\leq j\leq N} \abs{H(\sigma(j))-\E H(\sigma(j))}\\
	&\qquad\qquad\qquad + \sup_{\substack{\sigma,\sigma' \in [u,2u],\\ \abs{\sigma-\sigma'}\leq u/N}}\big(\abs{H(\sigma)-H(\sigma')}+\E\abs{H(\sigma)-H(\sigma')} \big)\\
	& \leq \max_{1\leq j\leq N} \abs{H(\sigma(j))-\E H(\sigma(j))}+L_1 N^{-1}\cdot \big(H(u)+\E H(u)\big).
	\end{align*}
	In the last inequality in the above display, we used the following estimates that hold for any $\sigma,\sigma' \in [u,2u]$, $\sigma\leq \sigma'$ with $\abs{\sigma-\sigma'}\leq u/N$: (i) By Lemma \ref{lem:monotone_est_err}-(2) and Proposition \ref{prop:est_err_stability},
	\begin{align*}
	\abs{\err(\sigma)-\err(\sigma')}&\leq \sup_{\tau \in [\sigma,\sigma']} \abs{\err'(\tau)}\abs{\sigma-\sigma'}\\
	&\leq \sup_{\tau \in [\sigma,\sigma']}\frac{2 \err(\tau)}{\tau}\cdot \frac{u}{N} \leq \frac{2\err(2u)}{N}\lesssim \frac{\err(u)}{N}.
	\end{align*}
	(ii) By Lemma \ref{lem:lrt_qual}-(1) and Proposition \ref{prop:lrt_prop}-(2), we have $\abs{\lrt'(\tau)}\leq \tau^{-1}\abs{\lrt(\tau)+\err(\tau)}\leq 2\tau^{-1}\lrt(\tau)$, so using Lemma \ref{lem:lrt_qual}-(3) and the same argument as above, 
	\begin{align*}
	\abs{\lrt(\sigma)-\lrt(\sigma')} \leq \sup_{\tau \in [\sigma,\sigma']}\frac{2 \lrt(\tau)}{\tau}\cdot \frac{u}{N}\lesssim \frac{\lrt(u)+\err(u)}{N}.
	\end{align*}
	Now using Proposition \ref{prop:est_err_variance}-(3) and Proposition \ref{prop:lrt_prop}-(2)(4), we have with probability at least $1-L_2 N e^{-t/L_2}$, it holds that
	\begin{align*}
	\sup_{\sigma \in [u,2u]} \abs{H(\sigma)-\E H(\sigma)} \lesssim  u\cdot \E^{1/2} \err(2u)\cdot \sqrt{t}+ u^2\cdot t + N^{-1} u\int_0^u \frac{\E \err(\tau)}{\tau^2}\,\d{\tau}.
	\end{align*}
	The term $\E \err(2u)$ can be replaced by $\E \err(u)$ at the cost of an absolute multiplicative factor by the stability estimate in Proposition \ref{prop:est_err_stability}. The claimed inequality (\ref{ineq:est_err_lrt_sup_1}) follows by choosing 
	\begin{align*}
	N\geq 1+\delta_{T_K(\mu_0)} \geq 1+ \frac{\delta_{T_K(\mu_0)} }{\delta_K^{1/2}\sqrt{t}\vee t}\stackrel{(\ast)}{\geq } 1 + \frac{ \int_0^u \big(\E \err(\tau)/\tau^2\big)\,\d{\tau}}{\max\{\E^{1/2}\err(u)\sqrt{t}, u t\}}.
	\end{align*}
	Here in $(\ast)$ we  used $\delta_K\leq \E \err(\tau)/\tau^2 \leq \delta_{T_K(\mu_0)}$ (cf. Proposition \ref{prop:est_err_variance}-(1)).
	
	\noindent \textbf{(Step 2)}. Let $L\equiv \ceil{\log_2(M_0/\epsilon_0)}$. For $0\leq \ell \leq L$, let $\sigma_\ell\equiv \epsilon_0\cdot 2^{\ell-1}$. By (\ref{ineq:est_err_lrt_sup_1}), with probability at least 
	$1-C\cdot (1+\delta_{T_K(\mu_0)})\cdot e^{-t/C}$, we have
	\begin{align*}
	\abs{H(\sigma)-\E H(\sigma)}\leq C\big(\sigma\cdot \E^{1/2} \err(\sigma)\cdot \sqrt{t}+\sigma^2 t\big),\quad \forall \sigma \in [\sigma_\ell,\sigma_{\ell+1}].
	\end{align*}
	The claim follows by a union bound.
\end{proof}

\appendix

\section{Some technical tools}

The following version of convex Gaussian min-max theorem, proved quite easily using Gordon's min-max theorem \cite{gordan1985some,gordon1988milman}, is taken from \cite[Theorem 6.1]{thrampoulidis2018precise} or \cite[Theorem 5.1]{miolane2021distribution}.

\begin{theorem}[Convex Gaussian Min-Max Theorem]\label{thm:CGMT}
	Suppose $D_u \in \R^n, D_v \in \R^m$ are compact sets, and $Q: D_u\times D_v \to \R$ is continuous. Let $G=(G_{ij})_{i \in [n],j\in[m]}$ with $G_{ij}$'s i.i.d. $\mathcal{N}(0,1)$, and $g \sim \mathcal{N}(0,I_n)$, $h \sim \mathcal{N}(0,I_m)$ be independent Gaussian vectors. Define
	\begin{align}
	\Phi^{\textrm{p}} (G)& = \min_{u \in D_u}\max_{v \in D_v} \Big( u^\top G v + Q(u,v)\Big), \nonumber\\
	\Phi^{\textrm{a}}(g,h)& = \min_{u \in D_u}\max_{v \in D_v} \Big(\pnorm{v}{} g^\top u + \pnorm{u}{} h^\top v+ Q(u,v)\Big).
	\end{align}
	Then the following hold.
	\begin{enumerate}
		\item For all $t \in \R$, 
		\begin{align*}
		\Prob\big(\Phi^{\textrm{p}} (G)\leq t\big)\leq 2 \Prob\big(\Phi^{\textrm{a}}(g,h)\leq t\big).
		\end{align*}
		\item If $(u,v)\mapsto u^\top G v+ Q(u,v)$ satisfies the conditions of Sion's min-max theorem (cf. Lemma \ref{lem:sion_minmax}) the pair $(D_u,D_v)$ a.s. (for instance, $D_u,D_v$ are convex, and $Q$ is convex-concave), then 
		\begin{align*}
		\Prob\big(\Phi^{\textrm{p}} (G)\geq t\big)\leq 2 \Prob\big(\Phi^{\textrm{a}}(g,h)\geq t\big).
		\end{align*}
	\end{enumerate}
\end{theorem}

The following version of Gaussian log-Sobolev inequality will be useful. 
\begin{lemma}\label{lem:gaussian_log_sobolev}
	Let $G:\R^n \to \R^n$ be a differentiable function. Then 
	\begin{align*}
	\mathrm{Ent}(e^G)=\E [G(h)e^{G(h)}]-\E e^{G(h)}\cdot  \log \E\big[e^{G(h)}\big] \leq \frac{1}{2}\E\big[\pnorm{\nabla G(h)}{}^2 e^{G(h)}\big].
	\end{align*}
\end{lemma}
\begin{proof}
	The standard version of Gaussian log-Sobolev inequality states that for a differentiable function $f: \R^n \to \R^n$, $\mathrm{Ent}(f^2(h)) \leq 2 \E \pnorm{\nabla f(h)}{}^2$, where $\mathrm{Ent}(f(h))\equiv \E[f(h)\log f(h)]-\E[f(h)]\log \E[f(h)]$; see \cite[Theorem 5.4]{boucheron2013concentration}, or \cite[Theorem 2.5.6]{gine2015mathematical}. Now we take $f^2 = e^G$ to conclude by noting $\nabla f = \nabla e^{G/2} = \nabla G\cdot e^{G/2}/2$.
\end{proof}

The following `exponential Poincar\'e inequality' due to \cite{bobkov1999exponential} will be useful.
\begin{lemma}[Exponential Poincar\'e Inequality]\label{lem:expo_poincare_ineq}
Let $G:\R^n \to \R^n, \Gamma:\R^n \to \R$ be measurable functions such that $\E G(h)=0$, and 
\begin{align*}
\mathrm{Ent}(e^{t G}) \leq \frac{t^2}{2}\E [\Gamma^2(h)\cdot e^{tG(h)}]
\end{align*}
for all $t\geq 0$. Then $\E e^G \leq \E e^{\Gamma^2}$. 
\end{lemma}
\begin{proof}
We replicate some proof of \cite[Theorem 2.1]{bobkov1999exponential} for the convenience of the reader. Let $\beta(t)\equiv \log \E e^{t\Gamma^2}$. We claim that 
\begin{align}\label{ineq:expo_poincare_ineq_1}
\mathrm{Ent}\big(e^{tG}\big) \leq \beta(t^2)\cdot \E e^{tG}. 
\end{align}
To prove (\ref{ineq:expo_poincare_ineq_1}), note that $\E \exp\big[(t \Gamma)^2-\beta(t^2)\big]=1$, so by variational characterization of $\mathrm{Ent}(f)=\sup\{\E f g: \E \exp(g)\leq 1\}$, we have
\begin{align*}
\mathrm{Ent}\big(e^{tG}\big)\geq \E \Big(\big[  (t \Gamma)^2-\beta(t^2) \big] e^{tG}\Big) \geq 2 \mathrm{Ent}\big(e^{tG}\big) - \beta(t^2) \E e^{tG}.
\end{align*}
Here the last inequality follows from the assumption. The claimed inequality (\ref{ineq:expo_poincare_ineq_1}) follows. Now let $u(t)$ be defined by $\E e^{tG} \equiv e^{t u(t)}$. Upon differentiation on both sides of this equality, some calculations lead to $\mathrm{Ent}(e^{tG})= t^2 u'(t) e^{t u(t)} = t^2 u'(t) \E e^{t G}$. Combined with (\ref{ineq:expo_poincare_ineq_1}), we arrive at the key inequality
\begin{align}\label{ineq:expo_poincare_ineq_2}
u'(t) \leq \frac{\beta(t^2)}{t^2}. 
\end{align}
It is easy to verify that $\beta$ is non-negative, non-decreasing and convex on $[0,\infty)$ so (\ref{ineq:expo_poincare_ineq_2}) leads to $u'(t)\leq \beta(1)$ for all $t \in [0,1]$. As $u(0)=\lim_{t \downarrow 0} t^{-1}\log \E e^{t G} = \E G = 0$, we obtain $u(1)\leq \beta(1)$, so $\E e^G = e^{u(1)}\leq e^{\beta(1)} = \E e^{\Gamma^2}$, as desired. 
\end{proof}

We need the following min-max theorem due to Sion \cite{sion1958general}.

\begin{lemma}[Sion's min-max theorem]\label{lem:sion_minmax}
Let $X$ be a compact convex subset of a linear topological space and $Y$ a convex subset of a linear topological space. If $f$ is a real-valued function on $X\times Y$ satisfying:
\begin{enumerate}
	\item $y\mapsto f(x,y)$ is upper-semicontinuous and quasi-concave for all $x \in X$;
	\item$x\mapsto f(x,y)$ is lower-semicontinuous and quasi-convex  for all $y \in Y$.
\end{enumerate}
Then $\min_{x \in X} \sup_{y \in Y} f(x,y)= \sup_{y \in Y} \min_{x \in X}f(x,y)$.
\end{lemma}

\section{Auxiliary lemmas}

\begin{lemma}\label{lem:gaussian_trunc}
Let $h\sim \mathcal{N}(0,1)$ and $\varphi,\Phi$ be the normal p.d.f. and c.d.f. respectively. Then for $x\geq 0$, the following hold:
\begin{enumerate}
	\item $\E h \bm{1}_{h\geq x} =\E h\bm{1}_{h\geq -x}= \varphi(x)$.
	\item $\E h^2 \bm{1}_{h\geq x} = x\varphi(x)+1-\Phi(x)$ and $\E h^2 \bm{1}_{h\geq -x} = \Phi(x)-x\varphi(x)$. 
\end{enumerate}
\end{lemma}
\begin{proof}
The facts that $\E h \bm{1}_{h\geq x} = \varphi(x)$ and $\E h^2 \bm{1}_{h\geq x} = x\varphi(x)+1-\Phi(x)$ are proven in \cite[Lemma 6.1]{han2022high}. To see the remaining equalities, first note that $\E h=0$ yields that $\E h\bm{1}_{h\geq -x} = -\E h\bm{1}_{h<-x}$. By symmetry of $h$, the latter equals $\E h\bm{1}_{h>x}$. As $h$ charges probability 0 at any single point, the equality $\E h \bm{1}_{h\geq x} =\E h\bm{1}_{h\geq -x}$ follows. Next using a similar argument as above, but now using that $\E h^2 =1$, we have $\E h^2\bm{1}_{h\geq -x} = 1- \E h^2 \bm{1}_{h<-x} = 1- \E h^2 \bm{1}_{h>x}$. 
\end{proof}

\begin{lemma}\label{lem:mill_ratio}
Let $\varphi,\Phi$ be the normal p.d.f. and c.d.f. respectively. Then 
\begin{align*}
\frac{1}{x}-\frac{1}{x^3}\leq \frac{1-\Phi(x)}{\varphi(x)}\leq \frac{1}{x}-\frac{1}{x^3}+\frac{3}{x^5}.
\end{align*}
\end{lemma}
\begin{proof}
This result is well known. We include a (simple) proof below for the convenience of the reader, based on iterative applications of the identity $\varphi'(t)=-t\varphi(t)$ and integration by parts:
\begin{align*}
1-\Phi(x) &= \int_x^\infty \varphi(t)\,\d{t} = - \int_x^\infty \frac{\varphi'(t)}{t}\,\d{t} = - \bigg[\frac{\varphi(t)}{t}\bigg\vert_{x}^\infty+\int_x^\infty \frac{\varphi(t)}{t^2}\,\d{t}\bigg]\\
& = \frac{\varphi(x)}{x}+ \int_x^\infty \frac{\varphi'(t)}{t^3}\,\d{t} = \frac{\varphi(x)}{x} + \bigg[\frac{\varphi(t)}{t^3}\bigg\vert_{x}^\infty+ \int_x^\infty \frac{3\varphi(t)}{t^4}\,\d{t}   \bigg]\\
& = \frac{\varphi(x)}{x}-\frac{\varphi(x)}{x^3}+ 3\int_x^\infty \frac{\varphi(t)}{t^4}\,\d{t}.
\end{align*}
The left inequality follows. For the right inequality, we perform the integration by parts again: 
\begin{align*}
\int_x^\infty \frac{\varphi(t)}{t^4}\,\d{t} = - \int_x^\infty \frac{\varphi'(t)}{t^5}\,\d{t} = -\bigg[\frac{\varphi(t)}{t^5}\bigg\lvert_{x}^\infty+5 \int_x^\infty \frac{\varphi(t)}{t^6}\,\d{t}\bigg]\leq \frac{\varphi(x)}{x^5}. 
\end{align*}
The right inequality follows. 
\end{proof}

\begin{lemma}\label{lem:tail_bound_generic}
	Suppose the random variable $Z$ satisfies the following: There exist $C_0>0$ and $a>0$ such that for $\lambda \in \R$ with $\lambda^2<1/a$, we have
	\begin{align*}
	\E \exp\big(\lambda\cdot Z\big)\leq \exp\bigg[\frac{C_0 \lambda^2}{1- a\lambda^2 }\bigg].
	\end{align*}
	Then there exists some universal constant $L>0$ such that
	\begin{align*}
	\Prob(\abs{Z}>t)\leq L\cdot \exp\bigg[-L^{-1}\bigg( \frac{t^2}{C_0}\wedge \frac{t}{a^{1/2}}\bigg)\bigg],\quad t\geq 0.
	\end{align*}
\end{lemma}
\begin{proof}
	A standard Chernoff bound yields that for any $t\geq 0$, 
	\begin{align*}
	\Prob\big(Z>t\big)\leq \exp\bigg[-\lambda t+ \frac{C_0 \lambda^2}{1- a\lambda^2 }\bigg], \quad 0\neq \lambda^2<1/a.
	\end{align*}
	We shall make different choices of $\lambda$ according to different regimes of $\lambda$:
	\begin{enumerate}
		\item Suppose $\lambda^2\leq 1/(2a)$. With the choice $\lambda=t/(4C_0)$, we have
		\begin{align*}
		\Prob(Z>t)\leq \exp\big(-\lambda t+ 2C_0 \lambda^2\big) = \exp(-t^2/(8C_0)).
		\end{align*}
		This holds in the regime $t^2/(16C_0^2)<1/(2a)$, i.e., $t^2\leq 8C_0^2/a$.
		\item Suppose $1/(2a)<\lambda^2< 1/a$. Then we simply take $\lambda^2 = 0.6/a$ which yields, with $c_0=0.6$,
		\begin{align*}
		\Prob(Z>t)\leq \exp\bigg(-c_0^{1/2}\cdot \frac{t}{a^{1/2}}+\frac{C_0 c_0}{1-c_0}\cdot \frac{1}{a}\bigg).
		\end{align*}
		As we only need to consider the regime $t^2\gtrsim C_0^2/a$, by enlarging $C_0$ by an appropriate absolute constant times, we may absorb the second term into the first one by adjusting absolute constants. 
	\end{enumerate}
	Summarizing the two regimes, we arrive at the desired inequality for $t\geq 0$. Similar considerations apply to $t<0$.
\end{proof}

\begin{lemma}\label{lem:moreau_thm}
Let $K\subset \R^n$ be a closed convex cone, and $h \in \R^n$. Then $h=\Pi_K(h)+\Pi_{K^\ast}(h)$ and $\pnorm{h}{}^2= \pnorm{\Pi_K(h)}{}^2+\pnorm{\Pi_{K^\ast}(h)}{}^2$. 
\end{lemma}
\begin{proof}
This is known as Moreau's theorem (cf. \cite[Theorem 31.5]{rockafellar1997convex}).
\end{proof}

\begin{lemma}\label{lem:proj_split}
Let $K$ be a closed convex set, and $\mu,\xi$ be such that $K-\mu \subset K$ and $\mu+\Pi_K(\xi) \in K$. Then $\dis^2(\mu+\xi,K)=\dis^2(\xi,K)$. 
\end{lemma}
\begin{proof}
It suffices to show that $\Pi_K(\mu+\xi)=\mu+\Pi_K(\xi)$. Recall $z= \Pi_K(\mu+\xi) \in K$ if and only if $\iprod{\mu+\xi-z}{v-z}\leq 0$ for all $v \in K$. Now with $z\equiv \mu+\Pi_K(\xi)$, we only need to verify that $\iprod{\mu+\xi-(\mu+\Pi_K(\xi))}{v-(\mu+\Pi_K(\xi))}\leq 0$ for all $v \in K$, or equivalently $\iprod{\xi-\Pi_K(\xi)}{(v-\mu)-\Pi_K(\xi)}\leq 0$ for all $v \in K$. This holds as $v-\mu \subset K$ by the assumption. 
\end{proof}

\begin{lemma}\label{lem:proj_l1_ball}
	For any $\lambda>0$, and $x \in \R^n$,
	\begin{align*}
	\big(\Pi_{K_{\ell_1,\lambda}}(x)\big)_i = \sign(x_i)\big(\abs{x_i}-\mu(x)_+\big)_+,\quad i=1,\ldots,n.
	\end{align*}
	where $\mu(x) \in \R$ is the unique solution to $\sum_{i=1}^n (\abs{x_i}-\mu(x))_+=\lambda$. 
\end{lemma}
\begin{proof}
	This result is a simple consequence of the KKT conditions. We include a proof for completeness. Note that $\Pi_{K_{\ell_1,\lambda}}(x)\equiv y^\ast\equiv \argmin_{\pnorm{y}{1}\leq \lambda}\pnorm{x-y}{}^2/2$, so with the Lagrangian function $
	L(y,\mu) \equiv \frac{1}{2}\pnorm{x-y}{}^2+\mu\big(\pnorm{y}{1}-\lambda\big)$,
	the KKT conditions become (i) (stationarity) $(y^\ast-x)+\mu\cdot \partial\pnorm{\cdot}{1}(y^\ast)=0$, (ii) (primal feasibility) $\pnorm{y^\ast}{1}\leq \lambda$, (iii) (dual feasibility) $\mu\geq 0$, and (iv) (complementary slackness) $\mu(\pnorm{y^\ast}{1}-\lambda)=0$. So from (i) we may solve (v) $y^\ast_i = \sign(x_i)(\abs{x_i}-\mu)_+$. Consider two cases:
	\begin{enumerate}
		\item If $\mu=0$, (v) yields that $y^\ast =x$, so by (ii) $y^\ast \in K_{\ell_1,\lambda}$.
		\item If $\mu>0$, (iv) yields that $\pnorm{y^\ast}{1}=\lambda$, i.e., $\sum_{i=1}^n \abs{\sign(x_i)(\abs{x_i}-\mu)_+}=\lambda$, or equivalently, $\sum_{i=1}^n (\abs{x_i}-\mu)_+=\lambda$. 
	\end{enumerate}
	As the equation $\sum_{i=1}^n (\abs{x_i}-\mu)_+=\lambda$ has a unique solution $\mu\geq 0$ if and only if $\pnorm{x}{1}\geq \lambda$, and has a unique solution $\mu \in \R$ for all $x \in \R^n$, we may also write the first case as $y^\ast_i = \sign(x_i)(\abs{x_i}-\mu_+)_+$, where $\mu \in \R$ is the solution to $\sum_{i=1}^n (\abs{x_i}-\mu)_+=\lambda$. The proof is complete. 
\end{proof}


\bibliographystyle{amsalpha}
\bibliography{mybib}

\providecommand{\bysame}{\leavevmode\hbox to3em{\hrulefill}\thinspace}
\providecommand{\MR}{\relax\ifhmode\unskip\space\fi MR }
\providecommand{\MRhref}[2]{%
  \href{http://www.ams.org/mathscinet-getitem?mr=#1}{#2}
}
\providecommand{\href}[2]{#2}
\begin{thebibliography}{CRPW12}

\bibitem[ALMT14]{amelunxen2014living}
Dennis Amelunxen, Martin Lotz, Michael~B. McCoy, and Joel~A. Tropp,
  \emph{Living on the edge: phase transitions in convex programs with random
  data}, Inf. Inference \textbf{3} (2014), no.~3, 224--294. \MR{3311453}

\bibitem[Bel18]{bellec2018sharp}
Pierre~C. Bellec, \emph{Sharp oracle inequalities for {L}east {S}quares
  estimators in shape restricted regression}, Ann. Statist. \textbf{46} (2018),
  no.~2, 745--780. \MR{3782383}

\bibitem[BG99]{bobkov1999exponential}
S.~G. Bobkov and F.~G\"{o}tze, \emph{Exponential integrability and
  transportation cost related to logarithmic {S}obolev inequalities}, J. Funct.
  Anal. \textbf{163} (1999), no.~1, 1--28. \MR{1682772}

\bibitem[BLM13]{boucheron2013concentration}
St\'ephane Boucheron, G\'abor Lugosi, and Pascal Massart, \emph{Concentration
  inequalities: {A} nonasymptotic theory of independence}, Oxford University
  Press, Oxford, 2013. \MR{3185193}

\bibitem[BM11]{bayati2011dynamics}
Mohsen Bayati and Andrea Montanari, \emph{The dynamics of message passing on
  dense graphs, with applications to compressed sensing}, IEEE Trans. Inform.
  Theory \textbf{57} (2011), no.~2, 764--785. \MR{2810285}

\bibitem[BM12]{bayati2012lasso}
\bysame, \emph{The {LASSO} risk for {G}aussian matrices}, IEEE Trans. Inform.
  Theory \textbf{58} (2012), no.~4, 1997--2017. \MR{2951312}

\bibitem[BMN20]{berthier2020state}
Rapha\"{e}l Berthier, Andrea Montanari, and Phan-Minh Nguyen, \emph{State
  evolution for approximate message passing with non-separable functions}, Inf.
  Inference \textbf{9} (2020), no.~1, 33--79. \MR{4079177}

\bibitem[BZ21]{bellec2021debias}
Pierre~C. Bellec and Cun-Hui Zhang, \emph{De-biasing convex regularized
  estimators and interval estimation in linear models}, arXiv preprint
  arXiv:1912.11943v4 (2021).

\bibitem[CGS15]{chatterjee2015risk}
Sabyasachi Chatterjee, Adityanand Guntuboyina, and Bodhisattva Sen, \emph{On
  risk bounds in isotonic and other shape restricted regression problems}, Ann.
  Statist. \textbf{43} (2015), no.~4, 1774--1800. \MR{3357878}

\bibitem[CGS18]{chatterjee2018matrix}
\bysame, \emph{On matrix estimation under monotonicity constraints}, Bernoulli
  \textbf{24} (2018), no.~2, 1072--1100. \MR{3706788}

\bibitem[Cha14]{chatterjee2014new}
Sourav Chatterjee, \emph{A new perspective on least squares under convex
  constraint}, Ann. Statist. \textbf{42} (2014), no.~6, 2340--2381.
  \MR{3269982}

\bibitem[CP10]{chen2010nonnegativity}
Donghui Chen and Robert~J. Plemmons, \emph{Nonnegativity constraints in
  numerical analysis}, The birth of numerical analysis, World Sci. Publ.,
  Hackensack, NJ, 2010, pp.~109--139. \MR{2604144}

\bibitem[CRPW12]{chandrasekaran2012convex}
Venkat Chandrasekaran, Benjamin Recht, Pablo~A. Parrilo, and Alan~S. Willsky,
  \emph{The convex geometry of linear inverse problems}, Found. Comput. Math.
  \textbf{12} (2012), no.~6, 805--849. \MR{2989474}

\bibitem[DM16]{donoho2016high}
David Donoho and Andrea Montanari, \emph{High dimensional robust
  {M}-estimation: asymptotic variance via approximate message passing}, Probab.
  Theory Related Fields \textbf{166} (2016), no.~3-4, 935--969. \MR{3568043}

\bibitem[EK18]{elkaroui2018impact}
Noureddine El~Karoui, \emph{On the impact of predictor geometry on the
  performance on high-dimensional ridge-regularized generalized robust
  regression estimators}, Probab. Theory Related Fields \textbf{170} (2018),
  no.~1-2, 95--175. \MR{3748322}

\bibitem[FGS21]{fang2021multivariate}
Billy Fang, Adityanand Guntuboyina, and Bodhisattva Sen, \emph{Multivariate
  extensions of isotonic regression and total variation denoising via entire
  monotonicity and {H}ardy-{K}rause variation}, Ann. Statist. \textbf{49}
  (2021), no.~2, 769--792. \MR{4255107}

\bibitem[GJ14]{groeneboom2014nonparametric}
Piet Groeneboom and Geurt Jongbloed, \emph{Nonparametric estimation under shape
  constraints}, Cambridge Series in Statistical and Probabilistic Mathematics,
  vol.~38, Cambridge University Press, New York, 2014. \MR{3445293}

\bibitem[GN16]{gine2015mathematical}
Evarist Gin\'{e} and Richard Nickl, \emph{Mathematical foundations of
  infinite-dimensional statistical models}, Cambridge Series in Statistical and
  Probabilistic Mathematics, [40], Cambridge University Press, New York, 2016.
  \MR{3588285}

\bibitem[GNP17]{goldstein2017gaussian}
Larry Goldstein, Ivan Nourdin, and Giovanni Peccati, \emph{Gaussian phase
  transitions and conic intrinsic volumes: {S}teining the {S}teiner formula},
  Ann. Appl. Probab. \textbf{27} (2017), no.~1, 1--47. \MR{3619780}

\bibitem[Gor85]{gordan1985some}
Yehoram Gordon, \emph{Some inequalities for {G}aussian processes and
  applications}, Israel J. Math. \textbf{50} (1985), no.~4, 265--289.
  \MR{800188}

\bibitem[Gor88]{gordon1988milman}
Y.~Gordon, \emph{On {M}ilman's inequality and random subspaces which escape
  through a mesh in {${\bf R}^n$}}, Geometric aspects of functional analysis
  (1986/87), Lecture Notes in Math., vol. 1317, Springer, Berlin, 1988,
  pp.~84--106. \MR{950977}

\bibitem[GS18]{guntuboyina2017nonparametric}
Adityanand Guntuboyina and Bodhisattva Sen, \emph{Nonparametric
  shape-restricted regression}, Statist. Sci. \textbf{33} (2018), no.~4,
  568--594. \MR{3881209}

\bibitem[HK22]{han2022berry}
Qiyang Han and Kengo Kato, \emph{Berry-{E}sseen bounds for {C}hernoff-type
  non-standard asymptotics in isotonic regression}, Ann. Appl. Probab., to
  appear. Available at arXiv:1910.09662 (2022).

\bibitem[HSS22]{han2022high}
Qiyang Han, Bodhisattva Sen, and Yandi Shen, \emph{High dimensional asymptotics
  of likelihood ratio tests in the {G}aussian sequence model under convex
  constraints}, Ann. Statist., to appear. Available at arXiv:2010.03145 (2022).

\bibitem[HW16]{han2016multivariate}
Qiyang Han and Jon~A. Wellner, \emph{Multivariate convex regression: global
  risk bounds and adaptation}, arXiv preprint arXiv:1601.06844 (2016).

\bibitem[HWCS19]{han2019isotonic}
Qiyang Han, Tengyao Wang, Sabyasachi Chatterjee, and Richard~J. Samworth,
  \emph{Isotonic regression in general dimensions}, Ann. Statist. \textbf{47}
  (2019), no.~5, 2440--2471. \MR{3988762}

\bibitem[HZ20]{han2019limit}
Qiyang Han and Cun-Hui Zhang, \emph{Limit distribution theory for block
  estimators in multiple isotonic regression}, Ann. Statist. \textbf{48}
  (2020), no.~6, 3251--3282. \MR{4185808}

\bibitem[JM13]{javanmard2013state}
Adel Javanmard and Andrea Montanari, \emph{State evolution for general
  approximate message passing algorithms, with applications to spatial
  coupling}, Inf. Inference \textbf{2} (2013), no.~2, 115--144. \MR{3311445}

\bibitem[Kat09]{kato2009degrees}
Kengo Kato, \emph{On the degrees of freedom in shrinkage estimation}, J.
  Multivariate Anal. \textbf{100} (2009), no.~7, 1338--1352. \MR{2514133}

\bibitem[KGGS20]{kur2020convex}
Gil Kur, Fuchang Gao, Adityanand Guntuboyina, and Bodhisattva Sen, \emph{Convex
  regression in multidimensions: Suboptimality of least squares estimators},
  arXiv preprint arXiv:2006.02044 (2020).

\bibitem[Kol11]{koltchinskii2008oracle}
Vladimir Koltchinskii, \emph{Oracle inequalities in empirical risk minimization
  and sparse recovery problems}, Lecture Notes in Mathematics, vol. 2033,
  Springer, Heidelberg, 2011, Lectures from the 38th Probability Summer School
  held in Saint-Flour, 2008, \'Ecole d'\'Et\'e de Probabilit\'es de
  Saint-Flour. [Saint-Flour Probability Summer School]. \MR{2829871}

\bibitem[KP08]{kim2008nonegative}
Hyunsoo Kim and Haesun Park, \emph{Nonnegative matrix factorization based on
  alternating nonnegativity constrained least squares and active set method},
  SIAM J. Matrix Anal. Appl. \textbf{30} (2008), no.~2, 713--730. \MR{2421467}

\bibitem[Kud63]{kudo1963multivariate}
Akio Kud\^{o}, \emph{A multivariate analogue of the one-sided test}, Biometrika
  \textbf{50} (1963), 403--418. \MR{163386}

\bibitem[LH95]{lawson1995solving}
Charles~L. Lawson and Richard~J. Hanson, \emph{Solving least squares problems},
  Classics in Applied Mathematics, vol.~15, Society for Industrial and Applied
  Mathematics (SIAM), Philadelphia, PA, 1995, Revised reprint of the 1974
  original. \MR{1349828}

\bibitem[Mas07]{massart2007concentration}
Pascal Massart, \emph{Concentration inequalities and model selection}, Lecture
  Notes in Mathematics, vol. 1896, Springer, Berlin, 2007, Lectures from the
  33rd Summer School on Probability Theory held in Saint-Flour, July 6--23,
  2003, With a foreword by Jean Picard. \MR{2319879 (2010a:62008)}

\bibitem[MM21]{miolane2021distribution}
L\'{e}o Miolane and Andrea Montanari, \emph{The distribution of the {L}asso:
  uniform control over sparse balls and adaptive parameter tuning}, Ann.
  Statist. \textbf{49} (2021), no.~4, 2313--2335. \MR{4319252}

\bibitem[MW00]{meyer2000degrees}
Mary Meyer and Michael Woodroofe, \emph{On the degrees of freedom in
  shape-restricted regression}, Ann. Statist. \textbf{28} (2000), no.~4,
  1083--1104. \MR{1810920}

\bibitem[OH16]{oymak2016sharp}
Samet Oymak and Babak Hassibi, \emph{Sharp {MSE} bounds for proximal
  denoising}, Found. Comput. Math. \textbf{16} (2016), no.~4, 965--1029.
  \MR{3529131}

\bibitem[OTH13]{oymak2013squared}
Samet Oymak, Christos Thrampoulidis, and Babak Hassibi, \emph{The squared-error
  of generalized lasso: A precise analysis}, 2013 51st Annual Allerton
  Conference on Communication, Control, and Computing (Allerton), IEEE, 2013,
  pp.~1002--1009.

\bibitem[RLN86]{raubertas1986hypothesis}
Richard~F. Raubertas, Chu-In~Charles Lee, and Erik~V. Nordheim,
  \emph{Hypothesis tests for normal means constrained by linear inequalities},
  Comm. Statist. A---Theory Methods \textbf{15} (1986), no.~9, 2809--2833.
  \MR{855765}

\bibitem[Roc97]{rockafellar1997convex}
R.~Tyrrell Rockafellar, \emph{Convex {A}nalysis}, Princeton Landmarks in
  Mathematics, Princeton University Press, Princeton, NJ, 1997, Reprint of the
  1970 original, Princeton Paperbacks. \MR{1451876 (97m:49001)}

\bibitem[SC19]{sur2019modern}
Pragya Sur and Emmanuel~J. Cand\`es, \emph{A modern maximum-likelihood theory
  for high-dimensional logistic regression}, Proc. Natl. Acad. Sci. USA
  \textbf{116} (2019), no.~29, 14516--14525. \MR{3984492}

\bibitem[Sch14]{schneider2014convex}
Rolf Schneider, \emph{Convex bodies: the {B}runn-{M}inkowski theory}, expanded
  ed., Encyclopedia of Mathematics and its Applications, vol. 151, Cambridge
  University Press, Cambridge, 2014. \MR{3155183}

\bibitem[Sio58]{sion1958general}
Maurice Sion, \emph{On general minimax theorems}, Pacific J. Math. \textbf{8}
  (1958), 171--176. \MR{97026}

\bibitem[Sto13]{stojnic2013framework}
Mihailo Stojnic, \emph{A framework to characterize performance of lasso
  algorithms}, arXiv preprint arXiv:1303.7291 (2013).

\bibitem[TAH18]{thrampoulidis2018precise}
Christos Thrampoulidis, Ehsan Abbasi, and Babak Hassibi, \emph{Precise error
  analysis of regularized {$M$}-estimators in high dimensions}, IEEE Trans.
  Inform. Theory \textbf{64} (2018), no.~8, 5592--5628. \MR{3832326}

\bibitem[Tib96]{tibshirani1996regression}
Robert Tibshirani, \emph{Regression shrinkage and selection via the lasso}, J.
  Roy. Statist. Soc. Ser. B \textbf{58} (1996), no.~1, 267--288. \MR{1379242}

\bibitem[TOH14]{thrampoulidis2014simple}
Christos Thrampoulidis, Samet Oymak, and Babak Hassibi, \emph{Simple error
  bounds for regularized noisy linear inverse problems}, 2014 IEEE
  International Symposium on Information Theory, IEEE, 2014, pp.~3007--3011.

\bibitem[TOH15a]{thrampoulidis2015recovering}
\bysame, \emph{Recovering structured signals in noise: Least-squares meets
  compressed sensing}, Compressed Sensing and Its Applications, Springer, 2015,
  pp.~97--141.

\bibitem[TOH15b]{thrampoulidis2015regularized}
\bysame, \emph{Regularized linear regression: A precise analysis of the
  estimation error}, Conference on Learning Theory, PMLR, 2015, pp.~1683--1709.

\bibitem[Tro15]{tropp2015tropp}
Joel~A. Tropp, \emph{Convex recovery of a structured signal from independent
  random linear measurements}, Sampling theory, a renaissance, Appl. Numer.
  Harmon. Anal., Birkh\"{a}user/Springer, Cham, 2015, pp.~67--101. \MR{3467419}

\bibitem[vdG00]{van2000empirical}
Sara van~de Geer, \emph{Applications of {E}mpirical {P}rocess {T}heory},
  Cambridge Series in Statistical and Probabilistic Mathematics, vol.~6,
  Cambridge University Press, Cambridge, 2000. \MR{1739079 (2001h:62002)}

\bibitem[vdGW17]{van2015concentration}
Sara van~de Geer and Martin~J. Wainwright, \emph{On concentration for
  (regularized) empirical risk minimization}, Sankhya A \textbf{79} (2017),
  no.~2, 159--200. \MR{3707417}

\bibitem[vdVW96]{van1996weak}
Aad van~der Vaart and Jon~A. Wellner, \emph{Weak {C}onvergence and {E}mpirical
  {P}rocesses}, Springer Series in Statistics, Springer-Verlag, New York, 1996.
  \MR{1385671 (97g:60035)}

\bibitem[Zha02]{zhang2002risk}
Cun-Hui Zhang, \emph{Risk bounds in isotonic regression}, Ann. Statist.
  \textbf{30} (2002), no.~2, 528--555. \MR{1902898 (2003e:62084)}

\end{thebibliography}

\end{document}